\documentclass[a4paper,twoside,10pt]{article}

\usepackage{blindtext}
\usepackage[a4paper,top=3cm,bottom=3cm,left=3cm,right=3cm,bindingoffset=5mm]{geometry}
\usepackage{amsmath,amssymb}
\usepackage{amsthm} 
\usepackage{graphicx} 
\usepackage{subcaption}
\usepackage{stmaryrd} 
\usepackage{cases} 
\usepackage{bm}
\usepackage{color}
\usepackage{comment}
\usepackage{dirtytalk}
\usepackage{multirow} 
\usepackage{booktabs} 
\usepackage{hyperref}
\usepackage{authblk}
\numberwithin{equation}{section}
\newtheorem{theorem}{Theorem}[section]
\newtheorem{proposition}{Proposition}[section]
\newtheorem{corollary}{Corollary}[theorem]
\newtheorem{lemma}[theorem]{Lemma}
\theoremstyle{remark}
\newtheorem{remark}[theorem]{Remark}
\newcommand{\numberset}{\mathbb}
\newcommand{\N}{\numberset{N}}
\newcommand{\Z}{\numberset{Z}}

\newcommand{\R}{\numberset{R}}
\newcommand{\C}{\numberset{C}}
\newcommand{\Pol}{\numberset{P}}

\newcommand{\abs}[1]{\left\lvert#1\right\rvert} 
\newcommand{\norm}[1]{\left\lVert#1\right\rVert} 
\newcommand{\dnorm}[1]{{\left\vert\kern-0.25ex\left\vert\kern-0.25ex\left\vert #1 \right\vert\kern-0.25ex\right\vert\kern-0.25ex\right\vert}} 

\usepackage{stackengine}
\stackMath
\newcommand\tenq[2][1]{%
 \def\useanchorwidth{T}%
  \ifnum#1>1%
    \stackunder[0pt]{\tenq[\numexpr#1-1\relax]{#2}}{\scriptscriptstyle\sim}%
  \else%
    \stackunder[1pt]{#2}{\scriptscriptstyle\sim}%
  \fi%
}

\def\aishK{\alpha_{\star,h_K}}    
\def\asihK{\alpha^{\star}_{h_K}}  

\def\asih{\alpha^{\star}_{h}}     
\def\aish{\alpha_{\star,h}}       

\def\asuph{\alpha^\star_h}        
\def\asdownh{\alpha_{\star,h}}    

\title{Volumetric locking-free Mixed Virtual Element Methods for Contact Problems}
\date{}
\author[1,2]{{{C. Lovadina}}\thanks{carlo.lovadina@unimi.it}}
\author[1]{{{L. Molinari}},\thanks{loris.molinari@unimi.it}}

\affil[1]{Dipartimento di Matematica ``F. Enriques'', Universit\`a degli Studi di Milano, Via Cesare Saldini 50 - 20133 Milano, Italy}
\affil[2]{IMATI-CNR, Via Adolfo Ferrata 5 - 27100 Pavia, Italy}
\begin{document}
\maketitle
\begin{abstract}
We consider the approximation of the 2D frictionless contact problem in elasticity using the Virtual Element Methods (VEMs). To overcome the volumetric locking phenomenon in the nearly incompressible case, we adopt a mixed displacement/pressure ($u/p$) variational formulation, where pressure is introduced as an independent unknown. We present the VEM discretization and develop a general error analysis, keeping explicit track of the constants involved in the error estimates, thus allowing to consider meshes with \say{small edges}.
As examples, we consider two possible VEM schemes: a first-order scheme and a second-order scheme. The numerical results confirm the theoretical predictions, specifically both schemes show: 1) robustness with respect to the volumetric parameter $\lambda$, thus preventing the occurrence of the volumetric locking phenomenon; 2) good behavior even in the presence of \say{small edges}; 3) achievement of the expected theoretical convergence rates.
\end{abstract}
\section{Introduction}
Virtual Element Methods (VEMs) can be viewed as an evolution of both the Finite Element Methods (FEMs) and the Mimetic Finite Differences (MFDs). They share the same variational structure as FEMs while offering exceptional flexibility in handling polygonal and polyhedral meshes.

Since their introduction in 2013 in the seminal paper \cite{volley}, VEMs have garnered significant and growing interest within both the mathematical and engineering communities. Very substantial advances have been made in both the theoretical aspects and the practical applications of this approach. Currently, the VEM literature is so vast that it is beyond the scope of this introduction to review it effectively. For a recent (though non-exhaustive) overview of the theory and applications of VEMs, we refer the interested reader to the monograph \cite{sema-simai}.

In this paper we consider the contact problem in elasticity and its approximation (see \cite{sofonea2005analysis} and \cite{ACTA-contact}, for instance), using the Virtual Element Methods. 
From a mathematical perspective, this problem leads to a system of variational inequalities. Fundamental contributions to the Galerkin approximation of such inequalities in abstract settings can be found, for instance, in \cite{Falk1974}, \cite{BrezziHagerRaviart1977} and \cite{BrezziHagerRaviart1978}.  

VEMs for contact Mechanics have been employed for the first time in \cite{Wriggers_Rust_Reddy}, where the authors exploited the VEM ease of treating meshes (in particular, non-matching grids) to design a simple and efficient displacement-based numerical scheme, where a node-to-node approach is used to treat the contact region. After that pioneering work, several contributions have been developed extending the ideas in \cite{Wriggers_Rust_Reddy}. For example, in \cite{Wriggers2019} frictional contact and large deformations have been considered, in \cite{Aldakheel2020} curved virtual elements are proposed for 2D contact problems. See also the recent book \cite{Wriggers2023} for an account of the VEM literature about this topic. From a theoretical viewpoint, the paper \cite{Wang2022} covers the convergence analysis of the basic displacement-based VEM scheme applied to the infinitesimal elasticity frictionless contact problem (see also the paper \cite{XIAO2023} and the references therein, for more general and complex frameworks in Contact Mechanics). However, these latter contributions, although very interesting, do not cover the case of nearly incompressible materials.    
Since we are interested in developing accurate numerical schemes also for that regime, we employ a suitable mixed formulation, in which a pressure-like quantity explicitly enters as an independent unknown (see \cite{Boffi_Brezzi_Fortin}, for instance). 
As a first research step, we consider the simplest case, i.e. we assume that: 1) the problem is 2D; 2) in each body the constitutive law is homogeneous and isotropic; 3) the elastic bodies undergo small displacements and small deformations; 4) no friction is involved in the contact. 
In such a situation, we develop a general analysis, following the lines of \cite{Belgacem_Renard_Slimane}, and we present a couple of possible Virtual Element schemes which fit the abstract framework: one is based on first order consistency and approximation properties, while the other is designed in order to satisfy second order consistency and approximation properties. The analysis and the numerical results highlight the following interesting features:

\begin{itemize}

\item The methods are robust with respect to the choice of the Lam\'e's first parameter, i.e. no volumetric locking occurs.  

\item The methods show a good behaviour even in presence of ``small edges'' (for example, when they arise in non matching grids as a consequence of the contact). From a theoretical viewpoint, ``small edges'' introduce logarithmic degeneracy terms in the stability constants and in the approximation estimates. 
However, due to the slow growth of those terms, the degeneracy is not observed from a practical viewpoint    

\item The theoretical convergence rates are the expected ones for this kind of problems, under usual assumptions. In particular, we remark that for the second order scheme, which potentially could show a quadratic convergence rate (in fact, this is the behaviour when applied to linear elasticity problems), there is a barrier of $5/2$ in the converge order with respect to the meshsize, even for regular analytical solutions. This loss of $1/2$ points of convergence is typical in contact problems (see \cite{Belhachmi_BenBelgacem}, for example) and it is due to the discretization of the convex set of admissible displacements, which introduces a further consistency term.   

\end{itemize}

Of course, VEMs are not the only approach, successfully used to numerically approximate contact problems in elasticity. Among the others, the mortar method (see \cite{BELGACEM1998},  \cite{temizer2012mixed}, \cite{wohlmuth2000}, \cite{puso2004mortar}, for instance) is one of the most popular in the framework of the Finite Element Methods, to handle the case of non macthing grids. Another recent approach takes advantage of the Nitsche's method, see e.g. \cite{chouly2013} and \cite{Stenberg2020}. Furthermore, also Isogeometric Analysis has been extensively used, see e.g. \cite{DeLorenzis2014}, \cite{DeLorenzis2015} and \cite{Temizer2011}. Regarding polygonal/polyhedral schemes, we mention the hybrid high order methods (HHO), see \cite{HHO-book-2}, \cite{di_pietro2025hho}, \cite{chouly2020hho_nitsche}, for example.   

The paper is organized as follows.
Section \ref{sec:mixed-contact} introduces the static elastic contact problem in the infinitesimal displacement and deformation regime. In particular, a suitable displacement/pressure formulation is considered, leading to a system of variational inequalities. Compared to the approach detailed in \cite{Belgacem_Renard_Slimane}, we do not need to rely on any augmented formulation to deal with the case where only Dirichlet and contact boundary conditions occur (i.e. no Neumann boundary condition enters into play).

Section \ref{sec:VEM_contact} is about the Virtual Element discretization of the variational system presented in Section \ref{sec:mixed-contact}. We detail the technique used for managing possible non-matching grids (i.e. the node insertion strategy introduced in the paper \cite{Wriggers_Rust_Reddy}). Then we define the VEM discrete problem, for which we develop a general convergence result, under suitable standard hypotheses. This analysis may be considered as a VEM adaptation of the study presented in \cite{Belgacem_Renard_Slimane} for the Signorini problem. However, here we keep precise track of the stability constants entering into play. Furthermore, we allow such quantities to be possibly dependent on the meshsize: this aspect is indeed important when the meshes exhibit the presence of ``small edges'', an occurrence which must be considered if the node insertion technique is employed. 

In Section \ref{sec:VEMexamples} we provide a couple of examples which fit the framework of Section \ref{sec:VEM_contact}: a first order and a second order VEM schemes; this latter method takes advantage of the virtual element spaces developed in \cite{BLV:2017}. 

In Section \ref{sec:quadratic-conv} a complete analysis is developed for the second order method, and in particular we prove suitable estimates for the stability constants. The first order scheme can be treated using essentially the same arguments. 
We observe that part of the analysis could be performed using finer tools than the ones here employed for the treatment of the ``small edges'' case. For example, one could take advantage of the techniques presented in \cite{BLR:2017} or \cite{brenner-sung:2018} to avoid logarithmic terms arising in some estimates from the occurrence of ``small edges''. However, a logarithmic term would still remain, stemming from the \emph{inf-sup} constant estimate. Therefore, we decide to follow more classical and easier techniques to establish our error estimates.

Section \ref{sec:numer} presents some numerical results that confirm the theoretical predictions. In particular, robustness with respect to the volumetric parameter and with respect to the presence of ``small edges'' is clearly observed.

Finally, in Section \ref{sec:conclusions} we draw some conclusions.

Throughout the paper we use standard notations for Sobolev spaces, norms and seminorms (for instance, see \cite{LeoniBook}). 

Furthermore, we extensively employ the notation $a\lesssim b$ to denote $a\leq C\, b$, where $C$ is a positive constant independent of both the mesh size $h$ and the volumetric (first Lamé) parameter $\lambda$. Analogously, the notation $a\simeq b$ signifies that $a \leq C\, b \leq C'\, a$, where the positive constant $C'$ likewise remains independent of $h$ and $\lambda$.
\section{A Mixed Formulation of the Contact Problem}\label{sec:mixed-contact}
Let $\Omega^i \subset \R^2$ for $i=1,2$ be an open, bounded and Lipschitz domain, representing the reference configuration of the elastic body. The boundary is denoted by $\Gamma^i:=\partial \Omega ^i$, and is partitioned into three relatively open, mutually disjoint, and measurable parts $\Gamma^i_D$, $\Gamma^i_N$ and $\Gamma^i_C$ with $\text{meas}\:({\Gamma^i_D})>0$. Displacements are prescribed on the Dirichlet boundary $\Gamma^i_D$, surface tractions are applied on the Neumann boundary $\Gamma^i_N$, and potential contact may occur on the contact boundary $\Gamma^i_C$. 
\par
As usual in Contact Mechanics, we adopt the master-slave paradigm (see \cite{ACTA-contact}, \cite{Chouly_Hild_Renard}), choosing $\Gamma^1_C$ as the slave boundary and $\Gamma^2_C$ as the master boundary. In addition, we assume the existence of a sufficiently smooth bijective mapping $\chi : \Gamma^1_C \to \Gamma^2_C$ which defines the contact pairing and satisfies $\chi(\Gamma^1_C) = \Gamma^2_C$. Then, we define the common contact interface $\Gamma_C:=\Gamma^1_C=\chi^{-1}(\Gamma^2_C)$.\\

\par
For each body, we consider homogeneous and isotropic materials and we restrict ourselves to the regime of small deformations. Then, the linearized strain-displacement relation is
\begin{equation*}
    \boldsymbol{\varepsilon} ({\bf u}^i) = \frac{1}{2} (\boldsymbol{\nabla} {\bf u}^i + \boldsymbol{\nabla}^T {\bf u }^i). 
\end{equation*}
The constitutive equation for the symmetric Cauchy stress tensor $\boldsymbol{\sigma}$ is given in terms of the forth-order Hooke tensor $\mathcal{C}$ by
\begin{equation*}
    \boldsymbol{\sigma}({\bf u}^i) = \mathcal{C} {\bf u}^i := 2 \mu^i \boldsymbol{\varepsilon}({\bf u}^i) + \lambda^i \:\text{tr}(\boldsymbol{\varepsilon}({\bf u}^i)) {\bf Id},
\end{equation*}
where $\text{tr}$ denotes the trace operator and $\bf Id$ is the identity tensor. The positive coefficients $\mu^i$ and $\lambda^i$ are the Lamé parameters, which are assumed to be constant in each domain $\Omega^i$, but have possibly different values on each body.
\par
The two bodies satisfy the linearized elastic equilibrium equation, together with Dirichlet and Neumann boundary conditions
  \begin{equation}
    \left\{
    \begin{aligned}
         - \text{div} \boldsymbol{\sigma} ({\bf u}^i) = {\bf f}^i & \quad\text{in } \Omega^i, \\
         {\bf u}^i = {\bf 0} & \quad\text{on } \Gamma^i_D, \\
         \boldsymbol{\sigma}({\bf u}^i) {\bf n}^i = {\bf g}^i & \quad\text{on } \Gamma^i_N, 
    \end{aligned}   
    \right.
  \end{equation}
where ${\bf n}^i$ denotes the outer unit normal vector on $\Gamma^i$, which is almost everywhere well-defined. Here, the volume load ${\bf f}^i$ and the surface traction ${\bf g}^i$ are assumed to be in $[L^2(\Omega^i)]^2$ and $[L^2(\Gamma^i_N)]^2$, respectively. Homogeneous Dirichlet boundary conditions are imposed for the sake of simplicity, but non-homogeneous conditions can be treated in usual ways.
\par
To formulate the contact conditions, we first define the contact stress as
\begin{equation*}
    \boldsymbol{\sigma}:= \boldsymbol{\sigma}({\bf u}^1){\bf n}^1 \quad \text{on }\Gamma_C,
\end{equation*}
along with its normal and tangential components on the slave boundary
\begin{equation}\label{eq:sigma_n_t}
    \sigma_n:= \boldsymbol{\sigma}({\bf u}^1){\bf n}^1 \cdot {\bf n}^1, \quad \boldsymbol{\sigma}_t:= \boldsymbol{\sigma} - \sigma_n {\bf n}^1 \quad \text{on }\Gamma_C.
\end{equation}
In addition, we introduce the jump of the normal displacement along the contact boundary, given by
\begin{equation}\label{eq:u_n}
    \llbracket u_n \rrbracket := ({\bf u}^1 - {\bf u}^2 \circ \chi) \cdot {\bf n}^1 \quad \text{on }\Gamma_C.
\end{equation}
We also define the initial gap function $g_0$, which is assumed to be nonnegative and represents the initial distance between the two bodies in the reference configuration. 
Then, the contact constraints can be expressed as follows
\begin{equation} \label{eq:Signorini_conditions}
    \llbracket u_n \rrbracket \leq g_0, \quad \sigma_n \leq 0,\quad \sigma_n ( \llbracket u_n \rrbracket - g_0) = 0 \quad \text{on } \Gamma_C.
\end{equation}
Conditions (\ref{eq:Signorini_conditions}) are known as the Signorini conditions. The first inequality ensures that there is no interpenetration between the two bodies, while the second condition implies that contact pressure $\sigma_n$ is purely compressive. The third equality, known as the complementary condition, states that the contact pressure can be non-zero only when the two bodies are in contact.
Moreover, the absence of friction along the contact boundary is enforced by
\begin{equation*}
    \boldsymbol{\sigma}_t = {\bf 0} \quad \text{on } \Gamma_C.
\end{equation*}
\par
Finally, Newton's third law requires that the contact forces be in equilibrium over the region where contact occurs; this condition reads as
\begin{equation*}
    \boldsymbol{\sigma}({\bf u}^1) {\bf n}^1 + \boldsymbol{\sigma}({\bf u}^2 \circ \chi)({\bf n}^2 \circ \chi) J_{\chi}={\bf 0} \quad \text{on } \Gamma_C,
\end{equation*}
where $J_\chi$ represents the Jacobian of the mapping $\chi$. 
\newline
\par
Even in the linear elasticity problem, when dealing with nearly incompressible materials (i.e. $\lambda^i \to \infty$), it is well known that the discretization of the displacement-based formulation suffers from numerical volumetric locking, which deteriorates the accuracy of the computed displacement field (see \cite{Babuska_Suri}). To avoid this phenomenon, one of the classical strategies is to adopt the mixed variational formulation of the problem (see \cite{Boffi_Brezzi_Fortin}, \cite{Belgacem_Renard_Slimane}). This approach requires the introduction of another variable (for $i=1,2$), here called \emph{pressure} and given by
\begin{equation*}
    p^i:= \lambda^i \text{div }{\bf u}^i.
\end{equation*}
\par
To derive the weak formulation of the problem, we first present the functional framework suitable for addressing the nearly incompressible contact problem. We introduce the space of admissible displacements, defined as
\begin{equation} \label{eq:displacement_space}
    {\bf V} :={\bf V}^1 \times {\bf V}^2 :=\Big\{ {\bf v} = ({\bf v}^1, {\bf v}^2) \in \big[H^1(\Omega^1)\big]^2 \times \big[ H^1(\Omega^2) \big]^2 : {\bf v}^i = {\bf 0} \text{ on } \Gamma^i_D \text{ for } i=1,2 \Big\},
\end{equation}
endowed with the usual norm and semi-norm
\begin{equation*}
    \norm{{\bf v}}_{{\bf V}}^2 := \norm{{\bf v}^1}_{1,\Omega^1}^2 +  \norm{{\bf v}^2}_{1,\Omega^2}^2 \quad \text{and} \quad \abs{{\bf v}}_{1}^2 := \abs{{\bf v}^1}_{1,\Omega^1}^2 +  \abs{{\bf v}^2}_{1,\Omega^2}^2 \quad \forall {\bf v} \in {\bf V}.
\end{equation*}
Moreover, we define the space of pressures as 
\begin{equation} \label{eq:pressure_space}
    {\bf Q} := Q^1 \times Q^2 := L^2(\Omega^1) \times L^2(\Omega^2),
\end{equation}
endowed with the $L^2$-norm
\begin{equation*}
    \norm{{\bf q}}_{{\bf Q}}^2 := \norm{q^1}_{0,\Omega^1}^2 + \norm{q^2}_{0,\Omega^2}^2 \quad \forall {\bf q} \in {\bf Q}.
\end{equation*}
To impose the kinematic contact conditions, we consider the non-empty, closed and convex subset ${\bf K} \subset {\bf V}$ given by (cf. \eqref{eq:Signorini_conditions}):
\begin{equation} \label{eq:convex_subset}
    {\bf K} := \{ {\bf v} \in {\bf V} : \llbracket v_{n} \rrbracket - g_0 \leq 0 \text{ a.e. on } \Gamma_C \}.
\end{equation}
The mixed variational formulation of the frictionless contact problem is obtained by standard arguments (see \cite{Chouly_Hild_Renard}) and it reads as
\begin{equation}
    \begin{cases} \label{eq:contact_problem}
        \text{find } ({\bf u},{\bf p}) \in {\bf K} \times {\bf Q} \text{ such that } & \\
        a({\bf u}, {\bf v} - {\bf u}) + b({\bf v} -{\bf u}, {\bf p}) \geq F({\bf v} - {\bf u}) & \forall {\bf v} \in {\bf K}, \\
        b({\bf u},{\bf q}) - c_{\lambda}({\bf p},{\bf q}) =0 & \forall {\bf q} \in {\bf Q},
    \end{cases}
\end{equation}
where we set for ${\bf u}, {\bf v} \in {\bf V}$ and ${\bf p}, {\bf q} \in {\bf Q}$
\begin{equation} \label{eq:bilinear_a}
    a({\bf u}, {\bf v}) := \sum_{i=1}^2 a^i({\bf u}^i,{\bf v}^i), \quad a^{i}({\bf u}^i,{\bf v}^i) := 2\mu^i \int_{\Omega^i} \boldsymbol{\varepsilon} ({\bf u}^i) : \boldsymbol{\varepsilon}({\bf v}^i)\; dx,
\end{equation}
\begin{equation} \label{eq:bilinear_b}
    b({\bf v},{\bf q}) := \sum_{i=1}^2 b^i({\bf v}^i, q^i), \quad  b^i({\bf v}^i, q^i):= \int_{\Omega^i} \text{div } {\bf v}^i q^i \; dx,
\end{equation}
\begin{equation} \label{eq:bilinear_c}
    c_{\lambda}({\bf p}, {\bf q}) = \sum_{i=1}^2 \frac{1}{\lambda^i}c^i(p^i, q^i), \quad c^i(p^i,q^i) := \int_{\Omega^i} p^i q^i \; dx,
\end{equation}
\begin{equation} \label{eq:linear_F}
    F({\bf v}) := \sum_{i=1}^2 F^i({\bf v}^i), \quad F^i({\bf v}^i):= \int_{\Omega^i} {\bf f}^i \cdot {\bf v}^i \; dx + \int_{\Gamma_N^i} {\bf g}^i \cdot {\bf v}^i \; d\Gamma.
\end{equation}
We notice that the symmetric bilinear form $a(\cdot, \cdot)$ is continuous and coercive on ${\bf V}$. Moreover, the symmetric bilinear form $c_\lambda(\cdot,\cdot)$ is continuous on ${\bf Q}$ and coercive with a coercivity constant $\gamma(\lambda) = 1/\max \{\lambda^1,\lambda^2 \}$.

Regarding the bilinear form $b(\cdot, \cdot)$, we notice that it is continuous on ${\bf V}\times{\bf Q}$ and it satisfies an appropriate \emph{inf-sup} condition. More precisely, let us introduce
\begin{equation}\label{eq:W-def}
    {\bf W}^i:=\{ {\bf v}^i \in {\bf V}^i: {\bf v}^i_{|\Gamma_C} = {\bf 0}\},
\end{equation}
and we take ${\bf W}:={\bf W}^1 \times {\bf W}^2$. Obviously, we have ${\bf W} \subset {\bf K}$. Let us also denote with $B_i: {\bf V}^i\to (Q^{i})'\equiv Q^i$ the continuous linear operator naturally induced by $b^i(\cdot,\cdot)$, and denote with $B_i^t$ its formal adjoint operator.
We now need to distinguish two cases. 
\begin{enumerate}
\item If $\abs{\Gamma^i_N}=0$, then only Dirichlet and contact boundary conditions are present. In this case, ${\bf W}^i=[H^1_0(\Omega^i)]^2$ and the kernel of $B_i^t$ is
    \begin{equation*}
        H^i:=\ker B_i^t = \R,
    \end{equation*}
which gives    
    \begin{equation*}
        Q^i/H^i = L^2(\Omega^i)/\R \simeq L^2_0(\Omega^i).
    \end{equation*}
\item If $\abs{\Gamma^i_N}>0$, Neumann boundary conditions comes into play, and the kernel of $B_i^t$ becomes trivial
    \begin{equation*}
        H^i:=\ker B_i^t=\{ 0\}.
    \end{equation*}
    Thus, we have
    \begin{equation*}
        Q^i/H^i = Q^i=L^2(\Omega^i).
    \end{equation*}
\end{enumerate} 
Setting ${\bf H}:=H^1\times H^2$, we infer that the inf-sup condition
\begin{equation} \label{eq:continuous_inf_sup}
    \exists \beta>0 : \sup_{{\bf w} \in {\bf W}} \frac{b({\bf w},{\bf q})}{\norm{{\bf w}}_{{\bf V}}} \geq \beta \norm{{\bf q}}_{{\bf Q}/{\bf H}} \quad \text{for all } {\bf q} \in {\bf Q},
\end{equation}
directly follows from standard results (see \cite{Boffi_Brezzi_Fortin}, \cite{Girault_Raviart}), no matter which of the four possible combinations for $H^i$ ($i=1,2$) occurs. Here above, we set ${\bf Q}/{\bf H}:=(Q^1\times Q^2)/(H^1 \times H^2)$ and $\norm{\cdot}_{{\bf Q}/{\bf H}}$ denotes the usual quotient norm.

Therefore, the following existence, uniqueness and uniform stability result for the contact problem (\ref{eq:contact_problem}) holds true.
\begin{proposition}
    The contact problem (\ref{eq:contact_problem}) has a unique solution $({\bf u},{\bf p}) \in {\bf K}\times {\bf Q}$ such that
    \begin{equation*}
        \norm{{\bf u}}_{{\bf V}} + \norm{{\bf p}}_{{\bf Q}/{\bf H}} \leq C \norm{F}_{{\bf V}'},
    \end{equation*}
    where $C$ is a positive constant independent from $\lambda^i$.
    
\end{proposition}
\section{Virtual Element Discretization of the Contact Problem}\label{sec:VEM_contact}
We describe the Virtual Element discretization of the contact problem (\ref{eq:contact_problem}). Given any subset $\omega \subset \R^2$ and $k \in \N$, we will denote by $\Pol_k(\omega)$ the polynomials up to degree $k$ defined on $\omega$. 
\subsection{Virtual element meshes}\label{ss:VEMspaces}
Let $\{ \mathcal{T}^i_h\}_h$ be a sequence of decompositions of $\Omega^i$, $i=1,2$. We assume that each decomposition is made of a finite number of non-overlapping polygons $K$ with
\begin{equation*}
    h_K:=\text{diameter}(K), \quad h^i:=\max_{K \in \mathcal{T}_h^i} h_K, \quad h:=\max_{i=1,2} h^i.
\end{equation*}
Given an element $K \in \mathcal{T}_h^i$, we denote by $\abs{K}$ its area and we assume that its boundary $\partial K$ is divided into $N=N(K)$ straight segments, which are called \emph{edges}. Furthermore, the length of an edge $e\in \partial K$ is denoted by $h_e$. The endpoints of each edge are called \emph{vertices} of the element. We emphasize that some edges may be collinear; consequently, the number of edges may be greater than the number of straight segments that make up the polygon $K$. 
\par As in \cite{BLR:2017}, we deal with the following three assumptions on the decompositions of the domains:
\begin{itemize}
    \item [(A1)] There exists $\gamma \in \R^+$, independent of $h$, such that all elements $K\in \mathcal{T}^i_h$ are star-shaped with respect to a ball $B_K$ of radius $\rho_K \geq \gamma h_K$ and center ${\bf x}_K$;
    \item [(A2)] There exists $C \in \N$, independent of $h$, such that $N(K) \leq C$  for all $K\in \mathcal{T}^i_h$;
    \item [(A3)] There exists $\eta \in \R^+$, independent of $h$, such that for all edges $e \in \partial K$ it holds $h_e \geq \eta h_K$.
\end{itemize}

We observe that the contact boundary $\Gamma_C$ inherits two different decompositions from $\mathcal{T}_h^1$ and $\mathcal{T}_h^2$, where the elements correspond to entire edges of internal polygons. As a result, the mapping $\chi$ does not guarantee any correspondence between the nodes and elements of the two contact boundary meshes.  This issue is well recognized in the finite element literature and frequently leads to a more complex enforcement of the discrete contact conditions, as well as a more involved analysis of the problem.
However, the flexibility of virtual elements in handling hanging nodes enables the use of the \emph{node insertion algorithm} described in \cite{Wriggers_Rust_Reddy}. By inserting additional nodes along the contact boundary decompositions, it is always possible to restore node and element matching. We will refer to this situation as the \emph{contact matching condition}, which facilitates the enforcement of discrete node-to-node contact conditions both theoretically and in practical implementation.
It is important to note, however, that this algorithm may naturally produce decompositions containing “small edges”. More specifically, the insertion of new nodes within pre-existing polygons can cause a violation of assumption (A3). For this reason, we will focus our analysis under the weaker assumptions (A1) and (A2). \\
Unless the \emph{node insertion algorithm} is applied, here and in the rest of the paper we assume the \emph{contact matching condition} holds true, but we allow for \emph{for the occurrence of arbitrary small edeges}; the decomposition $\mathcal{T}_h$ represents the union of the two original decompositions $\mathcal{T}_h^1$ and $\mathcal{T}_h^2$ enriched with the nodes added by the algorithm. Consequently, the contact boundary $\Gamma_C$ has a unique decomposition $\mathcal{T}_h^c$ with nodes ${\bf x}_l^c$ for $l=0,\cdots,l^\star$ and edges $e^c_l$ for $l=0,\cdots,l^\star-1$, where 
\begin{equation*}
    e_l^c := ({\bf x}_l^c, {\bf x}_{l+1}^c) \quad \text{for } l=0,\cdots, l^\star -1.
\end{equation*}
Furthermore, we assume that the endpoints ${\bf c}_1$ and ${\bf c}_2$ of the contact boundary $\Gamma_C$ are vertices of some polygon $K$. We will denote by ${\bf x}_{l+1/2}^c$ for $l=0,\cdots,l^\star-1$ the set of middle points on each contact edge.
In the sequel, we will denote with $\Omega$ the union $\Omega^1\cup\Omega^2$ and we will mainly use the above mentioned mesh $\mathcal{T}_h$, unless we need to make explicit reference to the subdomains $\Omega^i$ and meshes $\mathcal{T}_h^i$.

For simplicity, in what follows we assume that $\Gamma_C$ is a straight line. However, as noted in \cite{ACTA-contact}, the case of piecewise linear (or, more generally, piecewise smooth) $\Gamma_C$ can be handled exactly using the same technique, defining all the quantities on each straight segment, and considering product spaces and broken dualities pairs.  

Moreover, for what follows we need to define
\begin{equation}\label{eq:hm}
    h_m := \min_{K \in \mathcal{T}_h} h_{m(K)},
\end{equation}
where 
\begin{equation}\label{eq:hmK}
h_{m(K)}=\min_{e \in \partial K} \abs{e}
\text{ is the length of the smallest edge of the element $K$}.
\end{equation}

The global bilinear forms $a(\cdot,\cdot)$, $b(\cdot,\cdot)$ and $c_\lambda(\cdot,\cdot)$, involving integrals on $\Omega=\Omega^1\cup\Omega^2$ (cf. \eqref{eq:bilinear_a}-\eqref{eq:bilinear_c}), as well as the norms $\norm{\cdot}_{1,\Omega}$ and $\norm{\cdot}_{0,\Omega}$ can be decomposed into local contributions. Indeed, using obvious notations, we have
\begin{equation*}
    a({\bf u},{\bf v}) = \sum_{K \in \mathcal{T}_h} a^{K}({\bf u},{\bf v}) \quad \forall {\bf u},{\bf v} \in {\bf V},
\end{equation*}
\begin{equation*}
    b({\bf v},{\bf q}) = \sum_{K \in \mathcal{T}_h} b^{K}({\bf v},{\bf q}) \quad \forall {\bf u} \in {\bf V} \text{ and } \forall {\bf q} \in {\bf Q},
\end{equation*}
\begin{equation*}
    c_\lambda({\bf p},{\bf q}) = \sum_{K \in \mathcal{T}_h} c^{K}_\lambda({\bf p},{\bf q}) \quad \forall {\bf p},{\bf q}\in {\bf Q},
\end{equation*}
and
\begin{equation*}
    \norm{{\bf v}}_{{\bf V}}=\norm{{\bf v}}_{1,\Omega} = \Big(\sum_{K \in \mathcal{T}_h} \norm{{\bf v}}_{1,K}^2 \Big)^{1/2} \quad \forall {\bf v} \in {\bf V}, \qquad \norm{{\bf q}}_{0,\Omega} = \Big(\sum_{K \in \mathcal{T}_h} \norm{{\bf q}}_{0,K}^2 \Big)^{1/2} \quad \forall {\bf q} \in {\bf Q}.
\end{equation*}
Above and in the sequel, with a little abuse of notation, when ${\bf p}$ and ${\bf q}$ are considered in the element $K$, we agree that ${\bf p}=p^i_{|K}$ and ${\bf q}=q^i_{|K}$, if $K\in\mathcal{T}_h^i$. The same applies to ${\bf u}$ and ${\bf v}$. 
Moreover, we define the $H^1$-broken semi-norm on $\Omega^i$ based on the polygonal decomposition $\mathcal{T}_h^i$
\begin{equation*}
    \abs{{\bf v}}_{1,\Omega,h}:=\big( \sum_{K \in \mathcal{T}_h} \abs{{\bf v}}^2_{1,K}\big)^{1/2} \quad \forall {\bf v}\in [L^2(\Omega)]^2 \text{ such that } 
    {\bf v}_{|K}\in [H^1(K)]^2.
\end{equation*}
For the discrete spaces, we consider two finite-dimensional subspaces ${\bf V}_h:={\bf V}^1_h \times {\bf V}^2_h \subset {\bf V}$ and ${\bf Q}_h:=Q_h^1 \times Q_h^2 \subset {\bf Q}$. We assume that both ${\bf V}^i_h$ and $Q_h^i$ are obtained by gluing local spaces ${\bf V}^i_{h|K}$ and $Q_{h|K}^i$. In addition, we require that, for any $K \in \mathcal{T}^i_h$ and for some integer $k\geq 1$, we have $[\Pol_k(K)]^2 \subset {\bf V}_{h|K}^i$.

\subsection{Discrete bilinear forms}\label{ssa:discrete_forms}
We now define the discrete versions of the bilinear forms involved in the discretization of the contact problem. For what concerns $b(\cdot,\cdot)$ and $c_\lambda(\cdot,\cdot)$, we simply consider their restriction to the discrete spaces, assuming that these bilinear forms are computable. 
As usual in the VEM framework, we introduce a symmetric and computable discrete bilinear form $a_h: {\bf V}_h \times {\bf V}_h \to \R $ which can be split as
\begin{equation} \label{eq:discrete_decomposition_property_a}
    a_h ({\bf u}_h,{\bf v}_h) = \sum_{K \in \mathcal{T}_h} a_h^{K}({\bf u}_h,{\bf v}_h) \quad \forall {\bf u}_h,{\bf v}_h \in {\bf V}_h,
\end{equation}
where the local discrete bilinear forms
\begin{equation*}
    a_h^{K} : {\bf V}_{h|K} \times {\bf V}_{h|K} \to \R
\end{equation*}
approximate the continuous bilinear forms $a^{K}(\cdot,\cdot)$ and satisfy the following properties:
\begin{itemize}
    \item {\bf k-consistency}: for all ${\bf z}_k \in [\Pol_{k}(K)]^2$ and ${\bf v}_h \in {\bf V}_{h|K}$
    \begin{equation} \label{eq:k_consistency}
        a_h^{K}({\bf z}_k,{\bf v}_h) = a^{K}({\bf z}_k,{\bf v}_h);
    \end{equation}
    \item {\bf stability}: there exist two positive quantities $\aishK$ and $\asihK$, which may depend on $h_K$, such that, for all ${\bf v}_h \in {\bf V}_{h|K}$, it holds
    \begin{equation} \label{eq:stability}
        \aishK \;a^{K}({\bf v}_h,{\bf v}_h) \leq a_h^{K}({\bf v}_h,{\bf v}_h) \leq \asihK \;a^{K}({\bf v}_h,{\bf v}_h).
    \end{equation}
\end{itemize}
To define the bilinear forms $a^{K}_h(\cdot,\cdot)$, we adopt the following strategy, see \cite{volley}. 
For any $K \in \mathcal{T}_h$, we introduce the energy projection operator $\Pi_k^{\varepsilon,K}:{\bf V}_{h|K} \to [\Pol_k(K)]^2$ defined as
\begin{equation} \label{eq:energy_projection_operator}
    \begin{cases}
        a^{K}({\bf z}_k, {\bf v}_h-\Pi_k^{\varepsilon,K}{\bf v}_h) = 0 &\forall {\bf z}_k \in [\Pol_k(K)]^2, \\
        P_j^{K}({\bf v}_h-\Pi_k^{\varepsilon,K}{\bf v}_h) = 0 & \text{for } j=1,2,3,
    \end{cases}
\end{equation}
where $P_j^{K}$ are suitable operators needed to fix the rigid body motions. It is straightforward to check that the energy projection is well-defined and $\Pi_k^{\varepsilon,K}{\bf z}_k= {\bf z}_k$ for all ${\bf z}_k \in [\Pol_k(K)]^2$. 
We also introduce a (symmetric) stabilizing bilinear form $S^{K}: {\bf V}_{h|K} \times {\bf V}_{h|K} \to \R$, and set
\begin{equation}\label{eq:abstr_ah}
    a_h^{K}({\bf u}_h,{\bf v}_h) := a^{K}(\Pi_k^{\varepsilon,K}{\bf u}_h,\Pi_k^{\varepsilon,K}{\bf v}_h) + S^{K}\big( (I-\Pi_k^{\varepsilon,K}){\bf u}_h,(I-\Pi_k^{\varepsilon,K}){\bf v}_h \big),
\end{equation}
for all ${\bf u}_h,{\bf v}_h \in {\bf V}_{h|K}$. \\
\par
The symmetry of $a^{K}_h(\cdot,\cdot)$ and the stability property (\ref{eq:stability}) easily imply the continuity of the local and global discrete bilinear forms
\begin{equation} \label{eq:local_discrete_continuity}
    a_h^{K}({\bf u}_h,{\bf v}_h) \leq 2\mu_K\asihK \abs{{\bf u}_h}_{1,K} \abs{{\bf v}_h}_{1,K} \quad \forall {\bf u}_h,{\bf v}_h \in {\bf V}_{h|K},
\end{equation}
where  $\mu_K = \mu^i$ if $K\in\mathcal{T}_h^i$.
Similarly, from the stability property (\ref{eq:stability}) and the coercivity of $a(\cdot,\cdot)$, it follows the coercivity of the discrete bilinear form:
\begin{equation} \label{eq:global_discrete_coercivity}
    \aish \; \alpha \norm{{\bf v}_h}_{1,\Omega}^2 \leq a_h({\bf v}_h,{\bf v}_h) \quad \forall {\bf v}_h \in {\bf V}_h,
\end{equation}
where 
\begin{equation} \label{eq:constant_discrete_global_coercivity}
    \aish:=\min_{K \in \mathcal{T}_h} \aishK,  \qquad \alpha:=\min\{\alpha^1,\alpha^2\}  .
\end{equation}
\subsection{The Discrete Problem}\label{ss:discrete_problem}
Let ${\bf K}_h \subset {\bf V}_h$ be a non-empty, closed, and convex subset of ${\bf V}_h$ containing the zero element $0_{\bf V}$ and not necessarily contained in ${\bf K}$. For any ${\bf F}_h \in {\bf V}_h'$, we consider the VEM discretization of the contact problem
\begin{equation} \label{eq:Discrete_problem}
    \begin{cases}
        \text{Find } ({\bf u}_h,{\bf p}_h) \in {\bf K}_h \times {\bf Q}_h \text{ such that } & \\
        a_h({\bf u}_h, {\bf v}_h - {\bf u}_h) + b({\bf v}_h -{\bf u}_h, {\bf p}_h) \geq {\bf F}_h({\bf v}_h - {\bf u}_h) & \forall {\bf v}_h \in {\bf K}_h, \\
        b({\bf u}_h,{\bf q}_h) - c_{\lambda}({\bf p}_h,{\bf q}_h) =0 & \forall {\bf q}_h \in {\bf Q}_h.
    \end{cases}
\end{equation}
Obviously, the bilinear forms $b(\cdot,\cdot)$ and $c_\lambda(\cdot,\cdot)$ are continuous also on the discrete spaces, and $c_\lambda(\cdot,\cdot)$ is coercive on ${\bf Q}_h$, too. Moreover, from (\ref{eq:local_discrete_continuity}) it follows that $a_h(\cdot,\cdot)$ is continuous on ${\bf V}_h$, i.e.
\begin{equation}
    a_h({\bf u}_h,{\bf v}_h) \leq 2 \mu_{max}\asuph \abs{{\bf u}_h}_{1,\Omega} \abs{{\bf v}_h}_{1,\Omega} \quad \forall {\bf u}_h,{\bf v}_h \in {\bf V}_h,
\end{equation}
where we set
\begin{equation}\label{eq:constant_continuity}
\mu_{max}:= \max\{\mu^1,\mu^2\}, \qquad \asuph:=\max_{K \in \mathcal{T}_h} \asihK.    
\end{equation}
\par
In order to show the well-posedness of the discrete problem, we assume that there exists ${\bf W}_h \subseteq {\bf K}_h \cap {\bf W}$ for which the discrete inf-sup condition holds (see \cite{Belgacem_Renard_Slimane}), i.e.
\begin{equation} \label{eq:discrete_inf_sup}
    \exists \beta_h >0 : \sup_{{\bf w}_h \in {\bf W}_h} \frac{b({\bf w}_h,{\bf q}_h)}{\norm{{\bf w}_h}_{\bf V}} \geq \beta_h \norm{{\bf q}_h}_{{\bf Q}/{\bf H}} \quad \text{for all } {\bf q}_h \in {\bf Q}_h.
\end{equation}
Under the previous assumptions, recalling also the coercivity property and \eqref{eq:global_discrete_coercivity}, we thus have that Problem \ref{eq:Discrete_problem} has a unique solution $({\bf u}_h,{\bf p}_h) \in {\bf K}_h \times {\bf Q}_h $.
\subsection{A general convergence result}\label{ss:abstract_convergence}
In this section, we provide an abstract estimate of the approximation error, considering the case of Dirichlet-contact boundary conditions; the other cases, in which Neumann boundary conditions appear, can be treated using similar and even simpler techniques.
Before proceeding, we introduce the following orthogonal decomposition:
\begin{equation*}
    q^i = \bar{q}^i+q_0^i \quad \forall q^i \in Q^i,
\end{equation*}
with $q_0^i \in H^i$ and $\bar{q}^i \in (H^i)^{\perp}$ and we note that
\begin{equation*}
    \norm{q^i}_{Q^i/H^i} := \inf_{c \in H^i} \norm{q^i+c}_{0,\Omega^i} = \norm{{\bar q}^i}_{0,\Omega^i} \quad \forall q^i \in Q^i.
\end{equation*}
\begin{proposition} \label{prop:pressure_estimate}
    Let $({\bf u},{\bf p})$ and $({\bf u}_h,{\bf p}_h)$ be the solutions of the continuous and discrete problems, respectively. Then, for any ${\bf p}_I \in {\bf Q}_h$, $F_h \in {\bf V}_h'$ and ${\bf u}_{\pi}=({\bf u}^1_{\pi},{\bf u}^2_{\pi})$ such that ${\bf u}^i_{\pi|K}\in[\Pol_k(K)]^2$, it holds
    \begin{equation*}
        \norm{{\bf p}-{\bf p}_h}_{{\bf Q}/{\bf H}} \leq C_p(h) \big(\norm{{\bf u} - {\bf u}_h}_{{\bf V}} + \abs{{\bf u}-{\bf u}_\pi}_{1,\Omega,h} + \norm{{\bf p}-{\bf p}_I}_{{\bf Q}/{\bf H}} + \norm{F-F_h}_{{\bf V}'} \big),
    \end{equation*}
    where 
    \begin{equation} \label{eq:constant_pressure}
        C_p(h):=\max\{1,{\beta_h}^{-1},2\mu_{max}{\beta_h}^{-1} (\asuph+1)\}.
    \end{equation}
\end{proposition}
\begin{proof}
    Due to the triangular inequality, we have
    \begin{equation}\label{eq:pressure_tri_ineq}
        \norm{{\bf p}-{\bf p}_h}_{{\bf Q}/{\bf H}} \leq \norm{{\bf p}-{\bf p}_I}_{{\bf Q}/{\bf H}} + \norm{{\bf p}_I-{\bf p}_h}_{{\bf Q}/{\bf H}} \quad \forall {\bf p}_I \in {\bf Q}_h.
    \end{equation}
    From the inf-sup condition, we get
    \begin{equation} \label{eq:pressure_aus}
        \beta_h \norm{{\bf p}_I-{\bf p}_h}_{{\bf Q}/{\bf H}} \leq \sup_{{\bf w}_h \in {\bf W}_h} \frac{b({\bf w}_h,{\bf p}_I-{\bf p}_h)}{\norm{{\bf w}_h}_{{1,\Omega}}} = \sup_{{\bf w}_h \in {\bf W}_h} \left(\frac{b({\bf w}_h,{\bf p}_I-{\bf p})}{\norm{{\bf w}_h}_{{1,\Omega}}}+\frac{b({\bf w}_h,{\bf p}-{\bf p}_h)}{\norm{{\bf w}_h}_{{1,\Omega}}}\right).
    \end{equation}
    Now, we bound each term separately. Firstly, we observe that the continuity of $b^i(\cdot,\cdot)$ implies
    \begin{equation*}
        b({\bf w}_h,{\bf p}_I-{\bf p}) = b({\bf w}_h,\bar {\bf p}_I-\bar {\bf p}) \leq  \norm{{\bf w}_h}_{1,\Omega} \norm{{\bf p}_I - {\bf p}}_{{\bf Q}/{\bf H}},
    \end{equation*}
    from which we have
    \begin{equation}\label{eq:pressure_first_part}
        \frac{\abs{b({\bf w}_h,{\bf p}_I-{\bf p})}}{\norm{{\bf w}_h}_{{1,\Omega}}} \leq \norm{{\bf p}_I - {\bf p}}_{{\bf Q}/{\bf H}}.
    \end{equation}
    To continue, since $({\bf u},{\bf p})$ and $({\bf u}_h,{\bf p}_h)$ are the solutions of \eqref{eq:contact_problem} and \eqref{eq:Discrete_problem}, respectively, it follows
    \begin{equation} \label{eq:var_eq_aus}
        a({\bf u},{\bf w})+b({\bf w},{\bf p}) = F({\bf w}) \quad \forall {\bf w} \in {\bf W},
    \end{equation}
    \begin{equation} \label{eq:var_eq_aus_discrete_0}
        a_h({\bf u}_h,{\bf w}_h)+b({\bf w}_h,{\bf p}_h) = F_h({\bf w}_h) \quad \forall {\bf w}_h \in {\bf W}_h.
    \end{equation}
    Using ${\bf W}_h \subset {\bf W}$, we can test equation (\ref{eq:var_eq_aus}) for any ${\bf w}_h \in {\bf W}_h$, leading to:
    \begin{equation} \label{eq:var_eq_aus_discrete}
        a({\bf u},{\bf w}_h)+b({\bf w}_h,{\bf p}) = F({\bf w}_h) \quad \forall {\bf w}_h \in {\bf W}_h.
    \end{equation}
    Subtracting (\ref{eq:var_eq_aus_discrete}) and (\ref{eq:var_eq_aus_discrete_0}), we get for all ${\bf w}_h \in {\bf W}_h$
    \begin{equation} \label{eq:mu1_mu2}
        b({\bf w}_h,{\bf p}-{\bf p}_h) = \langle F - F_h,{\bf w}_h \rangle_{{\bf V}'\times {\bf V}} + \big(a_h({\bf u}_h,{\bf w}_h)-a({\bf u},{\bf w}_h)\big)=: \mu_1({\bf w}_h)+\mu_2({\bf w}_h).
    \end{equation}
    The term $\mu_1({\bf w}_h)$ can be easily bounded by
    \begin{equation} \label{eq:estimate_mu1}
        \abs{\mu_1({\bf w}_h)} = \abs{\langle F - F_h,{\bf w}_h \rangle}_{{\bf V}'\times {\bf V}} \leq \norm{F-F_h}_{{\bf V}'} \norm{{\bf w}_h}_{{1,\Omega}}.
    \end{equation}
    For the term $\mu_2({\bf w}_h)$, we have
        \begin{align*}
            \mu_2({\bf w}_h) &= \sum_{K \in \mathcal{T}_h} \Big( a_h^{K}({\bf u}_h,{\bf w}_h)-a^{K}({\bf u},{\bf w}_h) \Big) \quad (\text{use } \pm {\bf u}_\pi\text{ and }(\ref{eq:k_consistency}) ) \\
            &= \sum_{K \in \mathcal{T}_h} \Big( a_h^{K}({\bf u}_h-{\bf u}_\pi,{\bf w}_h)+a^{K}({\bf u}_\pi-{\bf u},{\bf w}_h) \Big) \quad (\text{use }(\ref{eq:local_discrete_continuity}))\\
            &\leq 2\mu_{max} \big( \asih \abs{{\bf w}_h}_{1,\Omega}\abs{{\bf u}_h-{\bf u}_\pi}_{1,\Omega,h} + \abs{{\bf w}_h}_{1,\Omega}\abs{{\bf u}-{\bf u}_\pi}_{1,\Omega,h}\big)\\
            &\leq 2\mu_{max}\asuph\abs{{\bf u}_h-{\bf u}_\pi}_{1,\Omega,h} \norm{{\bf w}_h}_{1,\Omega}+2\mu_{max}\abs{{\bf u}-{\bf u}_\pi}_{1,\Omega,h} \norm{{\bf w}_h}_{1,\Omega},
        \end{align*}
    where ${\bf u}_\pi$ is a piecewise polynomial of degree $k$ on each $K \in \mathcal{T}_h$. Thus
    \begin{equation}\label{eq:estimate_mu2}
        \abs{\mu_2({\bf w}_h)} \leq 2\mu_{max}\asuph\abs{{\bf u}_h-{\bf u}_\pi}_{1,\Omega,h} \norm{{\bf w}_h}_{1,\Omega}+2\mu_{max}\abs{{\bf u}-{\bf u}_\pi}_{1,\Omega,h} \norm{{\bf w}_h}_{1,\Omega}.
    \end{equation}
    Then, combining (\ref{eq:estimate_mu1}) and (\ref{eq:estimate_mu2}) in (\ref{eq:mu1_mu2}), we get
    \begin{equation} \label{eq:pressure_second_part}
        \frac{\abs{b({\bf w}_h,{\bf p}-{\bf p}_h)}}{\norm{w_h}_{1,\Omega}} \leq \norm{F-F_h}_{{\bf V}'} +   2\mu_{max}\asuph\abs{{\bf u}_h-{\bf u}_\pi}_{1,\Omega,h} +2\mu_{max}\abs{{\bf u}-{\bf u}_\pi}_{1,\Omega,h}.
    \end{equation}
    Inserting (\ref{eq:pressure_first_part}) and (\ref{eq:pressure_second_part}) in (\ref{eq:pressure_aus}) and applying triangle's inequality, we have
    \begin{align*} 
        \beta_h \norm{{\bf p}_I-{\bf p}_h}_{{\bf Q}/{\bf H}} &\leq \norm{F-F_h}_{{\bf V}'} + 2\mu_{max}\;\asuph\abs{{\bf u}_h-{\bf u}}_{1,\Omega} + \\ 
        &+2\mu_{max}\;(\asuph+1)\abs{{\bf u}-{\bf u}_\pi}_{1,\Omega,h} + \norm{{\bf p}_I - {\bf p}}_{{\bf Q}/{\bf H}}.
    \end{align*}
    Defining $\tilde C_p(h):=\max\{\beta_h^{-1},2\mu_{max}\beta_h^{-1}(\asuph+1)\}$, we get
    \begin{equation} \label{eq:estimate_pressure_intermidiate}
        \norm{{\bf p}_I-{\bf p}_h}_{{\bf Q}/{\bf H}} \leq {\tilde C}_p(h) \big(\norm{F-F_h}_{{\bf V}'} + \abs{{\bf u}_h-{\bf u}}_{1,\Omega} + \abs{{\bf u}-{\bf u}_\pi}_{1,\Omega,h} + \norm{{\bf p}_I - {\bf p}}_{{\bf Q}/{\bf H}}\big).
    \end{equation}
    Finally, from (\ref{eq:pressure_tri_ineq}) we get
    \begin{equation*}
        \norm{{\bf p}-{\bf p}_h}_{{\bf Q}/{\bf H}} \leq C_p(h) \Big(\norm{F-F_h}_{{\bf V}'} + \norm{{\bf u}_h-{\bf u}}_{{\bf V}} + \abs{{\bf u}-{\bf u}_\pi}_{1,\Omega,h} + \norm{{\bf p}_I - {\bf p}}_{{\bf Q}/{\bf H}} \Big),
    \end{equation*}
    where $C_p(h):=\max\{1,\tilde C_p(h)\}$.
\end{proof}
%
%
Before proceeding to estimate the error on the displacement field, we notice that the continuity of the trace operator (see \cite{Kikuchi_Oden})
implies the existence of a positive constant $C_{\Gamma_C}$ such that 
\begin{equation}\label{eq:Ctrace}
    \int_{\Gamma_C} \psi \llbracket v_{n} \rrbracket d\Gamma \leq C_{\Gamma_C} \norm{\psi}_{H^{1/2}_{00}(\Gamma_C)'} \norm{{\bf v}}_{{\bf V}} 
    \quad \text{for all $\psi\in H^{1/2}_{00}(\Gamma_C)'$ and $\bf{v}\in {\bf V}$}.
\end{equation}
We are ready to prove the following result.
\begin{proposition} \label{prop:displacement_estimate}
    Let $({\bf u},{\bf p}) \in {\bf K}\times {\bf Q}$ and $({\bf u}_h,{\bf p}_h) \in {\bf K}_h \times {\bf Q}_h$ be the solutions of the continuous and discrete problems, respectively. Then, for every ${\bf u}_I \in {\bf K}_h$, ${\bf p}_I \in {\bf Q}_h$ such that ${({{\bf p}}_I)}_0={\bf p}_0$, $F_h \in {\bf V}_h'$, ${\bf u}_{\pi}=({\bf u}^1_{\pi},{\bf u}^2_{\pi})$ such that ${\bf u}_{\pi|K}\in[\Pol_k(K)]^2$ and $\psi_h \in L^2(\Gamma_C)$ such that $\psi_{h|e^c_l} \in \Pol_0(e^c_l)$ for all $0 \le l \le l^{\star}-1$, it holds
    \begin{equation*}
    \begin{split}
        \norm{{\bf u}-{\bf u}_h}_{{\bf V}}^2 &\leq \frac{4}{\asdownh\alpha} \int_{\Gamma_C} \psi_h \big( \llbracket u_{I,n} \rrbracket - \llbracket u_{h,n} \rrbracket \big) \;d\Gamma +
         C(h) \Big(\norm{{\bf u}-{\bf u}_I}_{{\bf V}}^2 + \abs{{\bf u} - {\bf u}_\pi}_{1,\Omega,h}^2 + \\
         &+\norm{{\bf p}- {\bf p}_I}_{{\bf Q}/{\bf H}}^2 + \norm{F-F_h}_{\bf V'}^2 + \norm{\sigma_n -\psi_h}^2_{H^{1/2}_{00}(\Gamma_C)'}\Big).
    \end{split}
    \end{equation*}
    Setting
    \begin{equation*}
            C_1(h):=\frac{96}{\asdownh \alpha}\mu_{max}^2\max\big\{(\asuph)^2,1\big\},
    \end{equation*}
    \begin{equation*}
        C_2(h):=\frac{1}{2}C_p(h)^2 \max\big\{ 5+\frac{1}{\lambda_{min}}+\frac{6}{\asdownh\alpha},1+\frac{5}{\lambda_{min}}+\frac{6}{\asdownh\alpha}(1+\frac{1}{\lambda_{min}^2})\big\},
    \end{equation*}
    and recalling \eqref{eq:Ctrace}, the constant $C(h)$ is given by
    \begin{equation} \label{eq:constant_displacement}
        C(h):=\frac{8}{\asdownh\alpha} \max\big\{ C_1(h),C_2(h),\frac{3}{\asdownh\alpha}, \frac{3C_{\Gamma_C}^2}{\asdownh\alpha}\big\}.
    \end{equation}
\end{proposition}

\begin{proof}
Let $({\bf u},{\bf p}) \in {\bf K} \times {\bf Q}$ and $({\bf u}_h,{\bf p}_h) \in {\bf K}_h\times {\bf Q}_h$ be the solutions of the continuous and discrete problems, respectively. Let ${\bf u}_I \in {\bf K}_h$ and we set $\boldsymbol{\delta}_h:={\bf u}_I -{\bf u}_h$. Finally, let ${\bf u}_\pi$ be a piecewise polynomial function, i.e. ${\bf u}_{\pi|K} \in [\Pol_k(K)]^2$ for all $K \in \mathcal{T}_h$. Thus, using standard technique we have:
\begin{equation}\label{eq:delta-est}
\begin{aligned}
    \asdownh &\alpha \norm{\boldsymbol{\delta}_h}_{1,\Omega }^2 \leq a_h(\boldsymbol{\delta}_h,\boldsymbol{\delta}_h) =  a_h({\bf u}_I,\boldsymbol{\delta}_h) - a_h({\bf u}_h,\boldsymbol{\delta}_h) \\ 
    & \leq \sum_{K \in \mathcal{T}_h} \Big( a^{K}_h({\bf u}_I -{\bf u}_\pi,\boldsymbol{\delta}_h) +a^{K}({\bf u}_\pi-{\bf u},\boldsymbol{\delta}_h) \Big) + \Big( a({\bf u},\boldsymbol{\delta}_h)+ b(\boldsymbol{\delta}_h,{\bf p})-F(\boldsymbol{\delta}_h) \Big)\\
    & + \Big( b(\boldsymbol{\delta}_h,{\bf p}_h)- b(\boldsymbol{\delta}_h,{\bf p}) \Big)
    + \Big(F(\boldsymbol{\delta}_h)- F_h(\boldsymbol{\delta}_h)\Big)=: R_1+R_2+R_3+R_4.
\end{aligned}
\end{equation}
Now, we are going to estimate $R_1,R_2,R_3$ and $R_4$ separately. \\
\par 
\emph{Estimate of $R_1$.} Using local continuity of the bilinear forms
\begin{equation*}
    R_1 \leq  \sum_{K \in \mathcal{T}_h}\left( 2\mu_{K}\asihK \abs{{\bf u}_I -{\bf u}_\pi}_{1,K} \abs{\boldsymbol{\delta}_h}_{1,K} + 2\mu_{K}\abs{{\bf u}_\pi -{\bf u}}_{1,K} \abs{\boldsymbol{\delta}_h}_{1,K}\right) .
\end{equation*}
By applying Cauchy-Schwarz and triangle inequalities, we have
\begin{align*}
    R_1 &\leq  2\mu_{max}\asih \abs{{\bf u}_I -{\bf u}_\pi}_{1,\Omega,h} \abs{\boldsymbol{\delta}_h}_{1,\Omega} + 2\mu_{max}\abs{{\bf u}_\pi -{\bf u}}_{1,\Omega,h}\abs{\boldsymbol{\delta}_h}_{1,\Omega}\\
    &\leq 4\mu_{max} \max\{\asuph,1\} \abs{\boldsymbol{\delta}_h}_{1,\Omega} \big(\abs{{\bf u}_I -{\bf u}}_{1,\Omega} + \abs{{\bf u}_\pi -{\bf u}}_{1,\Omega,h} \big).
\end{align*}
Finally, by Young's inequality, we obtain
\begin{equation} \label{eq:R1}
    R_1 \leq \frac{1}{12} \asdownh \alpha \norm{\boldsymbol{\delta}_h}_{{\bf V}}^2 + C_1(h)(\norm{{\bf u}_I -{\bf u}}_{{\bf V}}^2 + \abs{{\bf u}_\pi -{\bf u}}_{1,\Omega,h}^2),
\end{equation}
where
\begin{equation} \label{eq:constant_R1}
    C_1(h):=\frac{96}{\asdownh \alpha}\mu_{max}^2\max\big\{(\asuph)^2,1\big\}.
\end{equation}
\par
\emph{Estimate of $R_2$}. (The treatment of this term follows the lines detailed in \cite{Belgacem_Renard_Slimane} and \cite{Belhachmi_BenBelgacem}). It is straightforward to check that $R_2$ can be written as
\begin{equation*}
    R_2 = a({\bf u},\boldsymbol{\delta}_h) + b(\boldsymbol{\delta}_h,{\bf p})-F(\boldsymbol{\delta}_h) = \int_{\Gamma_C} \sigma_n \big( \llbracket u_{I,n} \rrbracket - \llbracket u_{h,n} \rrbracket \big) d\Gamma.
\end{equation*}
Let $\psi_h \in L^2(\Gamma_C)$ such that $\psi_{h|e^c_l} \in \Pol_0(e^c_l)$ for all $0 \le l \le l^{\star}-1$. By adding and subtracting $\psi_h$, we get:
\begin{equation*}
    R_2 = \int_{\Gamma_C} (\sigma_n -\psi_h) \big( \llbracket u_{I,n} \rrbracket - \llbracket u_{h,n} \rrbracket \big) d\Gamma + \int_{\Gamma_C} \psi_h \big( \llbracket u_{I,n} \rrbracket - \llbracket u_{h,n} \rrbracket \big) d\Gamma.
\end{equation*}
Now, we are going to bound the first term. 
By estimate \eqref{eq:Ctrace} and Young's inequality, we have
\begin{align*}
    \int_{\Gamma_C} (\sigma_n -\psi_h) \big( \llbracket u_{I,n} \rrbracket - \llbracket u_{h,n} \rrbracket \big) d\Gamma &\leq C_{\Gamma_C} \norm{\sigma_n - \psi_h}_{H^{1/2}_{00}(\Gamma_C)'} \norm{\boldsymbol{\delta}_h}_{{1,\Omega}} \\
    &\leq \frac{1}{12}\asdownh \alpha \norm{\boldsymbol{\delta}_h}_{{1,\Omega}}^2 + \frac{3C_{\Gamma_C}^2}{\asdownh \alpha} \norm{\sigma_n - \psi_h}_{H^{1/2}_{00}(\Gamma_C)'}^2.
\end{align*}
Thus
\begin{equation} \label{eq:R2}
    R_2 \le \frac{1}{12}\asdownh \alpha \norm{\boldsymbol{\delta}_h}_{{\bf V}}^2 + \frac{3C_{\Gamma_C}^2}{\asdownh \alpha} \norm{\sigma_n - \psi_h}_{H^{1/2}_{00}(\Gamma_C)'}^2 + \int_{\Gamma_C} \psi_h \big( \llbracket u_{I,n} \rrbracket - \llbracket u_{h,n} \rrbracket \big) d\Gamma.
\end{equation}
\par
\emph{Estimate of $R_3$}. For any ${\bf p}_I \in {\bf Q}_h$ such that ${({\bf p}_{I})}_0={\bf p}_0$, we have (adding and subtracting ${\bf p}_I$ and ${\bf u}$, and using \eqref{eq:Discrete_problem} and \eqref{eq:contact_problem})
\begin{align*}
    R_3 =  b(\boldsymbol{\delta}_h,{\bf p}_h)- b(\boldsymbol{\delta}_h,{\bf p}) = \big( b({\bf u}_I-{\bf u},{\bf p}_h - {\bf p}_I )  - b(\boldsymbol{\delta}_h, {\bf p} - {\bf p}_I ) \big) - c_\lambda ({\bf p}_h-{\bf p},{\bf p}_h-{\bf p}_I)=: I +II
\end{align*}
We bound $I$ and $II$ separately. For $I$ we easily have
\begin{align*}
    I &= b({\bf u}_I-{\bf u},{\bf p}_h- {\bf p}_I) - b(\boldsymbol{\delta}_h,{\bf p}- {\bf p}_I)\leq \norm{{\bf u}_I-{\bf u}}_{{1,\Omega}}\norm{ {\bf p}_h- {\bf p}_I}_{{\bf Q}/{\bf H}} + \norm{\boldsymbol{\delta}_h}_{{1,\Omega}} \norm{{\bf p}- {\bf p}_I}_{{\bf Q}/{\bf H}}.
\end{align*}
By applying estimate (\ref{eq:estimate_pressure_intermidiate}) and triangle inequality, we obtain
\begin{equation}\label{eq:R3I}
\begin{split}
    I &\leq {\tilde C}_p(h) \big(\norm{F-F_h}_{{\bf V}'} + \norm{\boldsymbol{\delta}_h}_{{1,\Omega}} + \norm{{\bf u}-{\bf u}_I}_{{1,\Omega}} + \abs{{\bf u}-{\bf u}_\pi}_{1,\Omega,h} +\\
    &+\norm{{\bf p}_I - {\bf p}}_{{\bf Q}/{\bf H}}\big)\norm{{\bf u}_I-{\bf u}}_{{1,\Omega}}+\norm{\boldsymbol{\delta}_h}_{{1,\Omega}} \norm{{\bf p}_I - {\bf p}}_{{\bf Q}/{\bf H}}.
\end{split}
\end{equation}
For $II$ we get (adding and subtracting ${\bf p}$, and recalling ${({\bf p}_{I})}_0={\bf p}_0$, the definition of $c_\lambda(\cdot,\cdot)$ and $\lambda_{min}=\min{(\lambda^1,\lambda^2)}$)
\begin{align*}
    II &= - c_\lambda ({\bf p}_h-{\bf p},{\bf p}_h-{\bf p}_I) 
    = - c_\lambda({\bf p}_h-{\bf p},{\bf p}_h-{\bf p})  - c_\lambda({\bf p}_h-{\bf p},{\bf p}-{\bf p}_I)  \\
    &\leq  - c_\lambda({\bf p}_h-{\bf p},{\bf p}-{\bf p}_I) = - c_\lambda (\bar {\bf p}_h-\bar {\bf p},\bar {\bf p}-\bar {\bf p}_I) \leq \frac{1}{\lambda_{min}}\norm{{\bf p}_h-{\bf p}}_{{\bf Q}/{\bf H}} \norm{{\bf p}-{\bf p}_I}_{{\bf Q}/{\bf H}}.
\end{align*}
Using Proposition \ref{prop:pressure_estimate} and triangle inequality, we get
\begin{equation}\label{eq:R3II}
    \begin{split}
    II &\leq\frac{1}{\lambda_{min}}C_p(h) \big(\norm{\boldsymbol{\delta}_h}_{{1,\Omega}}+\norm{{\bf u} - {\bf u}_I}_{{1,\Omega}} + \abs{{\bf u}-{\bf u}_\pi}_{1,\Omega,h} + \\
    &+\norm{{\bf p}-{\bf p}_I}_{{\bf Q}/{\bf H}} + \norm{F-F_h}_{{\bf V}'} \big)\norm{{\bf p}-{\bf p}_I}_{{\bf Q}/{\bf H}}.
    \end{split}
\end{equation}
Thus, combining (\ref{eq:R3I}) and (\ref{eq:R3II}), applying Young's inequality suitably, and observing that $C_p(h) \geq \max\{1,\tilde C_p(h)\}$, we have
\begin{equation} \label{eq:R3}
    R_3 \leq \frac{3}{12} \asdownh\alpha \norm{\boldsymbol{\delta}_h}_{{\bf V}}^2 + C_2(h) \big(\norm{{\bf u}-{\bf u}_I}_{{\bf V}}^2 + \abs{{\bf u} - {\bf u}_\pi}_{1,\Omega,h}^2 + \norm{{\bf p}- {\bf p}_I}_{{\bf Q}/{\bf H}}^2 + \norm{F-F_h}_{\bf V'}^2 \big),
\end{equation}
where
\begin{equation}
    C_2(h):=\frac{1}{2}C_p(h)^2 \max\big\{ 5+\frac{1}{\lambda_{min}}+\frac{6}{\asdownh\alpha},1+\frac{5}{\lambda_{min}}+\frac{6}{\asdownh\alpha}(1+\frac{1}{\lambda_{min}^2})\big\}.
\end{equation}
\par
\emph{Estimate of $R_4$}. By continuity, we have
\begin{equation*}
    R_4 =  F(\boldsymbol{\delta}_h)- F_h(\boldsymbol{\delta}_h) 
    \leq \norm{F-F_h}_{{\bf V}'} \norm{{\boldsymbol{\delta}}_h}_{{\bf V}}.
\end{equation*}
Consequently, by applying Young's inequality, we get
\begin{equation} \label{eq:R4}
    R_4 \leq \frac{1}{12} \asdownh\alpha\norm{\boldsymbol{\delta}_h}^2_{{\bf V}}+\frac{3}{\asdownh\alpha}\norm{F-F_h}_{{\bf V}'}^2.
\end{equation}
Combining (\ref{eq:R1}), (\ref{eq:R2}), (\ref{eq:R3}) and (\ref{eq:R4}), from \eqref{eq:delta-est} we infer
\begin{equation} \label{eq:partial_estimate}
    \begin{split}
        \frac{1}{2}\asdownh \alpha \norm{\boldsymbol{\delta}_h}_{{\bf V}}^2 &\leq C_3(h) \big(\norm{{\bf u}-{\bf u}_I}_{{\bf V}}^2 + \abs{{\bf u} - {\bf u}_\pi}_{1,\Omega,h}^2 + \norm{{\bf p}- {\bf p}_I}_{{\bf Q}/{\bf H}}^2 + \norm{F-F_h}_{\bf V'}^2 \\
        &+ \norm{\sigma_n - \psi_h}_{H^{1/2}_{00}(\Gamma_C)'}^2\big)
        + \int_{\Gamma_C} \psi_h \big( \llbracket u_{I,n} \rrbracket - \llbracket u_{h,n} \rrbracket \big) d\Gamma,
    \end{split}
\end{equation}
where
\begin{equation}
    C_3(h):=2\max\big\{ C_1(h), C_2(h), \frac{3}{\asdownh\alpha}, \frac{3C_{\Gamma_C}^2}{\asdownh\alpha} \big\}.
\end{equation}
From (\ref{eq:partial_estimate}) and triangle inequality, the statement follows.
\end{proof}

Finally, combining Propositions \ref{prop:pressure_estimate} and \ref{prop:displacement_estimate}, we have the following convergence result for the contact problem.
\begin{theorem} \label{teo:abstract_convergence_result}
    Let $({\bf u},{\bf p}) \in {\bf K}\times {\bf Q}$ and $({\bf u}_h,{\bf p}_h) \in {\bf K}_h \times {\bf Q}_h$ be the solutions of the continuous and discrete problems, respectively. Then, for every ${\bf u}_I \in {\bf K}_h$, ${\bf p}_I \in {\bf Q}_h$ such that $({p}^i_{I})_0=p^i_0$, $F_h \in {\bf V}_h'$, ${\bf u}_{\pi}=({\bf u}^1_{\pi},{\bf u}^2_{\pi})$ such that ${\bf u}^i_{\pi|K}\in[\Pol_k(K)]^2$ and $\psi_h \in L^2(\Gamma_C)$ such that $\psi_{h|e^c_l} \in \Pol_0(e^c_l)$ for all $0 \le l \le l^{\star}-1$, it holds
    \begin{equation*}
    \begin{split}
        \norm{{\bf u}-{\bf u}_h}_{{\bf V}}^2 + \norm{{\bf p}-{\bf p}_h}_{{\bf Q}/{\bf H}}^2&\leq C_{fin}(h) \Big( \int_{\Gamma_C} \psi_h \big( \llbracket u_{I,n} \rrbracket - \llbracket u_{h,n} \rrbracket \big) \;d\Gamma + \norm{{\bf u}-{\bf u}_I}_{{\bf V}}^2 + \\
        &+\abs{{\bf u} - {\bf u}_\pi}_{1,\Omega,h}^2 + \norm{{\bf p}- {\bf p}_I}_{{\bf Q}/{\bf H}}^2 + \norm{F-F_h}_{\bf V'}^2 + \norm{\sigma_n -\psi_h}^2_{H^{1/2}_{00}(\Gamma_C)'}\Big).
    \end{split}
    \end{equation*}
    The constant $C_{fin}(h)$ is given by
    \begin{equation} \label{eq:abstract_constant}
        C_{fin}(h):=2 \max\big\{ 8C_p(h)^2,\frac{4}{\asdownh\alpha}(1+8C_p(h)^2), (1+8C_p(h)^2) C(h)\big\},
    \end{equation}
    where $C_p(h)$ and $C(h)$ are defined as in (\ref{eq:constant_pressure}) and (\ref{eq:constant_displacement}), respectively.
\end{theorem}

\begin{remark}\label{rm:incompr}
    The constant $C_{fin}(h)$ in the Theorem \ref{teo:abstract_convergence_result} depends on the Lamé coefficients. However, it remains bounded as the incompressibility parameters $\lambda^i$ increase because it depends only on $1/\lambda_{min}$. Therefore, no volumetric locking effect appears for nearly incompressible materials.
\end{remark}

\begin{remark}\label{rm:quantities}
    Given the material properties, the quantity $C_{fin}(h)$ in the Theorem \ref{teo:abstract_convergence_result} depends on the discrete scheme through the quantities $\aish$ (see \eqref{eq:stability} and \eqref{eq:constant_discrete_global_coercivity}), $\asuph$ (see \eqref{eq:stability} and \eqref{eq:constant_continuity}) and $\beta_h$ (see \eqref{eq:discrete_inf_sup}). 
\end{remark}
\section{Two examples of Virtual Element Approximation}\label{sec:VEMexamples}
In this section, we detail a couple of specific choices of the spaces ${\bf V}_h$ and ${\bf Q}_h$, together with a suitable convex subset ${\bf K}_h$ and discrete forms. We set
\begin{equation}\label{eq:m123}
    m_1(x,y):=1\ , \quad m_2(x,y):= \frac{{x}-{x}_K}{h_K}\ ,\quad m_3(x,y):= \frac{{y}-{y}_K}{h_K},
\end{equation}
where $(x_K,y_K)$ denotes the centroid of the element $K$.
\subsection{A first order element (\texorpdfstring{$k=1$}{k=1}): local spaces}\label{sss:first}
On each element $K \in \mathcal{T}_h^i$, let us consider the following Stokes-like local virtual element space. First, consider the problem: find ${\bf v}$ and $s$ such that
\begin{equation} \label{eq:Stokes_element_system_lo}
\begin{cases}
    -{\bf \Delta v} - \nabla s = {\bf{0}} \quad \text{in } K, \\
    \text{div }{\bf v} \in \Pol_0(K), \\
    {\bf v}_{|\partial K} \in [C^0(\partial K)]^2, \quad ({\bf v}\cdot{\bf t}_e)_{|e} \in \Pol_1(e),\quad ({\bf v}\cdot{\bf n}_e)_{|e} \in \Pol_2(e) \text{ for all } e \in \partial K,
\end{cases}
\end{equation}
Above ${\bf t}_e$ and ${\bf n}_e$ are the tangential and normal vectors to the edge $e$, respectively, and 
all the equations are intended in a distributional sense. Then, we set

\begin{equation*}
{\bf V}_{h|K}^i:= \{ {\bf v} \in [H^1(K)]^2 : {\bf v} \text{ satisfies  (\ref{eq:Stokes_element_system_lo}) for some $s \in L^2(K)$ }\}.
\end{equation*}

Given ${\bf v} \in {\bf V}_{h|K}^i$, the associated set of local degrees of freedom $\Xi$ is given by (see Fig. \ref{fig:k1}):
\begin{itemize}
    \item the vector values $\Xi_v({\bf v})$ of ${\bf v}$ at the vertices of the polygon $K$,
    \item the value $\Xi_e({\bf v})$ of ${\bf v}\cdot{\bf n}$ at the middle point of each edge $e \in \partial K$.
\end{itemize}
Moreover, on each $K \in \mathcal{T}_h^i$, we consider the local polynomial space
\begin{equation*}
    Q_{h|K}^i:=\Pol_0(K),
\end{equation*}
whose local degree of freedom $D_Q$ is given by 
    \[
    D_{Q}(q):=\frac{1}{\abs{K}}\int_{K} q \;dx .
    \]

\subsection{A quadratic element (\texorpdfstring{$k=2$}{k=2}): local spaces}\label{sss:second}
On each element $K \in \mathcal{T}_h^i$, let us consider the following Stokes-like local virtual element space (see \cite{BLV:2017}). First, consider the problem: find ${\bf v}$ and $s$ such that
\begin{equation} \label{eq:Stokes_element_system}
\begin{cases}
    -{\bf \Delta v} - \nabla s = {\bf{0}} \quad \text{in } K, \\
    \text{div }{\bf v} \in \Pol_1(K), \\
    {\bf v}_{|\partial K} \in [C^0(\partial K)]^2, \quad {\bf v}_{|e} \in [\Pol_2(K)]^2 \text{ for all } e \in \partial K,
\end{cases}
\end{equation}
with all the equations to be intended in a distributional sense. Then, we set
\begin{equation*}
{\bf V}_{h|K}^i:= \{ {\bf v} \in [H^1(K)]^2 : {\bf v} \text{ satisfies  (\ref{eq:Stokes_element_system}) for some $s \in L^2(K)$ }\}.
\end{equation*}

Given ${\bf v} \in {\bf V}_{h|K}^i$, the associated set of local degrees of freedom $\Xi$ (divided into boundary ones $\Xi^\partial$ and internal ones $\Xi^\circ$) is given by (see Fig. \ref{fig:k2}):
\begin{itemize}
    \item the vector values $\Xi^\partial_v({\bf v})$ of ${\bf v}$ at the vertices of the polygon $K$,
    \item the vector values $\Xi^\partial_e({\bf v})$ of ${\bf v}$ at the middle point of each edge $e \in \partial K$,
    \item the internal moments of $\text{div}\; \bf v$ against the scaled polynomial basis $\{ m_j \}_{j =2,3}$ of $\Pol_1(K)/\R$, \emph{i.e.}
    \begin{equation}\label{eq:divDoFs}
    \Xi^\circ_j({\bf v}):=\frac{h_K}{\abs{K}}\int_{K} (\text{div}\; {\bf v}) m_j \; dx \quad j=2,3.
     \end{equation}
\end{itemize}
Moreover, on each $K \in \mathcal{T}_h^i$, we consider the local polynomial space
\begin{equation*}
    Q_{h|K}^i:=\Pol_1(K),
\end{equation*}
whose local set of degrees of freedom $D_Q$ is given by the internal moments against the scaled polynomial basis $\{ m_j \}_{j=1,2,3}$ of $\Pol_1(K)$, \emph{i.e.}
    \[
    D_{Q,j}(q):=\frac{1}{\abs{K}}\int_{K} q \: m_j \;dx \quad j=1,2,3.
    \]

\begin{figure}[htbp]
    \centering
    \begin{subfigure}[t]{0.48\textwidth}
        \centering
        \includegraphics[width=\linewidth]{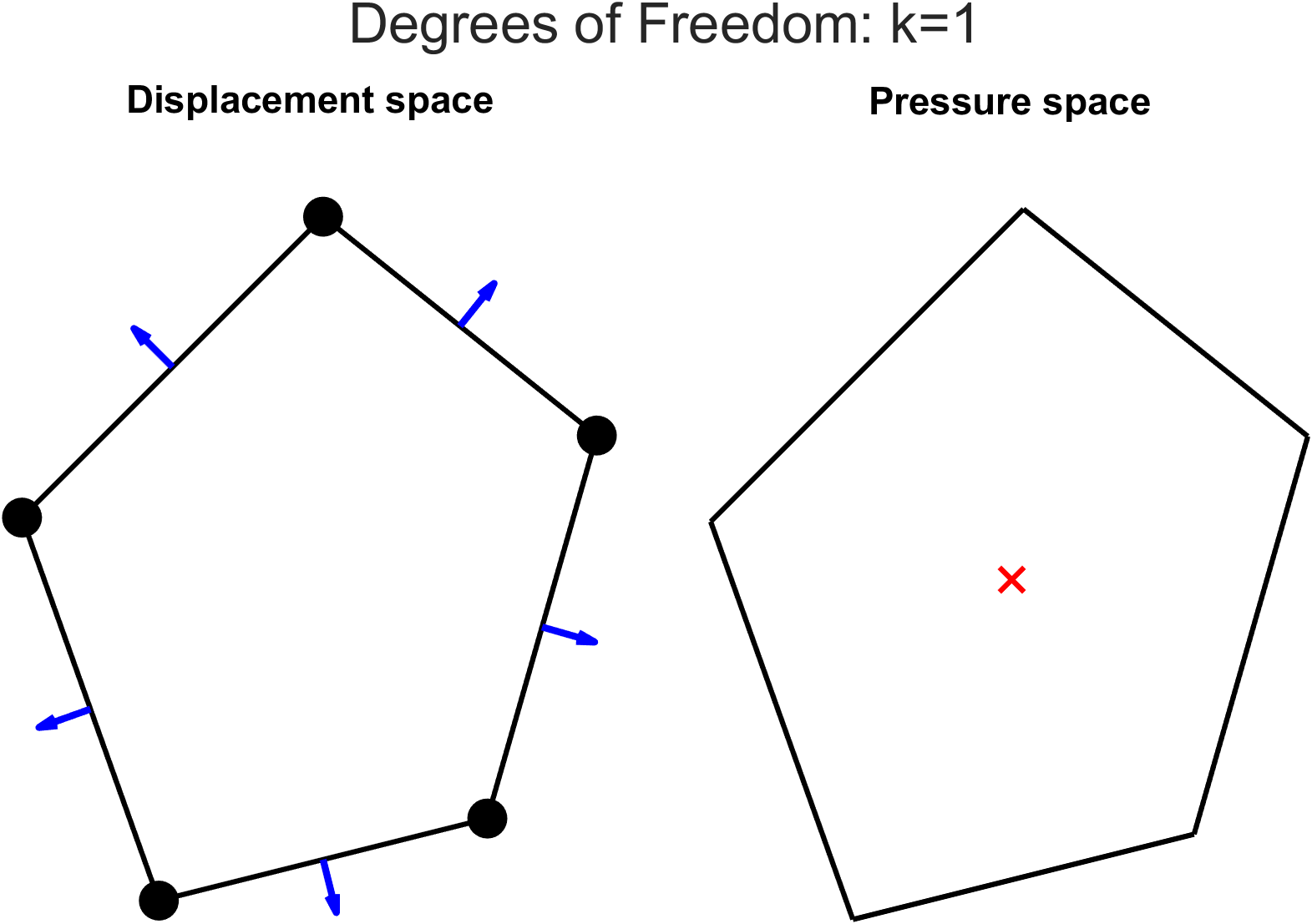}
        \caption{We denote $\Xi_v$ with the black dots, $\Xi_e$ with the blue arrows and $D_{Q}$ with red crosses inside the element.}
        \label{fig:k1}
    \end{subfigure}
    \hfill
    \begin{subfigure}[t]{0.48\textwidth}
        \centering
        \includegraphics[width=\linewidth]{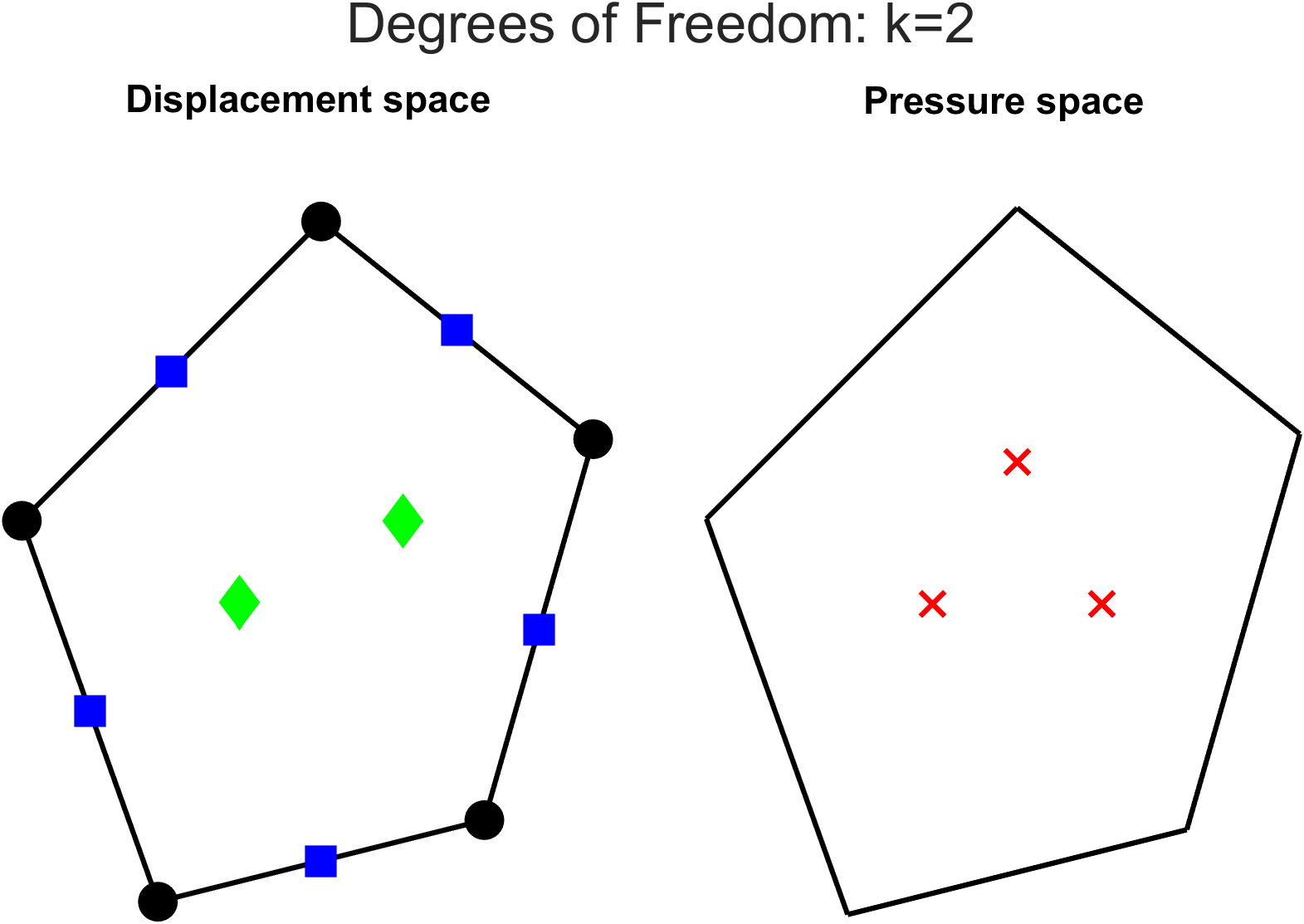}
        \caption{We denote $\Xi^\partial_v$ with the black dots, $\Xi^\partial_e$ with the blue squares, $\Xi^\circ_j$ with the internal green diamonds and $D_{Q,j}$ with the internal crosses.}
        \label{fig:k2}
    \end{subfigure}
    \caption{Degrees of freedom for $k=1$ and $k=2$.}
    \label{fig:dofs}
\end{figure}

\subsection{Spaces and forms}
For both the elements in sections \ref{sss:first} and \ref{sss:second}, the global virtual element spaces (for $i=1,2$) are
\begin{equation*}
    {\bf V}_h^i:=\{{\bf v}_h^i \in {\bf V}^i : {\bf v}_{h|K}^i \in {\bf V}_{h|K}^i \quad \text{for all } K \in \mathcal{T}_h^i \},
\end{equation*}
and
\begin{equation*}
    Q_h^i:=\{ q_h^i \in L^2(\Omega^i) : q_{h|K}^i \in Q_{h|K}^i \quad \text{ for all } K \in \mathcal{T}_h^i \},
\end{equation*}
with the obvious associated set of degrees of freedom. In addition, we set
\begin{equation} \label{eq:Stokes_like_global_spaces}
    {\bf V}_h:= {\bf V}_h^1 \times {\bf V}_h^2, \qquad {\bf Q}_h:=Q_{h}^1 \times Q_h^2.
\end{equation}

We remark that, from the definition of the local VEM spaces, it follows immediately
\begin{equation*}
    \text{div}\; {\bf V}^i_h \subseteq Q^i_h.
\end{equation*}

The construction of the local discrete bilinear form $a^{K}_h(\cdot,\cdot)$ is carried out by following the abstract construction of Section \ref{sec:VEM_contact}. Here, however, we provide a specific choice of both the local energy projection and the stabilizing bilinear form. The energy projection operator $\Pi_2^{\varepsilon,K}$ is defined as in (\ref{eq:energy_projection_operator}). To fix the rigid body motion of the projection, we need to specify the operators $P_j^{K}$ for $j=1,2,3$. Observing that the space $RM(K)$ of rigid body motions over the element $K$ can be written as (cf. \eqref{eq:m123})
\begin{equation*}
    RM(K) = \text{span } \Big\{ 
    \begin{pmatrix}
        m_1 \\
        0
    \end{pmatrix},
    \begin{pmatrix}
        0 \\
        m_1
    \end{pmatrix},
    \begin{pmatrix}
        -m_3 \\
        m_2
    \end{pmatrix}
    \Big\},
\end{equation*}
we define $\forall {\bf v} \in [H^1(K)]^2$
\begin{equation*}
    P_1^{K}({\bf v}):= \frac{1}{\abs{\partial K}}\int_{\partial K} {\bf v} \cdot  \begin{pmatrix}
        m_1 \\
        0
    \end{pmatrix} d\Gamma, 
    \quad P_2^{K}({\bf v}):= \frac{1}{\abs{\partial K}}\int_{\partial K} {\bf v} \cdot  \begin{pmatrix}
        0 \\
        m_1
    \end{pmatrix} d\Gamma,
\end{equation*}
\begin{equation*}
    \quad P_3^{K}({\bf v}):= \frac{1}{\abs{\partial K}}\int_{\partial K} {\bf v} \cdot  \begin{pmatrix}
        -m_3 \\
        m_2
    \end{pmatrix} d\Gamma.
\end{equation*}
The local stabilizing term $S^{K}(\cdot,\cdot)$, cf. \eqref{eq:abstr_ah}, is chosen as the classical \emph{DoFi-DoFi} bilinear form (see \cite{volley}), i.e.:
\begin{equation} \label{eq:dofi_dofi}
    S^{K}({\bf v}_h,{\bf w}_h) = \sum_{j} \Xi_j({\bf v}_h) \Xi_j ({\bf w}_h) \quad \forall {\bf v}_h,{\bf w}_h \in {\bf V}_{h|K}.
\end{equation}
The closed and convex subset ${\bf K}_h \subset {\bf V}_h$ is inspired by the choice made in \cite{Belgacem_Renard_Slimane}. In fact, the discrete contact conditions are imposed by enforcing the non-positivity of the normal discrete gap $\llbracket u_{h,n} \rrbracket - g_0$ at the vertices of the decomposition on $\Gamma_C$ and the non-positivity of its integral along the contact edges. More precisely, ${\bf K}_h$ is given by
\begin{equation} \label{eq:Stokes_like_convex_subset}
    {\bf K}_h:=\{ {\bf v}_h \in {\bf V}_h : (\llbracket v_{h,n} \rrbracket -g_0)({\bf x}_l^c) \leq 0 \quad \forall l=0,\cdots,l^\star, \: \int_{e^c_l} (\llbracket v_{h,n} \rrbracket -g_0) \; d\Gamma \leq 0 \quad \forall l=0,\cdots,l^{\star}-1 \}.
\end{equation}
We remark that choice \eqref{eq:Stokes_like_convex_subset} is suitable for both the linear and the quadratic elements of subsections \ref{sss:first} and \ref{sss:second}.  
Moreover, the subspace ${\bf W}_h \subset {\bf K}_h \cap {\bf W}$ is chosen as 
\begin{equation} \label{eq:Stokes_like_subspace_infsup}
{\bf W}_h={\bf W}_h^1 \times {\bf W}_h^2\quad \text{ with }\quad    {\bf W}_h^i :=\{{\bf v}_h \in {\bf V}_h^i : {\bf v}_{h|\Gamma_C} = {\bf 0} \} \text{ for } i=1,2.
\end{equation}
This subspace satisfies the discrete inf-sup condition (\ref{eq:discrete_inf_sup}). We defer the proof to section \ref{sssec:infsup-constants}. 

Finally, we provide a computable approximation $F_h \in {\bf V}_h'$ of the right-hand side. For any $K \in \mathcal{T}_h$, let ${\boldsymbol{\Pi}}^{0,K}_0 : [L^2(K)]^2 \to [\Pol_0(K)]^2$ be the $L^2$-projections onto constants. Then, we define
\begin{equation}\label{eq:fh1}
    {\bf f}^i_{h|K} := {\boldsymbol{\Pi}}^{0,K}_0 {\bf f}^i \quad \forall K \in \mathcal{T}_h^i.
\end{equation}
We set 
\begin{equation}\label{eq:fh2}
    \langle F_h,{\bf v}_h \rangle:=\sum_{i=1}^2 \langle{\bf f}_h^i,{\bf v}_h^i \rangle =\sum_{i=1}^2\sum_{K \in \mathcal{T}_h^i} \int_K {\bf f}_h^i \cdot {\bf v}_h^i \;dx .
\end{equation}
It's worth noting that $\langle F_h,{\bf v}_h \rangle$ is explicitly computable through the degrees of freedom (see \cite{BLV:2017} for details).

\section{Convergence analysis of the quadratic virtual element}\label{sec:quadratic-conv}
In this section, we provide the convergence analysis focusing on the quadratic element described in section \ref{sss:second}. For the first-order one, see section \ref{sss:first}, similar results can be obtained using analogous and slightly simplified arguments. 
The analysis takes advantage of the abstract framework of section \ref{ss:abstract_convergence}.

\subsection{Evaluation of the stability constants}\label{ssec:stab-constants}
We first show the validity of the discrete inf-sup condition (\ref{eq:discrete_inf_sup}) and the stability property (\ref{eq:stability}) under the assumptions (A1) - (A2) on the polygonal decompositions. 
As we have seen, both properties are essential for the analysis of the discrete problem. Special attention is given to the behavior of the constants $\aish$, $\asuph$ and $\beta_h$ as they contribute to determine the convergence order in Theorem \ref{teo:abstract_convergence_result} (see Remark \ref{rm:quantities}).
\subsubsection{Evaluating the continuity and coercivity constants \texorpdfstring{$\alpha_h^{\star}$}{αh*} and \texorpdfstring{$\alpha_{\star,h}$}{α*h}}\label{sssec:alpha-constants}
We begin by proving the following lemma.
\begin{lemma}\label{lm:fourier}
Assume that assumption (A1) holds true and fix $M>0$. Then, for every $M\geq\varepsilon>0$ and every $v\in H^{1/2+\varepsilon}(\partial K)$ we have
\begin{equation}\label{eq:fourier}
    || v ||_{\infty,\partial K}\lesssim \varepsilon^{-1/2} || v ||_{1/2+\varepsilon,\partial K} .    
\end{equation}
\end{lemma}
\begin{proof}
Due to assumption (A1), it suffices to prove estimate \eqref{eq:fourier} for $S^1=\partial B({\bf 0},1)\subset \C$ (i.e.: we identify $\C$ with $\R^2$). Then, given $v\in H^{1/2+\varepsilon}(S^1)$, we consider the Fourier expansion $v(x)= \sum_{n\in \Z} \widehat v_n e^{inx}$, where
$\widehat v_n = \frac{1}{2\pi }\int_{0}^{2\pi} v(\xi)e^{-in\xi}d\xi$. Since $v\in H^{1/2+\varepsilon}(S^1)$, we have
\begin{equation}\label{eq:fourier2}
\sum_{n\in\Z} (1+n^2)^{1/2+\varepsilon} |\widehat v_n|^2 \lesssim || v ||_{1/2+\varepsilon,S^1}^2 < +\infty\, 
\quad \text{ and } \quad  v\in C^0(S^1)  .
\end{equation}
Therefore, for every $x\in S^1$, it holds, setting $x=e^{i\varphi}$ with $\varphi\in [0,2\pi)$:
\begin{equation}\label{eq:fourier3}
\begin{aligned}
|v(x)| = |\sum_{n\in \Z} \widehat v_n e^{in\varphi}|&\leq \sum_{n\in \Z} |\widehat v_n| =
\sum_{n\in \Z} \frac{1}{(1+n^2)^{1/4+\varepsilon/2}}(1+n^2)^{1/4+\varepsilon/2}|\widehat v_n|\\ 
&\leq \left( \sum_{n\in \Z} \frac{1}{(1+n^2)^{1/2+\varepsilon}}\right)^{1/2} \left(\sum_{n\in\Z} (1+n^2)^{1/2+\varepsilon} |\widehat v_n|^2\right)^{1/2}\\
&\lesssim \left( \sum_{n\in \Z} \frac{1}{(1+n^2)^{1/2+\varepsilon}}\right)^{1/2}  || v ||_{1/2+\varepsilon,S^1}.
\end{aligned}
\end{equation}
Since for $M\geq \varepsilon>0$ we have
$$
 \sum_{n\in \Z} \frac{1}{(1+n^2)^{1/2+\varepsilon}}\lesssim \int_0^{+\infty}\frac{1}{(1+t^2)^{1/2+\varepsilon}} dt
 \lesssim \int_1^{+\infty}\frac{1}{t^{1+2\varepsilon}} dt = \frac{1}{2\varepsilon},
$$
from \eqref{eq:fourier3} we get
\begin{equation}\label{eq:fourier4}
|| v ||_{\infty,S^1}\lesssim \varepsilon^{-1/2} || v ||_{1/2+\varepsilon,S^1}.
\end{equation}
\end{proof}

The following result concerns the stability property of the DoFi-DoFi stabilization term. 
\begin{proposition}
    Let assumptions (A1) and (A2) hold true. Then
    \begin{equation}\label{eq:coerc_1}
        S^{K}({\bf v}_h,{\bf v}_h) \lesssim C(h) \; a^{K} ({\bf v}_h,{\bf v}_h) \quad \forall {\bf v}_h \in {\bf V}_{h|K} \text{ s.t. } \Pi^{\varepsilon,K}_2 {\bf v}_h = {\bf 0},
    \end{equation}
    \begin{equation}\label{eq:coerc_2}
        a^{K} ({\bf v}_h,{\bf v}_h) \lesssim C(h) \; S^{K}({\bf v}_h,{\bf v}_h) \quad \forall {\bf v}_h \in {\bf V}_{h|K},
    \end{equation}
    where 
    \begin{equation*}
        C(h):= \max_{K \in \mathcal{T}_h} \log(1+\frac{h_K}{h_{m(K)}}),     
    \end{equation*}
with $h_{m(K)}$ defined in \eqref{eq:hmK}.

If the stronger assumption (A3) holds, then clearly $C(h) \simeq 1$.
\end{proposition}
\begin{proof}
We omit the proof of estimate \eqref{eq:coerc_2} since it uses arguments similar to Lemma 6.6 and Proposition 4.3 in \cite{BLR:2017}. Hence, let us consider estimate \eqref{eq:coerc_1} and choose ${\bf v}_h \in {\bf V}_{h|K}$ such that $\Pi^{\varepsilon,K}_2 {\bf v}_h = {\bf 0}$. Using a vectorial version of Lemma \ref{lm:fourier} and an inverse estimate for 1D polynomials, for $\varepsilon>0$ we get
\begin{equation}\label{eq:coerc_1.1}
|| {\bf v}_h ||_{\infty,\partial K} \lesssim  \varepsilon^{-1/2} || {\bf v}_h ||_{1/2+\varepsilon,\partial K}\lesssim \varepsilon^{-1/2} h_{m(K)}^{-\varepsilon}|| {\bf v}_h ||_{1/2,\partial K}   . 
\end{equation}
Choosing $\varepsilon = \left[\log(1+\frac{h_K}{h_{m(K)}})\right]^{-1}$, from \eqref{eq:coerc_1.1} we obtain
\begin{equation}\label{eq:coerc_1.2}
|| {\bf v}_h ||_{\infty,\partial K} \lesssim \left[\log(1+\frac{h_K}{h_{m(K)}})\right]^{1/2} || {\bf v}_h ||_{1/2,\partial K}   . 
\end{equation}
We now observe that it holds (cf. also \eqref{eq:divDoFs} and \eqref{eq:dofi_dofi}):
\begin{equation}\label{eq:coerc_1.3}
 S^{K}({\bf v}_h,{\bf v}_h) \lesssim || {\bf v}_h ||_{\infty,\partial K}^2 + || {\rm div}\,  {\bf v}_h ||_{0,K}^2 
 \lesssim  || {\bf v}_h ||_{\infty,\partial K}^2 + | {\bf v}_h |_{1,K}^2 . 
\end{equation}
Since $\Pi^{\varepsilon,K}_2 {\bf v}_h = {\bf 0}$, the function ${\bf v}_h $ has zero mean value on the boundary $\partial K$. Therefore, using also a trace estimate, it holds:
\begin{equation}\label{eq:coerc_1.4}
| {\bf v}_h |_{1,K}^2\lesssim a^{K} ({\bf v}_h,{\bf v}_h) \qquad \text{and} \qquad 
|| {\bf v}_h ||_{1/2,\partial K}^2 \lesssim | {\bf v}_h |_{1/2,\partial K}^2 \lesssim | {\bf v}_h |_{1,K}^2\lesssim a^{K} ({\bf v}_h,{\bf v}_h)  . 
\end{equation}
Taking advantage of \eqref{eq:coerc_1.2} and \eqref{eq:coerc_1.4}, from estimate \eqref{eq:coerc_1.3} we get
\begin{equation}\label{eq:coerc_1.5}
 S^{K}({\bf v}_h,{\bf v}_h) \lesssim \log(1+\frac{h_K}{h_{m(K)}}) a^{K} ({\bf v}_h,{\bf v}_h)  . 
\end{equation}
Then estimate \eqref{eq:coerc_1} immediately follows. 
\end{proof}

The next result follows immediately, showing the behavior of the stability constants.
\begin{corollary}
    Let assumptions (A1) and (A2) hold true. Then
    \begin{equation*}
        \alpha_h^{\star} \simeq \max_{K \in \mathcal{T}_h} \log(1+\frac{h_K}{h_{m(K)}}), \qquad \alpha_{\star,h} \simeq \frac{1}{\max_{K \in \mathcal{T}_h} \log(1+\frac{h_K}{h_{m(K)}})},
    \end{equation*}
    where the hidden constants are independent of the mesh size. If the stronger assumption (A3) holds, then clearly $1\lesssim\alpha_{\star,h} \lesssim \alpha_h^{\star} \lesssim 1 $.
\end{corollary}

\subsubsection{Evaluating the inf-sup constant \texorpdfstring{$\beta_h$}{βh} }\label{sssec:infsup-constants}
Let us introduce the following finite-dimensional subspace ${\bf V}_{h|K}^{L} \subset {\bf V}_{h|K}$. First, 
on each element $K \in \mathcal{T}_h$, consider the problem: find ${\bf v}$ and $s$ such that
\begin{equation} \label{eq:Stokes_element_system_reduced}
\begin{cases}
    -{\bf \Delta v} - \nabla s = {\bf{0}} \quad \text{in } K, \\
    \text{div }{\bf v} \in \Pol_0(K), \\
    {\bf v}_{|\partial K} \in [C^0(\partial K)]^2, \quad {\bf v}_{|e} \in [\Pol_2(e)]^2 \text{ for all } e \in \partial K,
\end{cases}
\end{equation}
with all the equations to be intended in a distributional sense. Then, we set
\begin{equation}\label{eq:VLhK}
{\bf V}_{h|K}^{L}:= \{ {\bf v} \in {\bf V}_{h|K} :
{\bf v} \text{ satisfies  (\ref{eq:Stokes_element_system_reduced}) for some $s \in L^2(K)$ } \},
\end{equation}
i.e. ${\bf V}_{h|K}^{L}$ is the subspace of ${\bf V}_{h|K}$ of displacements, whose divergence is constant, instead of linear.
We observe that, due to the compatibility conditions, the constant divergence is fully determined by the boundary data. The following lemma is a consequence of the well-posedness of 
Problem \eqref{eq:Stokes_element_system_reduced} and the mesh assumption (A1).  
\begin{lemma} \label{lemma:stability_Stokes_semi_norm}
    Under assumption (A1), for any ${\bf v} \in {\bf V}_{h|K}^{L}$ it holds
    \begin{equation*}
        \norm{{\bf v}}_{1,K} \lesssim \norm{{\bf v}}_{1/2,\partial K} \quad \text{and} \quad\abs{{\bf v}}_{1,K} \lesssim \abs{{\bf v}}_{1/2,\partial K},
    \end{equation*}
    where the hidden constant is independent of $K$.
\end{lemma}

We also introduce the global $H^1$-conforming space
\begin{equation}\label{eq:VWLh}
\begin{aligned}
    {\bf W}_{h}^{L}&:=\{ {\bf v} \in {\bf W}_h : {\bf v}_{|K} \in {\bf V}_{h|K}^{L} \quad \forall K \in \mathcal{T}_h \} \subset {\bf W}_h.
\end{aligned}    
\end{equation}
Since (\ref{eq:continuous_inf_sup}) holds, we prove the discrete inf-sup condition using Fortin's trick (see \cite{Boffi_Brezzi_Fortin}). 
The construction of the Fortin operator can be divided into three different steps:
\begin{enumerate}
    \item Construction of $\bar\Pi: {\bf W} \to {\bf W}_h^{L}$ such that
    \begin{equation*}
        \begin{cases}
            b(\bar\Pi{\bf v}-{\bf v},\bar{q}_h)=0 \quad \forall {\bf v}\in {\bf W}, \;\forall \bar{q}_h \text{ piecewise constant in } \mathcal{T}_h, \\
            \norm{\bar{\Pi}{\bf v}}_{{1,\Omega}} \leq C_{\bar\Pi}(h) \norm{{\bf v}}_{{1,\Omega}} \quad \forall {\bf v}\in {\bf W}.
        \end{cases}
    \end{equation*}
    \item Construction of $\widetilde{\Pi}:{\bf W} \to {\bf W}_h$ such that
    \begin{equation*}
        \begin{cases}
            b(\widetilde\Pi{\bf v},{q}_h)=b({\bf v}-\bar\Pi{\bf v},{q}_h) \quad \forall {\bf v}\in {\bf W}, \forall {q}_h \in Q_h, \\
            \norm{\widetilde{\Pi}{\bf v}}_{{1,\Omega}} \leq C \norm{{\bf v}-\bar\Pi{\bf v}}_{{1,\Omega}} \quad \forall {\bf v}\in {\bf W}.
        \end{cases}
    \end{equation*}
    \item Construction of the Fortin operator $\Pi: {\bf W} \to {\bf W}_h$ given by
    \begin{equation*}
        \Pi{\bf v}:=(\bar{\Pi} + \widetilde{\Pi}) {\bf v}.
    \end{equation*}
\end{enumerate}
Here, we detail only the first step. For details on the second step, see \cite{BLV:2017}. \\
\par
The construction of the operator $\bar\Pi$ requires the introduction of additional notations. We recall that $\mathcal{T}_h$ is the polygonal decomposition of $\Omega=\Omega^1\cup\Omega^2$ stemming from gluing the two meshes $\mathcal{T}_h^1$ and $\mathcal{T}_h^2$ of $\Omega^1$ and $\Omega^2$, respectively. 
We now introduce a quasi-uniform sub-triangulation of the polygonal decomposition $\mathcal{T}_h$, denoted by $\mathcal{T}_{h_m}$, with mesh size $h_m$. We set also ${\bf S}^{1,0}_{h_m}$ to be the linear finite element space over the triangulation $\mathcal{T}_{h_m}$, that is
\begin{equation*}
    {\bf S}^{1,0}_{h_m}:=\{ {\bf v} \in {\bf W} \cap [C^0(\overline\Omega)]^2 : {\bf v}_{|T} \in [\Pol_1(T)]^2 \quad \forall T \in \mathcal{T}_{h_m} \}.
\end{equation*}
Let us introduce the following operators:
\begin{itemize}
    \item the lowest degree Scott-Zhang interpolation operator (see \cite{ScottZhang} or \cite{brenner-scott:book}, for example), relative to the sub-triangulation $\mathcal{T}_{h_m}$
    \begin{equation*}
        \pi_{SZ}:{\bf W}\to {\bf S}^{1,0}_{h_m}.
    \end{equation*}
    \item the  linear Lagrange VEM interpolation operator relative to the polygonal mesh $\mathcal{T}_h$
    \begin{equation*}
        \pi_L: {\bf W} \cap [C^0(\overline\Omega)]^2\to {\bf W}_h^{L}.
    \end{equation*}
    \item the \say{edge bubble} operator
    \begin{equation}\label{eq:pi_bubble_0}
    \Pi_b: {\bf W}\to {\bf W}_h^{L}.
    \end{equation}
    This operator is defined by gluing local contributions $\Pi_{b,K}$, i.e. $(\Pi_b {\bf w})_{|K}= \Pi_{b,K}{\bf w}_{|K}$ for all ${\bf w}\in{\bf W}$. For ${\bf v} \in [H^1(K)]^2$, $\Pi_{b,K}{\bf v}\in{\bf V}_{h|K}^{L}$ is determined by requiring (cf. \eqref{eq:Stokes_element_system_reduced} and \eqref{eq:VLhK}): 
\begin{equation}\label{eq:pi_bubble}
    \begin{cases}
        \Pi_{b,K}{\bf v} \; (\mathcal{M}_e)=\frac{3}{2} \frac{1}{\abs{e}} \int_e {\bf v} \; d\Gamma, \quad \forall \text{ edge $e$ of $K$, with midpoint $\mathcal{M}_e$},\\
        \Pi_{b,K}{\bf v}\;(\mathcal{V}) = {\bf 0} \quad \forall \text{ vertex $\mathcal{V}$ of $K$}.
    \end{cases}
\end{equation}
We remark that the requirement above implies that $\int_e {\bf w}= \int_e \Pi_{b}{\bf w}$, for all ${\bf w}\in{\bf W}$ and for all edge $e$.
\end{itemize}
We are now ready to introduce the operator $\bar{\Pi}$. We set
\begin{equation}\label{eq:pi_bar_def}
    \bar{\Pi}{\bf v} := \pi_L \pi_{SZ} {\bf v} + \Pi_{b}(I-\pi_L \pi_{SZ}){\bf v} \quad \forall {\bf v} \in {\bf W}.
\end{equation}
We notice that $\bar{\Pi}$ is well-defined and $\bar{\Pi}{\bf v} \in {\bf W}_h^{L}$ for all ${\bf v} \in {\bf W}$.

The following two Lemmata are useful to study the properties of the operator $\bar\Pi$.
\begin{lemma}\label{lm:local_stability_bubble_operator}
    Let $\Pi_{b}$ defined as in \eqref{eq:pi_bubble_0} and \eqref{eq:pi_bubble}. Suppose that assumptions (A1) and (A2) are satisfied, then it holds
    \begin{equation}\label{eq:bubble-est}
        \begin{cases}
        \norm{\Pi_{b} {\bf v}}_{H^{1/2}_{00}(e)} \lesssim h_{e}^{-1/2} \norm{{\bf v}}_{0,e} \quad \forall e \in \partial K,\\
         \norm{\Pi_{b} {\bf v}}_{1/2,\partial K} \lesssim \sum_{e \in \partial K} h_e^{-1/2} \norm{{\bf v}}_{0,e},
        \end{cases}
    \end{equation}
    where the hidden constants are independent of the mesh size.
\end{lemma}
\begin{proof}
A standard 1D inverse estimate and \eqref{eq:pi_bubble} give:
\begin{equation}\label{eq:bubble-est2}
\norm{\Pi_{b} {\bf v}}_{H^{1/2}_{00}(e)}\lesssim h_e^{-1/2}\norm{\Pi_{b} {\bf v}}_{0,e}\lesssim h_e^{-1/2}\norm{ {\bf v}}_{0,e} , 
\end{equation}
i.e. the first estimate in \eqref{eq:bubble-est}.
Now, let us write $\Pi_{b} {\bf v} = \sum_{e\in\partial K} (\Pi_{b} {\bf v})\chi_e$, where $\chi_e:\partial K\to \R$ denotes the characteristic function of the side $e$. Due to assumptions (A1), it holds 
$$
\norm{\Pi_{b} {\bf v}}_{H^{1/2}_{00}(e)}\simeq \norm{\Pi_{b} ({\bf v})\chi_e}_{1/2,\partial K}. 
$$
Therefore, from assumption (A2) and inequality \eqref{eq:bubble-est2} we get
\begin{equation}
\norm{\Pi_{b} {\bf v}}_{1/2,\partial K} =  \norm{\sum_{e\in\partial K} (\Pi_{b} {\bf v})\chi_e}_{1/2,\partial K}\leq
\sum_{e\in\partial K} \norm{ (\Pi_{b} {\bf v})\chi_e}_{1/2,\partial K}\lesssim \sum_{e \in \partial K} h_e^{-1/2} \norm{{\bf v}}_{0,e} ,
\end{equation}
i.e. the second estimate in \eqref{eq:bubble-est}.
\end{proof}

Given $\omega\subset \R^2 $, we denote by $D_\omega$ the "diamond" of $\omega$ composed of triangles, \emph{i.e.}
    \begin{equation*}
        D_\omega:=\text{int} \Big(\bigcup \{\bar{T}' \in \mathcal{T}_{h_m} : \bar{T}' \cap \bar{\omega} \neq \emptyset \} \Big).
    \end{equation*}
\begin{lemma} \label{lemma:Clement_approx_boundary}
    Let ${\bf v} \in {\bf W}$ and consider $\pi_{SZ} {\bf v} \in {\bf S}^{1,0}_{h_m}$. Let $T \in \mathcal{T}_{h_m}$ and $l \in \partial T$ be one of its edges. Then, it holds
    \begin{equation*}
        \norm{(I-\pi_{SZ}){\bf v}}_{0,l} \lesssim h_m^{1/2} \abs{{\bf v}}_{1,D_T}.
    \end{equation*}
\end{lemma}
\begin{proof}
    Due to scaled trace inequality and the quasi-uniformity of the triangulation $\mathcal{T}_{h_m}$, we have
    \begin{equation*}
        \norm{(I-\pi_{SZ}){\bf v}}_{0,l} \lesssim h_{m}^{-1/2} \norm{(I-\pi_{SZ}){\bf v}}_{0,T} +h_m^{1/2}  \abs{(I-\pi_{SZ}){\bf v}}_{1,T}.
    \end{equation*}
    By approximation and stability properties of the Scott-Zhang operator, we get
    \begin{equation*}
        \norm{(I-\pi_{SZ}){\bf v}}_{0,l} \lesssim h_{m}^{-1/2} h_m \abs{{\bf v}}_{1,D_T} +h_m^{1/2}  \abs{{\bf v}}_{1,D_T} \lesssim h_m^{1/2}\abs{{\bf v}}_{1,D_T}.
    \end{equation*}
\end{proof}

We can now prove the following Proposition.
\begin{proposition} \label{prop:Fortin_P2_P0_part}
    Under assumptions (A1) and (A2), the operator $\bar{\Pi}: {\bf W} \to {\bf W}_h^{L}$ satisfies
    \begin{equation*}
        \begin{cases}
            b(\bar\Pi{\bf v}-{\bf v},\bar{q}_h)=0 \quad \forall {\bf v}\in {\bf W}, \;\forall \bar{q}_h \text{ \rm piecewise constant in } \mathcal{T}_h, \\
            \norm{\bar{\Pi}{\bf v}}_{1,\Omega} \lesssim C_{\bar\Pi}(h) \norm{{\bf v}}_{1,\Omega} \quad \forall {\bf v}\in {\bf W},
        \end{cases}
    \end{equation*}
    where $C_{\bar\Pi}(h)=\max_{K \in \mathcal{T}_h}\log(1+\frac{h_K}{h_{m(K)}})$.
\end{proposition}
\begin{proof}
    The first property is trivial (see \cite{Boffi_Brezzi_Fortin}). We now show the $H^1$-continuity estimate.
    Since ${\bf W}_h^{L}\subset {\bf W}$, by Poincaré inequality we have
    \begin{equation*}
        \norm{\bar{\Pi}{\bf v}}_{1,\Omega} \lesssim \abs{\bar{\Pi}{\bf v}}_{1,\Omega}.
    \end{equation*}
    Thus, it is sufficient to show that
    \begin{equation}\label{eq:suff-estimate}
        \abs{\bar{\Pi}{\bf v}}_{1,\Omega} \lesssim C_{\bar\Pi}(h) \norm{{\bf v}}_{1,\Omega}.
    \end{equation}
    Let us begin with a local estimate. For all $K \in \mathcal{T}_h$ we have (see \eqref{eq:pi_bar_def})
    \begin{align*}
        \abs{\bar{\Pi}{\bf v}}_{1,K} &= \abs{\pi_L \pi_{SZ} {\bf v} + \Pi_{b}(I-\pi_L \pi_{SZ}){\bf v}}_{1,K} \\
        &\leq \abs{\pi_L \pi_{SZ} {\bf v}}_{1,K} + \abs{\Pi_{b}(I-\pi_L \pi_{SZ}){\bf v}}_{1,K} \quad(\text{use Lemma } \ref{lemma:stability_Stokes_semi_norm}) \\
        & \lesssim \abs{\pi_L \pi_{SZ} {\bf v} }_{1/2,\partial K} + \abs{\Pi_{b}(I-\pi_L \pi_{SZ}){\bf v}}_{1/2,\partial K} = I + II.
    \end{align*}
    Now, we are going to bound each term separately. \\
    \par
    \emph{Estimate of I}. Let us write $(\pi_L \pi_{SZ} {\bf v})_{|\partial K}= \sum_{j}  \pi_{SZ}{\bf v}({\cal N}_j) \varphi_j$, where the ${\cal N}_j$'s denote the nodes on $\partial K$ and the $\varphi_j$'s are the usual corresponding Lagrange basis functions. We now notice that it holds, cf. \cite{BLR:2017}:
    \begin{equation}\label{eq:basis-log}
        \abs{\varphi_j}_{1/2,\partial K} \lesssim 
        \left[\log(1+\frac{h_K}{h_{m(K)}})\right]^{1/2}. 
    \end{equation}
    To continue, for all $\boldsymbol{\alpha} \in \R^2$, recalling that the $\varphi_j$'s are a partition of unity and $\pi_L \boldsymbol{\alpha} =\boldsymbol{\alpha}$, we have
    \begin{align*}
        \abs{\pi_L \pi_{SZ} {\bf v} }_{1/2,\partial K} &=  \abs{\pi_L (\pi_{SZ} {\bf v} - \boldsymbol{\alpha)}}_{1/2,\partial K}=
        \big|\sum_j(\pi_{SZ}{\bf v}({\cal N}_j)-\boldsymbol{\alpha}) \varphi_j\big|_{1/2,\partial K} \\
        & \leq \big(\sum_{j} \abs{\varphi_j}_{1/2,\partial K} \big) \norm{\pi_{SZ} {\bf v} - \boldsymbol{\alpha}}_{\infty,\partial K} \quad \text{(cf. estimates \eqref{eq:coerc_1.2} and \eqref{eq:basis-log})} \\
        & \lesssim \log(1+\frac{h_K}{h_{m(K)}}) \norm{\pi_{SZ} {\bf v} - \boldsymbol{\alpha}}_{1/2,\partial K} \quad (\text{use def of fractional norm}) \\
        & \lesssim \log(1+\frac{h_K}{h_{m(K)}}) ( h_{K}^{-1/2}\norm{\pi_{SZ} {\bf v} - \boldsymbol{\alpha}}_{0,\partial K}+ \abs{\pi_{SZ} {\bf v} - \boldsymbol{\alpha}}_{1/2,\partial K} ).
    \end{align*}
    Now, if we take $\boldsymbol{\alpha}=\frac{1}{\abs{\partial K}} \int_{\partial K} {\bf v} \; dx $ and we apply trace inequality, we get
    \begin{equation} \label{eq:estimate_I_Fortin}
        \abs{\pi_L \pi_{SZ} {\bf v} }_{1/2,\partial K} \lesssim \log(1+\frac{h_K}{h_{m(K)}}) \abs{\pi_{SZ} {\bf v}}_{1/2,\partial K} \lesssim \log(1+\frac{h_K}{h_{m(K)}}) \abs{\pi_{SZ} {\bf v}}_{1,K}.
    \end{equation}

    \emph{Estimate of II}. 
    \begin{equation}\label{eq:Pi_b_est}
    \begin{aligned}
        \abs{\Pi_{b}(I-\pi_L \pi_{SZ}){\bf v}}_{1/2,\partial K} &\leq \norm{\Pi_{b}(I-\pi_L \pi_{SZ}){\bf v}}_{1/2,\partial K} \quad (\text{use Lemma } \ref{lm:local_stability_bubble_operator}) \\
        & \lesssim \sum_{e \in \partial K} h_e^{-1/2} \norm{(I-\pi_L \pi_{SZ}){\bf v}}_{0,e} \quad (\pm \pi_{SZ}{\bf v}) \\
        & \lesssim \sum_{e \in \partial K} h_e^{-1/2} (\norm{(I- \pi_{SZ}){\bf v}}_{0,e} + \norm{(I- \pi_L){\pi_{SZ}\bf v}}_{0,e}).
    \end{aligned}
    \end{equation}
    Now, we observe that on each edge $e \in \partial K$ there is a sub-decomposition induced by the trace of the finer triangulation $\mathcal{T}_{h_m}$, consisting of edges $l$ of smaller triangles. Thus, by Lemma \ref{lemma:Clement_approx_boundary} we obtain
    \begin{equation*}
        \norm{(I- \pi_{SZ}){\bf v}}_{0,e}^2 = \sum_{l \in e} \norm{(I- \pi_{SZ}){\bf v}}_{0,l}^2 \lesssim h_m \sum_{l \in e} \abs{{\bf v}}^2_{1,D_{T_l}},
    \end{equation*}
    where $T_l$ is the triangle contained in $K$ that has $l$ as one of its edges. Since $\mathcal{T}_{h_m}$ is quasi-uniform, we get
    \begin{equation}\label{eq:est_I-piSZ}
         \norm{(I- \pi_{SZ}){\bf v}}_{0,e} \lesssim h_m^{1/2} \abs{{\bf v}}_{1,D_e}.
    \end{equation}
    On the other hand, it holds
    \begin{equation}\label{eq:est_I-piL}
        \norm{(I- \pi_L){\pi_{SZ}\bf v}}_{0,e} \lesssim h_e^{1/2} \abs{\pi_{SZ}{\bf v}}_{1/2,e} \lesssim h_e^{1/2} \abs{\pi_{SZ}{\bf v}}_{1/2,\partial K} \lesssim h_e^{1/2} \abs{\pi_{SZ}{\bf v}}_{1,K}.
    \end{equation}
    Combining \eqref{eq:est_I-piSZ} and \eqref{eq:est_I-piL}, from estimate \eqref{eq:Pi_b_est} we thus have
    \begin{equation*}
        \abs{\Pi_{b}(I-\pi_L \pi_{SZ}){\bf v}}_{1/2,\partial K} \lesssim \sum_{e \in \partial K} h_e^{-1/2}( h_m^{1/2} \abs{{\bf v}}_{1,D_e} + h_e^{1/2} \abs{\pi_{SZ}{\bf v}}_{1,K}).
    \end{equation*}
    Observing that $h_m \leq h_e$, $\sum_{e \in \partial K} \abs{{\bf v}}_{1,D_e}^2 \lesssim \abs{{\bf v}}_{1,D_K}^2$ and using assumption (A2), we get
    \begin{equation} \label{eq:estimate_II_Fortin}
        \abs{\Pi_{b}(I-\pi_L \pi_{SZ}){\bf v}}_{1/2,\partial K} \lesssim \abs{{\bf v}}_{1,D_K} + \abs{\pi_{SZ}{\bf v}}_{1,K}.
    \end{equation}
    By estimates (\ref{eq:estimate_I_Fortin}) and (\ref{eq:estimate_II_Fortin}), and recalling definition \eqref{eq:pi_bar_def}, it follows
    \begin{equation*}
        \abs{\bar{\Pi}{\bf v}}_{1,K}^2 \lesssim \log^2(1+\frac{h_K}{h_{m(K)}}) \abs{\pi_{SZ} {\bf v}}_{1,K}^2 + \abs{{\bf v}}_{1,D_K}^2 \lesssim C_{\bar\Pi}(h)^2 \abs{\pi_{SZ} {\bf v}}_{1,K}^2 + \abs{{\bf v}}_{1,D_K}^2,
    \end{equation*}
    where $C_{\bar\Pi}(h)=\max_{K \in \mathcal{T}_h}\log(1+\frac{h_K}{h_{m(K)}})$. \\
    \par
    Now
    \begin{align*}
        \abs{\bar{\Pi}{\bf v}}_{1,\Omega}^2 &= \sum_{K \in \mathcal{T}_h} \abs{\bar{\Pi}{\bf v}}_{1,K}^2 \\
        & \lesssim C_{\bar\Pi}(h)^2 \sum_{K \in \mathcal{T}_h} \abs{\pi_{SZ} {\bf v}}_{1,K}^2 + \sum_{K \in \mathcal{T}_h} \abs{{\bf v}}_{1,D_K}^2 \\
        &\lesssim  C_{\bar\Pi}(h)^2 \abs{\pi_{SZ} {\bf v}}_{1,\Omega}^2 + \abs{{\bf v}}_{1,\Omega}^2 \quad (\text{use global stability of } \pi_{SZ})\\
        & \lesssim C_{\bar\Pi}(h)^2 \abs{ {\bf v}}_{1,\Omega}^2.
    \end{align*}
    Thus
    \begin{equation*}
        \abs{\bar{\Pi}{\bf v}}_{1,\Omega} \lesssim C_{\bar\Pi}(h) \abs{ {\bf v}}_{1,\Omega} \lesssim C_{\bar\Pi}(h) \norm{ {\bf v}}_{1,\Omega}.
    \end{equation*}
\end{proof}

\begin{corollary}
    Let assumptions (A1) and (A2) hold true. Given the discrete spaces ${\bf V}_h$ and ${\bf Q}_h$ defined in (\ref{eq:Stokes_like_global_spaces}) and the subspace ${\bf W}_h$ defined in (\ref{eq:Stokes_like_subspace_infsup}), there exists a positive constant $\beta_h$ such that
    \begin{equation*}
        \sup_{{\bf w}_h \in {\bf W}_h} \frac{b({\bf w}_h, {\bf q}_h)}{\norm{{\bf w}_h}_{{\bf V}}} \geq \beta_h \norm{{\bf q}_h}_{{\bf Q}/{\bf H}} \quad \forall {\bf q}_h \in {\bf Q}_h,
    \end{equation*}
    where
    \begin{equation}\label{eq:beta_h_est}
        \beta_h \simeq \frac{1}{\max_{K \in \mathcal{T}_h}\log(1+\frac{h_K}{h_{m(K)}})},
    \end{equation}
    and the hidden constant is independent of the mesh size. If the stronger assumption (A3) holds, then clearly $\beta_h \simeq 1$.
\end{corollary}
\begin{proof}
    It follows from the continuous inf-sup condition (\ref{eq:continuous_inf_sup}) and Proposition \ref{prop:Fortin_P2_P0_part}.
\end{proof}
\begin{remark}\label{rm:small_on_Wh}
By a careful inspection of the proof of Proposition \ref{prop:Fortin_P2_P0_part}, one can realize that if \say{small edges} occur \emph{only on $\Gamma_C$ (i.e. only as a result of the node insertion procedure)}, then it holds $\beta_h \simeq 1$. In fact, if ${\bf v}\in {\bf W}$, then ${\bf v}_{|\Gamma_C}={\bf 0}$, cf. \eqref{eq:W-def}. As a consequence $(\bar{\Pi}{\bf v})_{|\Gamma_C}={\bf 0}$ and no effect of \say{small edges} enters into play in the \emph{inf-sup} condition \eqref{eq:discrete_inf_sup}.  
However, in typical time-dependent or large-deformation problems, the contact interface $\Gamma_C$ is not known a priori and evolves over time. Consequently, \say{small edges} may arise outside the active contact region unless the mesh is frequently updated, through de-refinement, whenever contact ceases to occur in a specific area.
That is the reason why we have preferred to treat the general case, thus avoiding any hypotheses of the localization of \say{small edges}. 
\end{remark}

From the analysis above,  we infer the well-posedness of the discrete problem and get an estimate of the quantity $C_{fin}(h)$ appearing in Theorem \ref{teo:abstract_convergence_result}. In fact, we have the following Proposition.
\begin{proposition}\label{pr:cfin}
    Let assumptions (A1)-(A2) hold true. Given the discrete spaces ${\bf V}_h$ and ${\bf Q}_h$ defined in (\ref{eq:Stokes_like_global_spaces}), the closed and convex subset ${\bf K}_h$ defined in (\ref{eq:Stokes_like_convex_subset}), the subspace ${\bf W}_h$ defined in (\ref{eq:Stokes_like_subspace_infsup}) and the stabilizing form (\ref{eq:dofi_dofi}), the discrete problem (\ref{eq:Discrete_problem}) with the choice detailed in section \ref{sss:second} has a unique solution $({\bf u}_h,{\bf p}_h) \in {\bf K}_h \times {\bf Q}_h$. Furthermore, for the quantity $C_{fin}(h)$ defined in \eqref{eq:abstract_constant} it holds
    \begin{equation*}
        C_{fin}(h) \simeq \big(\max_{K \in \mathcal{T}_h}\log(1+\frac{h_K}{h_{m(K)}}) \big)^{10},
    \end{equation*}
    where the hidden constant is independent of the mesh size.
\end{proposition}
\subsection{The convergence result}
Once an estimation for $C_{fin}(h)$ has been established, cf. Proposition \ref{pr:cfin}, an error bound for the method can be obtained by estimating the interpolation and approximation errors, cf. Theorem \ref{teo:abstract_convergence_result}. 
Therefore, we begin by introducing a suitable interpolation operator.
\begin{proposition}\label{pr:interpol}
     Let assumption (A1) hold. Let ${\bf u} \in {\bf V} \cap \big([H^\nu(\Omega^1)]^2 \times [H^\nu(\Omega^2)]^2 \big)$ with $1<\nu\leq 3$. 
   There exists ${\bf u}_I \in {\bf V}_h$ such that
    \begin{equation} \label{eq:interpolant_properties}
        \begin{cases}
            \norm{{\bf u}-{\bf u}_I}_{0,K} + h_K \abs{{\bf u}-{\bf u}_I}_{1,K} \lesssim h_K^{\nu} \abs{{\bf u}}_{H^\nu(K)} \quad \forall K \in \mathcal{T}_h, \\
            {\bf u} \in {\bf K} \implies {\bf u}_I \in {\bf K}_h, \\
            \int_{\Gamma_C}\llbracket u_n \rrbracket \; d\Gamma= \int_{\Gamma_C} \llbracket  {u_I}_n \rrbracket \; d\Gamma.
        \end{cases}
    \end{equation}
\end{proposition}
\begin{proof}
    The proof is based on a local construction. For each polygon $K \in \mathcal{T}_h$, let ${\bf u} \in [H^{\nu}(K)]^2$, with $1<\nu \leq 3$. We define ${\bf u}_I \in {\bf V}_h^K$ as follows
\begin{equation} \label{eq:definition_interpolant}
    \begin{cases}
        -\boldsymbol{\Delta} {\bf u}_I + \boldsymbol{\nabla} s = {\bf 0} & \text{in } K, \\
        \text{div }{\bf u}_I = \Pi^{0,K}_1 (\text{div }{\bf u}) & \text{in } K, \\
        {\bf u}_I = {\bf p}^\ast & \text{on } \partial K,
     \end{cases}
\end{equation}
where $s \in L^2(K)$ and ${\bf p}^\ast \in [C^0(\partial K)]^2$, ${\bf p}^\ast_{|e} \in [\Pol_2(e)]^2$ for all $e \in \partial K$ such that 
\begin{equation*}
    {\bf p}^\ast_{|e} = {\bf p}_1 + {\bf p}_2 \quad \forall e \in \partial K,
\end{equation*}
with
\begin{equation*}
\begin{aligned}
    &{\bf p}_1 \in [\Pol_1(e)]^2, \quad {\bf p}_1 = \text{linear Lagrange interpolant of } {\bf u} \text{ on the edge } e, \\
    &{\bf p}_2 \in [\Pol_2(e)]^2, \quad {\bf p}_2 = \text{bubble function on \emph{e} such that } \int_e {\bf p}_2 \; d \Gamma= \int_e {\bf u} - {\bf p}_1 \; d\Gamma.
\end{aligned}
\end{equation*}
We observe that (\ref{eq:definition_interpolant}) is well-defined, as we have the compatibility condition on data.\\
Setting $\Pi^{0,K}_0(\boldsymbol{\Delta}{\bf u})= \boldsymbol{\nabla} \tilde{s}$, we can write
\begin{equation*}
    \begin{cases}
        -\boldsymbol{\Delta} ({\bf u}_I - {\bf u})+ \boldsymbol{\nabla} (s - \tilde{s}) = \boldsymbol{\Delta }{\bf u}-\Pi^{0,K}_0(\boldsymbol{\Delta}{\bf u}) & \text{in } K, \\
        \text{div }({\bf u}_I - {\bf u}) = \Pi^{0,K}_1 (\text{div }{\bf u}) - \text{div } {\bf u}& \text{in } K, \\
        {\bf u}_I -{\bf u}= {\bf p}^\ast-{\bf u} & \text{on } \partial K.
     \end{cases}
\end{equation*}
From the stability of Stokes problem, we have
\begin{equation*}
    \abs{{\bf u}_I -{\bf u}}_{1,K} \lesssim \norm{\boldsymbol{\Delta }{\bf u}-\Pi^{0,K}_0(\boldsymbol{\Delta}{\bf u})}_{H^1(K)'} + \norm{\Pi^{0,K}_1 (\text{div }{\bf u}) - \text{div } {\bf u}}_{0,K} + \norm{{\bf p}^\ast-{\bf u}}_{1/2,\partial K},
\end{equation*}
where, due to assumption (A1), the hidden constant does not depend on $K$.
By standard arguments, it follows
\begin{equation*}
    \abs{{\bf u}_I -{\bf u}}_{1,K} \lesssim h_K^{\nu-1} \abs{\bf u}_{\nu,K},
\end{equation*}
from which we infer the first relation of (\ref{eq:interpolant_properties}). The second and third relations in (\ref{eq:interpolant_properties}) follow easily from the definition (\ref{eq:definition_interpolant}).
\end{proof}
Let us now consider the finite-dimensional space
\begin{equation*}
    M_h :=\{ v \in L^2(\Gamma_C) : v_{|e^c_l}  \in \Pol_0(e^c_l) \quad \forall l=0,\dots, l^\star-1\}.
\end{equation*}
\par
Given $f \in H^s(\Gamma_C)$ with $s \in \R$, $0 \le s \le1$, we define $\mathcal{P}^{M_h}(f) \in M_h$ as the piecewise local projection of $f$ onto constants, that is 
\begin{equation} \label{eq:definition_approx_contact_constants}
    \mathcal{P}^{M_h}(f)_{|e^c_l}:= \Pi^{0,e^c_l}_0 f \quad \forall l = 0, \dots, l^\star-1.
\end{equation}
We also set
\begin{equation*}
    h_{\Gamma_C}:=\max_{l=0,\cdots, l^\star-1} \abs{e^c_l}.
\end{equation*}

The following two lemmata can be established by slight adaptations of the results in \cite{Belhachmi_BenBelgacem} and \cite{Belgacem_Renard_Slimane}. 

\begin{lemma}\label{lm:boundary1}
    Let ${\bf u}^i \in [H^\nu(\Omega^i)]^2$ and $p^i \in H^{\nu-1}(\Omega^i)$ with $2 < \nu \leq 5/2$. Let $\psi_h:= \mathcal{P}^{M_h}(\sigma_{n})$. Then, it holds
    \begin{equation*}
        \norm{\sigma_n - \psi_h}_{H^{1/2}_{00}(\Gamma_C)'} \leq C h_{\Gamma_C}^{\nu -1} \big(\norm{{\bf u}}_{\nu,\Omega} + \norm{\bf p}_{\nu-1,\Omega} \big),
    \end{equation*}
    where $C>0$ is independent of $h_{\Gamma_C}$.
\end{lemma}

\begin{lemma}\label{lm:boundary2}
    Let ${\bf u}^i \in [H^\nu(\Omega^i)]^2$ and $p^i \in H^{\nu-1}(\Omega^i)$ with $2 < \nu \leq 5/2$. Let $\psi_h:= \mathcal{P}^{M_h}(\sigma_{n})$, $g_0 \in H^{\nu-1/2}(\Gamma_C)$. Then, it holds
    \begin{equation*}
        \int_{\Gamma_C} \psi_h (\llbracket u_{I,n} \rrbracket - \llbracket u_{h,n} \rrbracket) \; d\Gamma \leq C  h_{\Gamma_C}^{2\nu -2} \big( \norm{{\bf u}}^2_{\nu,\Omega} + \norm{{\bf p}}^2_{\nu-1,\Omega} + \norm{g_0}^2_{\nu-1/2,\Gamma_C} \big),
    \end{equation*}
    where $C>0$ is independent of $h_{\Gamma_C}$.
\end{lemma}

Now, we recall the following approximation result for polynomials in star-shaped domains, see for instance \cite{brenner-scott:book}.
\begin{proposition}\label{pr:polyappr}
    Let assumption (A1) hold. Let $K\in \mathcal{T}_h$ and $k \in \N$, $k \geq 1$. For any ${\bf u} \in [H^{s+1}(K)]^2$ with $s \in \R$, $0 \leq s \leq k$, there exists ${\bf u}_\pi \in [\Pol_k(K)]^2$ such that
    \begin{equation} \label{eq:Dupont_Scott_estimate}
        \norm{{\bf u}-{\bf u}_\pi}_{0,K} + h_K \abs{{\bf u}-{\bf u}_\pi}_{1,K} \leq C h_{K}^{s+1} |{\bf u}|_{s+1,K}, 
    \end{equation}
    where $C>0$ depends only on $k$ and $\gamma$.
\end{proposition}

Finally, based on Theorem \ref{teo:abstract_convergence_result} and on Proposition \ref{pr:cfin}, we now prove the following convergence theorem.
\begin{theorem}\label{th:converg}
    Let assumptions (A1)-(A2) hold. Let $({\bf u},{\bf p}) \in {\bf K}\times {\bf Q}$ and $({\bf u}_h,{\bf p}_h) \in {\bf K}_h \times {\bf Q}_h$ be the solutions of \eqref{eq:contact_problem} and (\ref{eq:Discrete_problem}), respectively. Assume ${\bf u}^i \in [H^\nu(\Omega^i)]^2$, $p^i \in H^{\nu-1}(\Omega^i)$, $g_0 \in H^{\nu-1/2}(\Gamma_C)$ and ${\bf f}^i \in[H^{\nu-2}(\Omega))]^2$ with $2<\nu \leq 5/2$, $i=1,2$. It holds:
    \begin{equation}\label{eq:err_est_quad}
        \norm{{\bf u}-{\bf u}_h}_{{\bf V}} + \norm{{\bf p}-{\bf p}_h}_{{\bf Q}/{\bf H}} \lesssim C(h) h^{\nu-1} \big(\norm{{\bf u} }_{\nu,\Omega} + \norm{{\bf p}}_{\nu-1,\Omega} + \norm{g_0}_{\nu-1/2,\Gamma_C}+  \norm{{\bf f}}_{\nu-2,\Omega}\big),
    \end{equation}
    where $C(h)=\big(\max_{K \in \mathcal{T}_h}\log(1+\frac{h_K}{h_{m(K)}})\big)^5$ and the hidden constant is independent of the mesh size. If the stronger assumption (A3) holds, then clearly $C(h) \simeq 1$. 
\end{theorem}
\begin{proof}
By noting that it holds, cf. \eqref{eq:fh1} and \eqref{eq:fh2}:
\begin{equation}
\norm{F-F_h}_{\bf V'}\leq C\, h^{\nu-1}  \norm{{\bf f}}_{\nu-2,\Omega},
\end{equation}
the proof immediately follows from Propositions \ref{pr:interpol} and \ref{pr:polyappr}, and Lemmata \ref{lm:boundary1} and \ref{lm:boundary2}.    
\end{proof}

\begin{remark}\label{rm:firstorder}
For the first order element detailed in section \ref{sss:first} an analogous convergence result can be obtained. More precisely, under the mesh assumptions (A1) and (A2), one can show what follows (cf. also \cite{Belgacem_Renard_Slimane} and \cite{Belhachmi_BenBelgacem}).
\begin{itemize}
\item For $1<\nu\leq 3/2$ the error estimate \eqref{eq:err_est_quad} holds true.
\item For $3/2<\nu < 2$, in order to get estimate \eqref{eq:err_est_quad} one needs to additionally assume that on $\Gamma_C$ there is only a finite number of points where the constraint changes from binding to nonbinding. Such an assumption is typically met in practical situations.  
\item For the particular case $\nu=2$, one can show a slightly worse estimate:
\begin{equation}\label{eq:err_est_lin}
\norm{{\bf u}-{\bf u}_h}_{{\bf V}} + \norm{{\bf p}-{\bf p}_h}_{{\bf Q}/{\bf H}} 
\lesssim |\log h|^{1/4} C(h) h \big(\norm{{\bf u} }_{2,\Omega} + \norm{{\bf p}}_{1,\Omega} 
+ \norm{g_0}_{3/2,\Gamma_C}+  \norm{{\bf f}}_{0,\Omega}\big),
\end{equation}
where we recall that $C(h)=\big(\max_{K \in \mathcal{T}_h}\log(1+\frac{h_K}{h_{m(K)}})\big)^5$
and the hidden constant is independent of the mesh size. Again, if the stronger assumption (A3) holds, then $C(h)\simeq 1$ and a logarithmic term still remains.
\end{itemize}

\begin{remark}\label{rm:log}
We remark that the logarithmic terms appearing in Theorem \ref{th:converg} and Remark \ref{rm:firstorder} are not excessively detrimental for the error estimates,
as they blow up slower than any negative power of the meshsize $h$. Indeed, these theoretical logarithmic degeneracies are not observed in the situations of practical interest. We also remark that similar logarithmic terms appear in other frameworks, e.g. in the Domain Decomposition methods, see \cite{ToselliWidlund2005}, for instance.
\end{remark}

\end{remark}
\section{Numerical tests}\label{sec:numer}
In this section we present some numerical experiments to verify the actual performance of the method. 

\subsection{Tests with fully available analytical solution}\label{ss:analytical}
We consider two square elastic bodies in their reference configuration, $\Omega^1=(0,1)\times (-1,0)$ and $\Omega^2=(0,1) \times (0,1)$, with no initial gap ($g_0 =0$). The contact boundary is $\Gamma_C:=[0,1]\times \{0\}$. The Lamé parameters are set to $\mu^1=\mu^2=1$ and $\lambda^1=\lambda^2=\lambda$; the value of $\lambda$ is varied to cover both the compressible and nearly incompressible regimes. As shown in Figure \ref{fig:reference_config_squares}, Dirichlet boundary conditions are imposed on the horizontal edges ($y=\pm 1$), while Neumann conditions are applied on the vertical edges ($x=0$ and $x=1$).
\begin{figure}[htbp]
  \centering
  \includegraphics[width=0.3\linewidth]{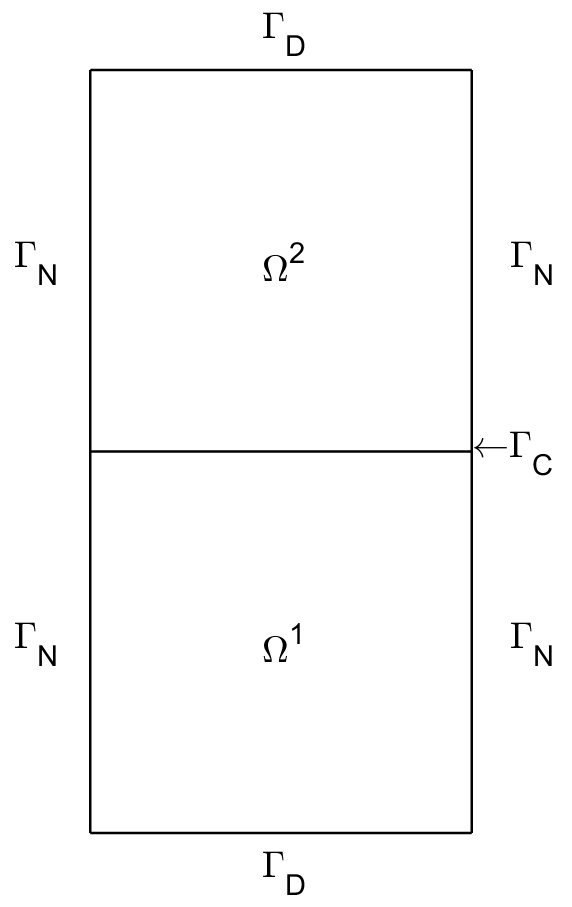}
  \caption{Reference configuration and boundary conditions.}
  \label{fig:reference_config_squares}
\end{figure}
The two domains are partitioned using different families of polygonal meshes:
\begin{itemize}
    \item $\{ \mathcal{Q}_h\}_h$: sequence of square meshes with $4$, $16$, $64$ and $256$ elements,
    \item $\{ \mathcal{H}_h\}_h$: sequence of hexagonal meshes with $4$, $16$, $64$ and $256$ elements,
    \item $\{ \mathcal{W}_h \}_h$: sequence of WEB-like meshes with $8$, $32$, $128$ and $512$ elements.
\end{itemize}
An example of the adopted meshes is shown in Figure \ref{fig:meshes}.
\begin{figure}[htbp]
  \centering
  \includegraphics[width=0.8\linewidth]{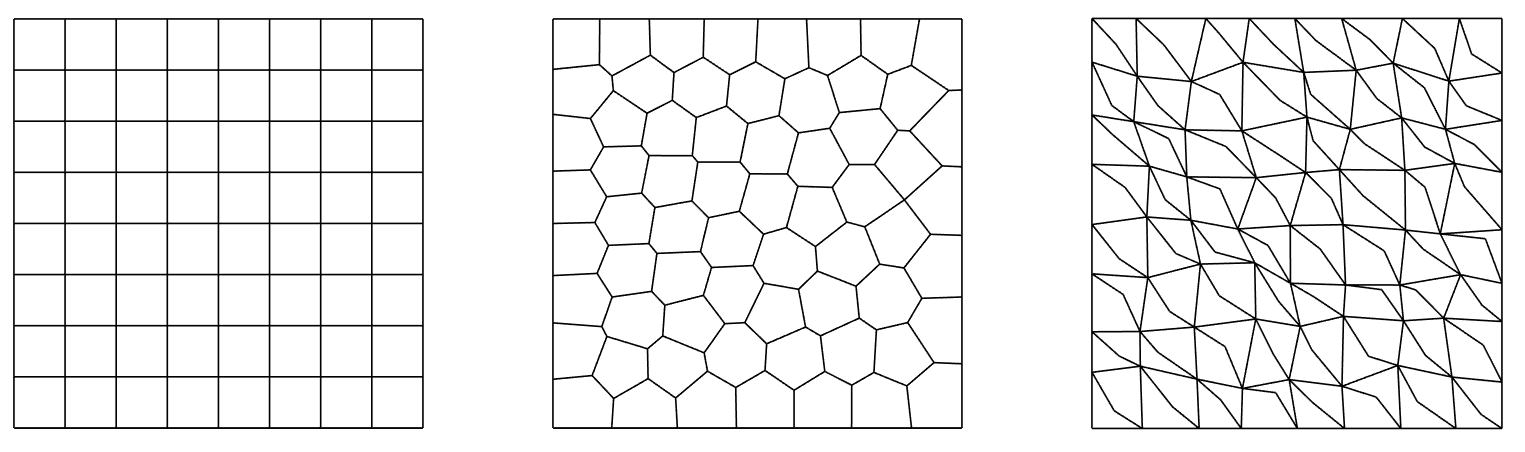}
  \caption{Example of polygonal meshes: $\mathcal{Q}_{64}$, $\mathcal{H}_{64}$, $\mathcal{W}_{128}$.}
  \label{fig:meshes}
\end{figure}
Throughout the tests we consider two different mesh configurations. In the first one, referred to as the \emph{Initial Matching} (I.M.) configuration, the two polygonal meshes are generated so that the nodes coincide along the contact interface, and the edge lengths are comparable to the maximum element diameter $h_{max}$. Starting from this setting, in the second configuration we randomly perturb the nodes of $\Omega^1$; as a result, the node insertion algorithm produces small edges on $\Gamma_C$, whose lengths range between $1\%$ and $2\%$ of the maximum element diameter $h_{max}$. We refer to this case as the \emph{Small Edges} (S.E.) configuration. 
It is worth recalling that the VEM solution ${\bf u}_h$ is not explicitly known point-wise inside the elements. As a consequence, the method error is not computable even if the exact solution $\bf u$ of the problem is available. As usual in the VEM framework, we compute the error by comparing the analytical solution $\bf{u}$ with a computable polynomial projection $\Pi^{\varepsilon,K} {\bf u}_h$ (specifically: $\Pi^{\varepsilon,K}_1 {\bf u}_h$ for the linear scheme of section \ref{sss:first}, $\Pi^{\varepsilon,K}_2 {\bf u}_h$ for the quadratic one of section \ref{sss:second}). The error quantities employed for the convergence analysis are
\begin{equation}
    \delta({\bf u}):= \Big( \sum_{K \in \mathcal{T}_h} \abs{{\bf u} - \Pi^{\varepsilon,K} ({\bf u}_h)}_{1,K}^2 \Big)^{1/2} \quad \text{and} \quad \delta({\bf p}):=\norm{{\bf p}-{\bf p}_h}_{\bf Q}.
\end{equation}
The solver we use for the variational inequalities \eqref{eq:Discrete_problem} is based on the Uzawa algorithm, see e.g. \cite{Allaire2007}.

\subsubsection{Patch test}\label{sss:patch}
The first example is a contact patch test. Dirichlet and Neumann boundary data, as well as the load terms ${\bf f}^i$ are chosen in such a way that the analytical solution is
\begin{equation}
    {\bf u}^i = \frac{1}{\lambda^i}\begin{pmatrix}
        0 \\ -(y+1)
    \end{pmatrix}, \quad
    p^i = \lambda^i \text{div }{\bf u}^i=-1, \quad 
    \text{for } i=1,2.
\end{equation}
Tests are performed on all families of decompositions and for both the methods detailed in sections \ref{sss:first} and \ref{sss:second}, but only for the coarser meshes. Results are summarized in Table \ref{tab:patch_test}. Looking at the results, we see that both our VEM methods pass the patch test, both in the compressible ($\lambda= 1$) and the nearly incompressible ($\lambda=10^3$ and $\lambda=10^8$) cases. We remark that the value $\lambda=10^8$ is exceptionally large and may not be significant for actual applications. Nevertheless, we include this case to demonstrate the robustness of the proposed schemes.

\begin{table}[htbp]
  \tiny
  \centering
  \begin{tabular}{lcccccccccc}
    \toprule
    \multicolumn{3}{c}{} & \multicolumn{2}{c}{Initial Matching (k=1)} & \multicolumn{2}{c}{Small Edges (k=1)} & \multicolumn{2}{c}{Initial Matching (k=2)} & \multicolumn{2}{c}{Small Edges (k=2)}\\
    \cmidrule(lr){4-5} \cmidrule(lr) {6-7} \cmidrule(lr){8-9} \cmidrule(lr) {10-11}
    $\lambda$ & Mesh & N Elms &$\delta({\bf u})$ & $\delta({\bf p})$ & $\delta({\bf u})$ &$\delta({\bf p})$&$\delta({\bf u})$ & $\delta({\bf p})$ & $\delta({\bf u})$ &$\delta({\bf p})$\\
    \midrule
    
    \multirow{6}{*}{$10$} & \multirow{2}{*}{$\mathcal{Q}_h$} & $4$ & $8.8$e$-16$ & $4.8$e$-16$ & $1.0$e$-15$ & $5.3$e$-16$ & $1.7$e$-15$ & $1.9$e$-15$ & $1.3$e$-15$ & $1.8$e$-15$\\
    & & $16$ & $7.2$e$-16$ & $6.3$e$-16$ & $1.0$e$-15$ & $6.5$e$-16$ & $1.1$e$-14$ & $8.6$e$-15$ & $8.6$e$-15$ & $7.0$e$-15$ \\
    & \multirow{2}{*}{$\mathcal{H}_h$} & $4$ & $3.2$e$-09$ & $1.6$e$-10$ & $3.6$e$-09$ & $1.6$e$-10$ & $9.0$e$-10$ & $9.1$e$-10$ & $8.1$e$-10$ & $8.5$e$-10$ \\
    & & $16$ & $2.6$e$-09$ & $6.1$e$-10$ & $2.9$e$-09$ & $6.0$e$-10$ & $7.9$e$-10$ & $7.4$e$-10$ & $7.5$e$-10$ & $7.2$e$-10$ \\
     & \multirow{2}{*}{$\mathcal{W}_h$} & $8$ & $4.8$e$-16$ & $6.2$e$-16$ & $4.7$e$-16$ & $5.0$e$-16$ & $5.1$e$-15$ & $6.1$e$-15$ & $1.1$e$-14$ & $9.4$e$-15$ \\
    & & $32$ & $1.3$e$-15$ & $1.4$e$-15$ & $2.7$e$-15$ & $2.6$e$-15$ & $6.8$e$-14$ & $5.7$e$-14$ & $6.8$e$-14$ & $4.9$e$-14$ \\
    
    \midrule
    \multirow{6}{*}{$10^3$} & \multirow{2}{*}{$\mathcal{Q}_h$} & $4$ & $1.3$e$-15$ & $2.6$e$-16$ & $8.7$e$-16$ & $5.1$e$-16$    & $7.4$e$-16$ & $1.3$e$-15$ & $7.4$e$-16$ & $1.3$e$-15$ \\
    & & $16$ & $7.5$e$-16$ & $6.0$e$-16$ & $6.4$e$-16$ & $6.4$e$-16$ & $2.9$e$-15$ & $1.8$e$-15$ & $5.3$e$-15$ & $4.1$e$-15$ \\
    & \multirow{2}{*}{$\mathcal{H}_h$} & $4$ & $2.8$e$-09$ & $1.6$e$-10$ & $3.2$e$-09$ & $1.6$e$-10$ & $7.2$e$-10$ & $8.9$e$-10$ & $6.5$e$-10$ & $8.4$e$-10$  \\
    & & $16$ & $2.2$e$-09$ & $6.0$e$-10$ & $2.5$e$-09$ & $5.9$e$-10$ & $6.2$e$-10$ & $7.1$e$-10$ & $5.9$e$-10$ & $6.9$e$-10$ \\
     & \multirow{2}{*}{$\mathcal{W}_h$} & $8$ & $4.4$e$-16$ & $5.2$e$-16$ & $7.2$e$-16$ & $6.9$e$-16$ & $8.0$e$-16$ & $1.0$e$-15$ & $1.3$e$-15$ & $1.6$e$-15$ \\
    & & $32$ & $1.4$e$-15$ & $1.3$e$-15$ & $2.4$e$-15$ & $2.7$e$-15$ & $5.2$e$-14$ & $6.8$e$-15$ & $9.2$e$-15$ & $1.1$e$-14$ \\
    \midrule
    
    \multirow{6}{*}{$10^8$} & \multirow{2}{*}{$\mathcal{Q}_h$} & $4$ & $1.3$e$-15$ & $4.9$e$-16$ & $5.4$e$-16$ & $4.0$e$-16$ & $4.4$e$-16$& $5.7$e$-16$ & $8.3$e$-16$ & $1.0$e$-15$ \\
    & & $16$ & $7.3$e$-16$ & $5.4$e$-16$ & $5.9$e$-16$ & $6.3$e$-16$ & $3.7$e$-15$ & $2.3$e$-15$ & $4.4$e$-15$ & $3.2$e$-15$ \\
    & \multirow{2}{*}{$\mathcal{H}_h$} & $4$ & $2.7$e$-09$ & $1.6$e$-10$ & $3.2$e$-09$ & $1.6$e$-10$ & $7.1$e$-10$ & $8.9$e$-10$ & $6.5$e$-10$ & $8.4$e$-10$\\
    & & $16$ & $2.2$e$-09$ & $6.0$e$-10$ & $2.5$e$-09$ & $5.9$e$-10$ & $6.2$e$-10$ & $7.1$e$-10$ & $5.9$e$-10$ & $6.9$e$-10$ \\
     & \multirow{2}{*}{$\mathcal{W}_h$} & $8$ & $5.5$e$-16$ & $6.7$e$-16$ & $5.3$e$-16$ & $7.1$e$-16$  & $9.9$e$-16$ & $9.2$e$-16$ & $1.3$e$-15$ & $1.9$e$-15$ \\
    & & $32$ & $1.2$e$-15$ & $1.2$e$-15$ & $2.6$e$-15$ & $3.2$e$-15$ & $5.1$e$-15$ & $6.5$e$-15$ & $8.7$e$-15$ & $1.2$e$-14$ \\
    \bottomrule
  \end{tabular}
  \caption{Patch test.}
  \label{tab:patch_test}
\end{table}

\subsubsection{Test with a trigonometric displacement solution}\label{sss:divfree}
In this second experiment, we test the convergence properties of the VEM methods presented in sections \ref{sss:first} and \ref{sss:second}. The data are chosen in accordance with the analytical solution
\begin{equation}
    {\bf u}^i = \begin{pmatrix}
        \pi x \cos (\pi y) \\ -\sin (\pi y)
    \end{pmatrix}, \quad
    p^i = \lambda^i \text{div }{\bf u}^i=0, \quad 
    \text{for } i=1,2.
\end{equation}
We study the behavior of $\delta({\bf u})$ and $\delta({\bf p})$ versus the maximum element diameter $h_{max}$, comparing the convergence rate of the methods for three different values of $\lambda$, as in section \ref{sss:patch}: $\lambda= 1$ (compressible case), $\lambda=10^3$ and $\lambda=10^8$ (nearly incompressible cases). We report the convergence diagrams for both the \emph{Initial Matching} and \emph{Small Edges} configurations.

In Figures \ref{fig:square_div_free_convergence-deg1}, \ref{fig:hexagon_div_free_convergence-deg1} and \ref{fig:web_div_free_convergence-deg1}, we display the results of the first order scheme presented in section \ref{sss:first} for the sequence of squared meshes $\mathcal{Q}_h$, hexagonal meshes $\mathcal{H}_h$ and web meshes $\mathcal{W}_h$, respectively. 
We notice that the theoretical prediction of Remark \ref{rm:firstorder} is confirmed, with $\nu=2$. As expected, the logarithmic degeneracy is not practically observed in the numerical experiments. 

Analogously, in Figures \ref{fig:square_div_free_convergence-deg2}, \ref{fig:hexagon_div_free_convergence-deg2} and \ref{fig:web_div_free_convergence-deg2}, we show the results of the second order scheme presented in section \ref{sss:second} for the same meshes.
We observe a super-convergence behavior of the method, as it exhibits a quadratic convergence rate, while the theoretical prediction of Theorem \ref{th:converg} gives a convergence rate of order $5/2$.
This is most likely due to the particularly simple structure of the problem under consideration. However, in the general case we do not expect a better convergence rate than $5/2$, even when the analytical solution is smooth. As before, no significant influence of the logarithmic term in estimate \eqref{eq:err_est_quad} appears in actual computations.

Finally, both the schemes experience no deterioration of the solution for large values of $\lambda$, confirming the robustness of the method with respect to the incompressibility parameter (no volumetric locking effect occurs).
\begin{figure}[htbp]
    \centering
    \begin{subfigure}[t]{0.48\textwidth}
        \centering
        \includegraphics[width=\linewidth]{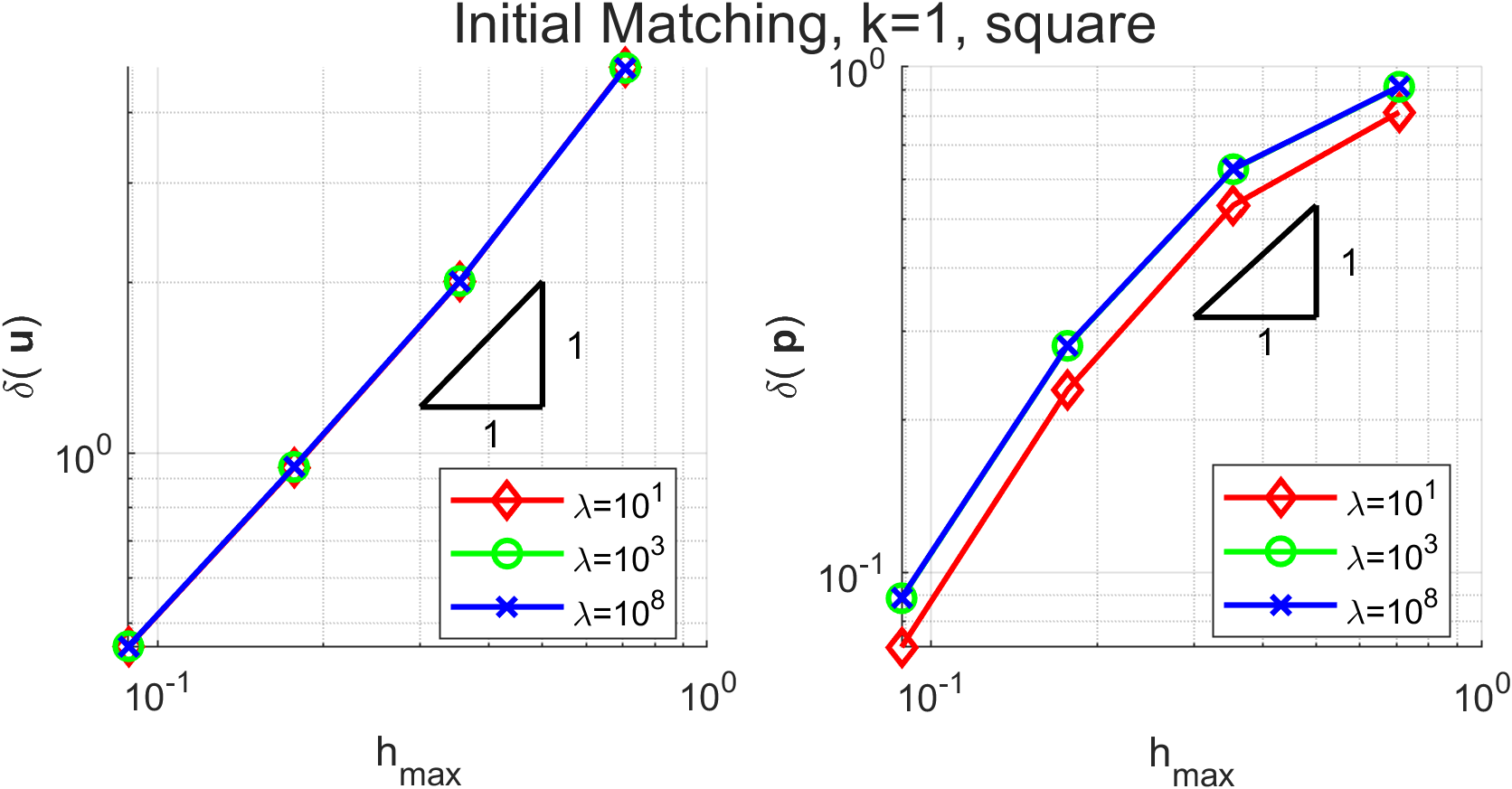}
    \end{subfigure}
    \hfill
    \begin{subfigure}[t]{0.48\textwidth}
        \centering
        \includegraphics[width=\linewidth]{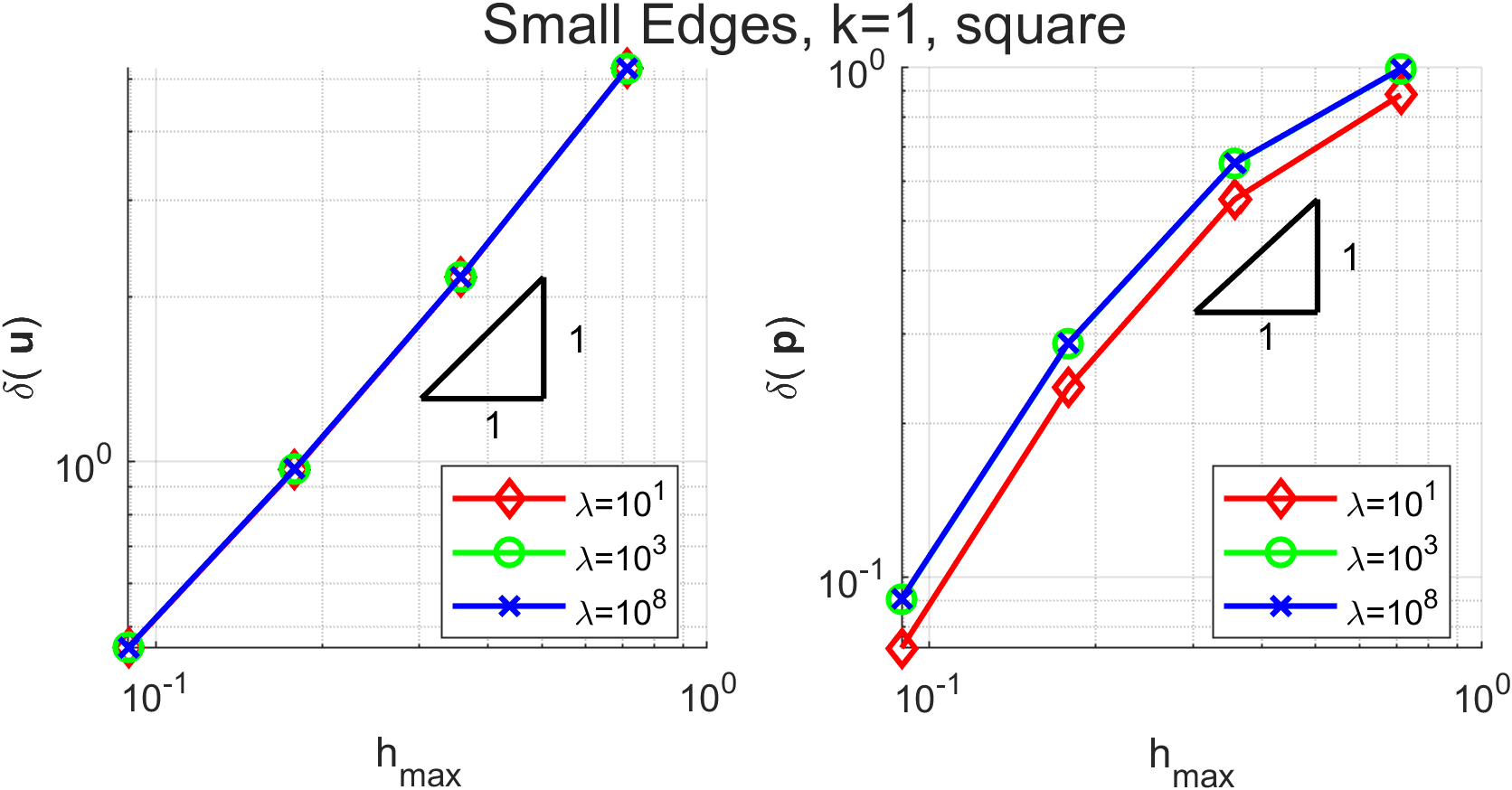}
    \end{subfigure}
    \caption{Test with trigonometric displacement solution (linear element): Behavior of $\delta({\bf u})$ and $\delta({\bf p})$ for the sequence of meshes $\mathcal{Q}_h$.}
  \label{fig:square_div_free_convergence-deg1}
\end{figure}

\begin{figure}[htbp]
    \centering
    \begin{subfigure}[t]{0.48\textwidth}
        \centering
        \includegraphics[width=\linewidth]{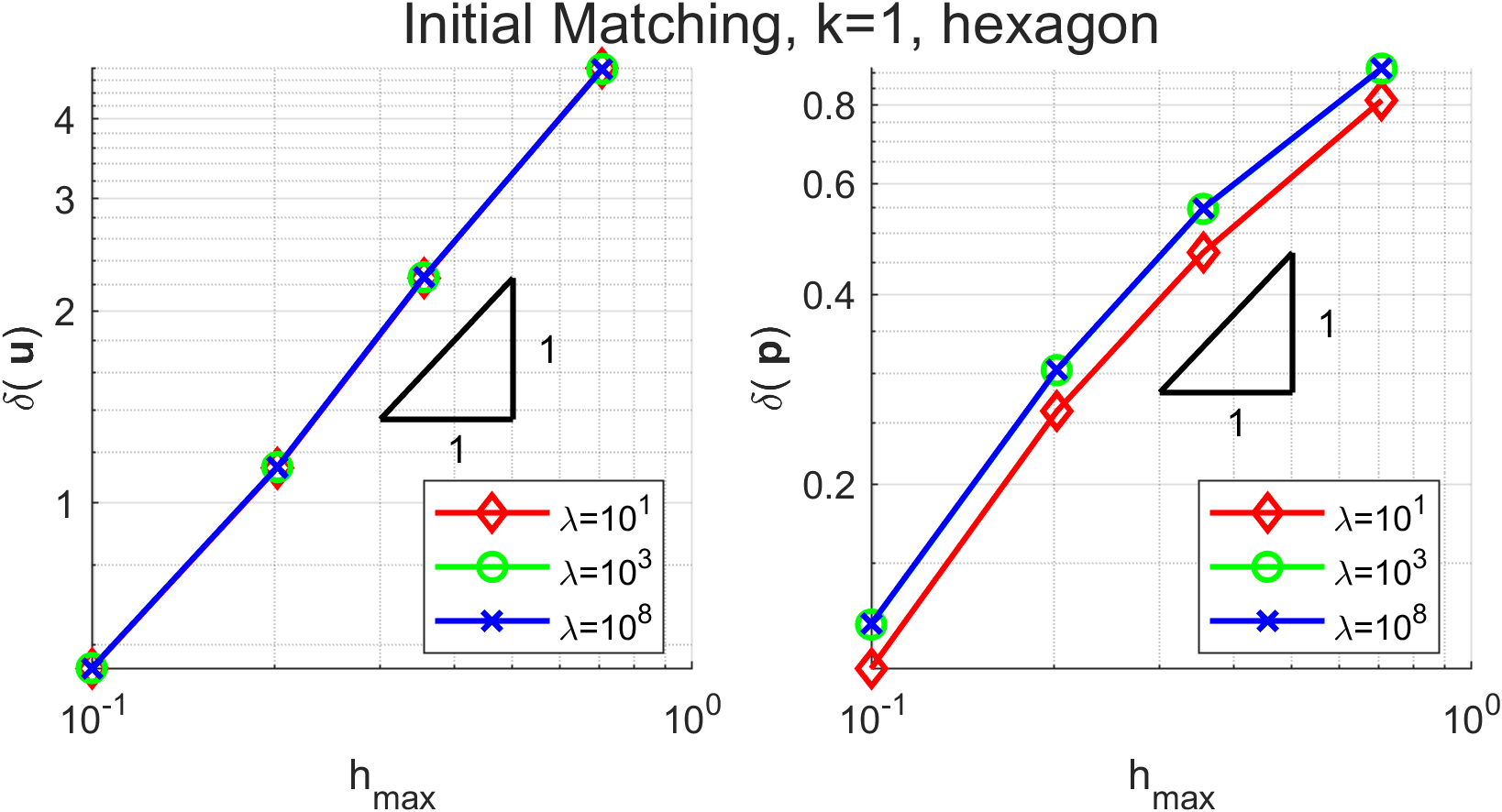}
    \end{subfigure}
    \hfill
    \begin{subfigure}[t]{0.48\textwidth}
        \centering
        \includegraphics[width=\linewidth]{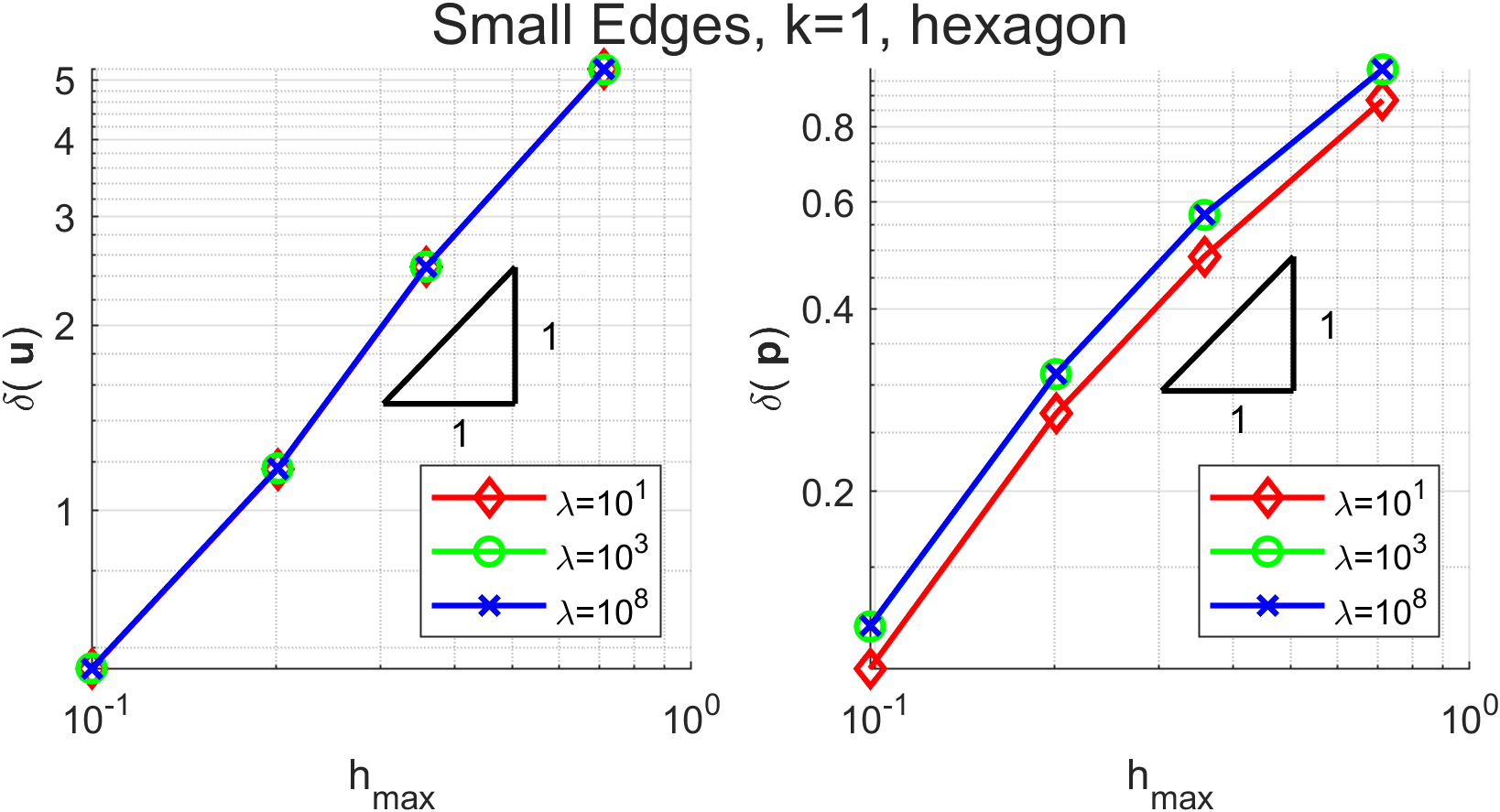}
    \end{subfigure}
    \caption{Test with trigonometric displacement solution (linear element): Behavior of $\delta({\bf u})$ and $\delta({\bf p})$ for the sequence of meshes $\mathcal{H}_h$.}
  \label{fig:hexagon_div_free_convergence-deg1}
\end{figure}

\begin{figure}[htbp]
    \centering
    \begin{subfigure}[t]{0.48\textwidth}
        \centering
        \includegraphics[width=\linewidth]{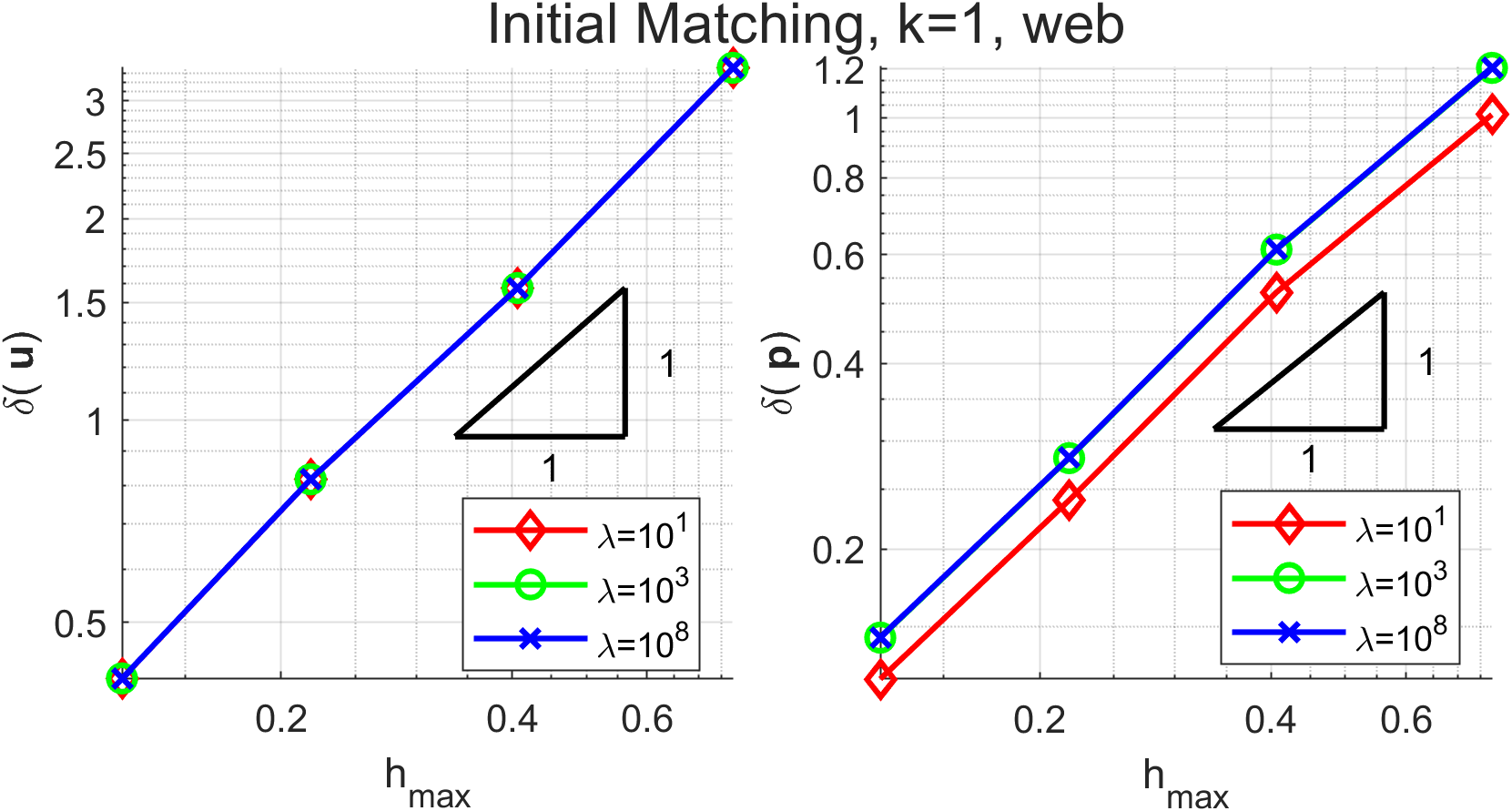}
    \end{subfigure}
    \hfill
    \begin{subfigure}[t]{0.48\textwidth}
        \centering
        \includegraphics[width=\linewidth]{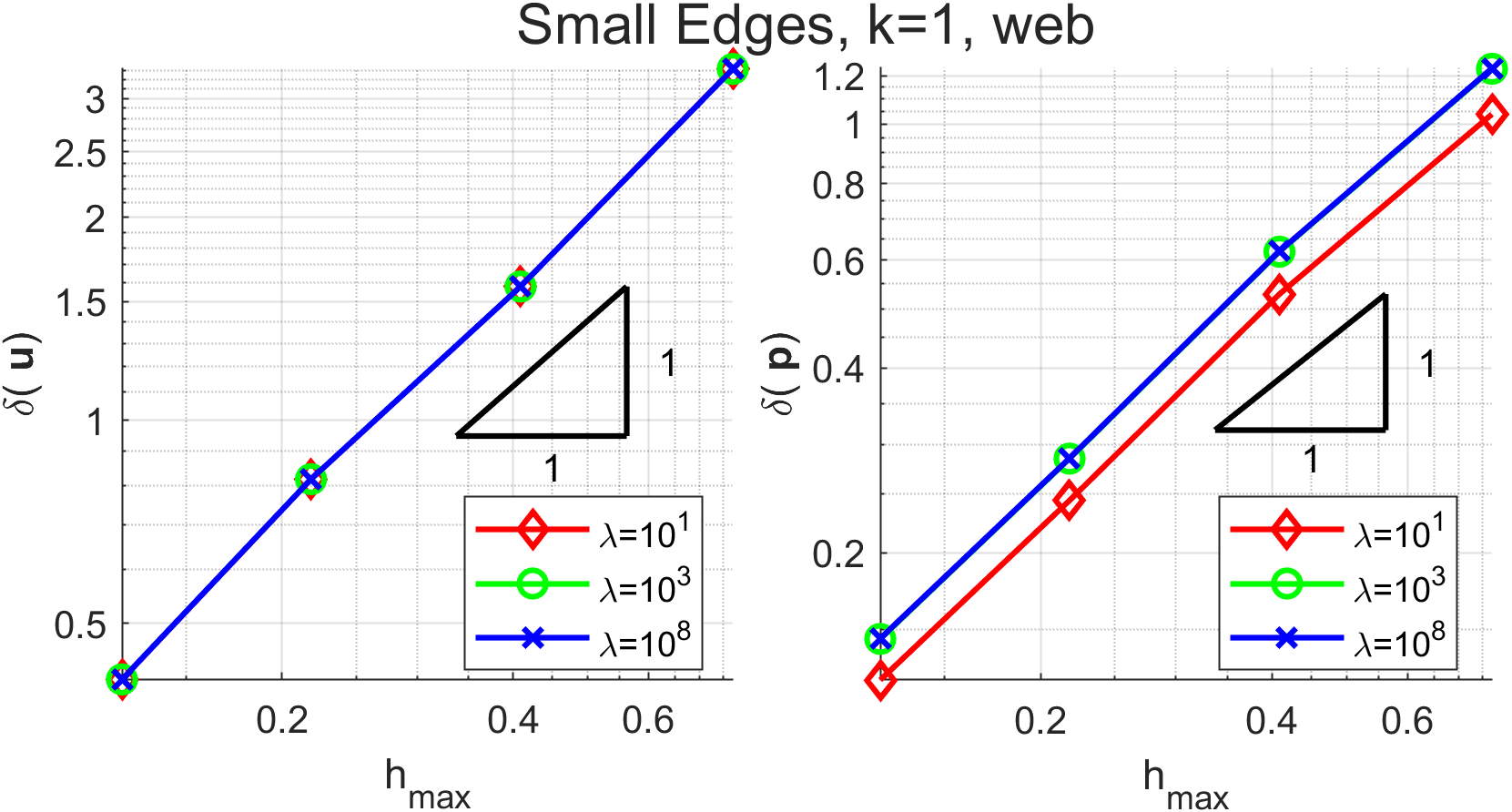}
    \end{subfigure}
     \caption{Test with trigonometric displacement solution (linear element): Behavior of $\delta({\bf u})$ and $\delta({\bf p})$ for the sequence of meshes $\mathcal{W}_h$.}
    \label{fig:web_div_free_convergence-deg1}
\end{figure}

\begin{figure}[htbp]
    \centering
    \begin{subfigure}[t]{0.48\textwidth}
        \centering
        \includegraphics[width=\linewidth]{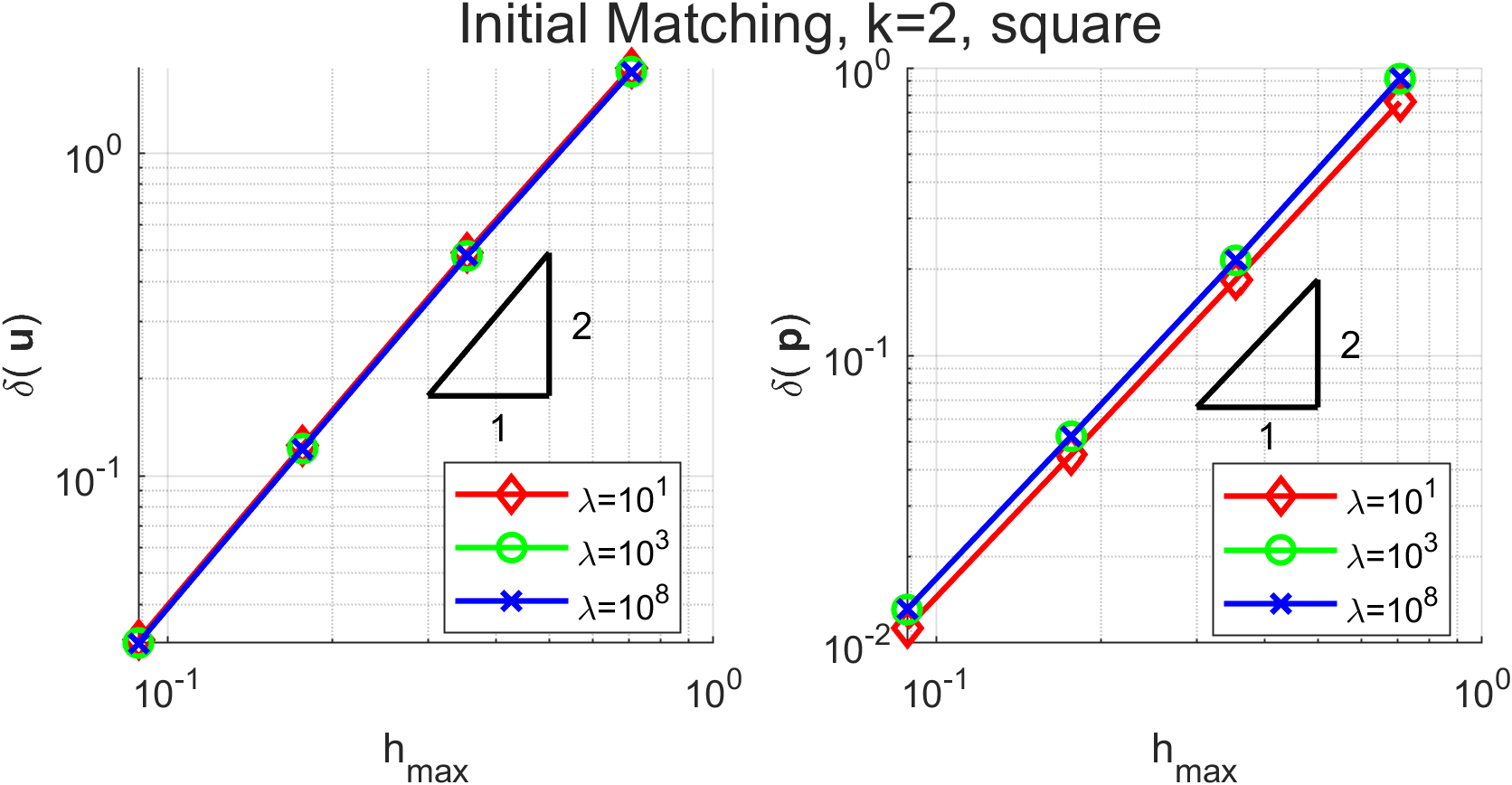}
    \end{subfigure}
    \hfill
    \begin{subfigure}[t]{0.48\textwidth}
        \centering
        \includegraphics[width=\linewidth]{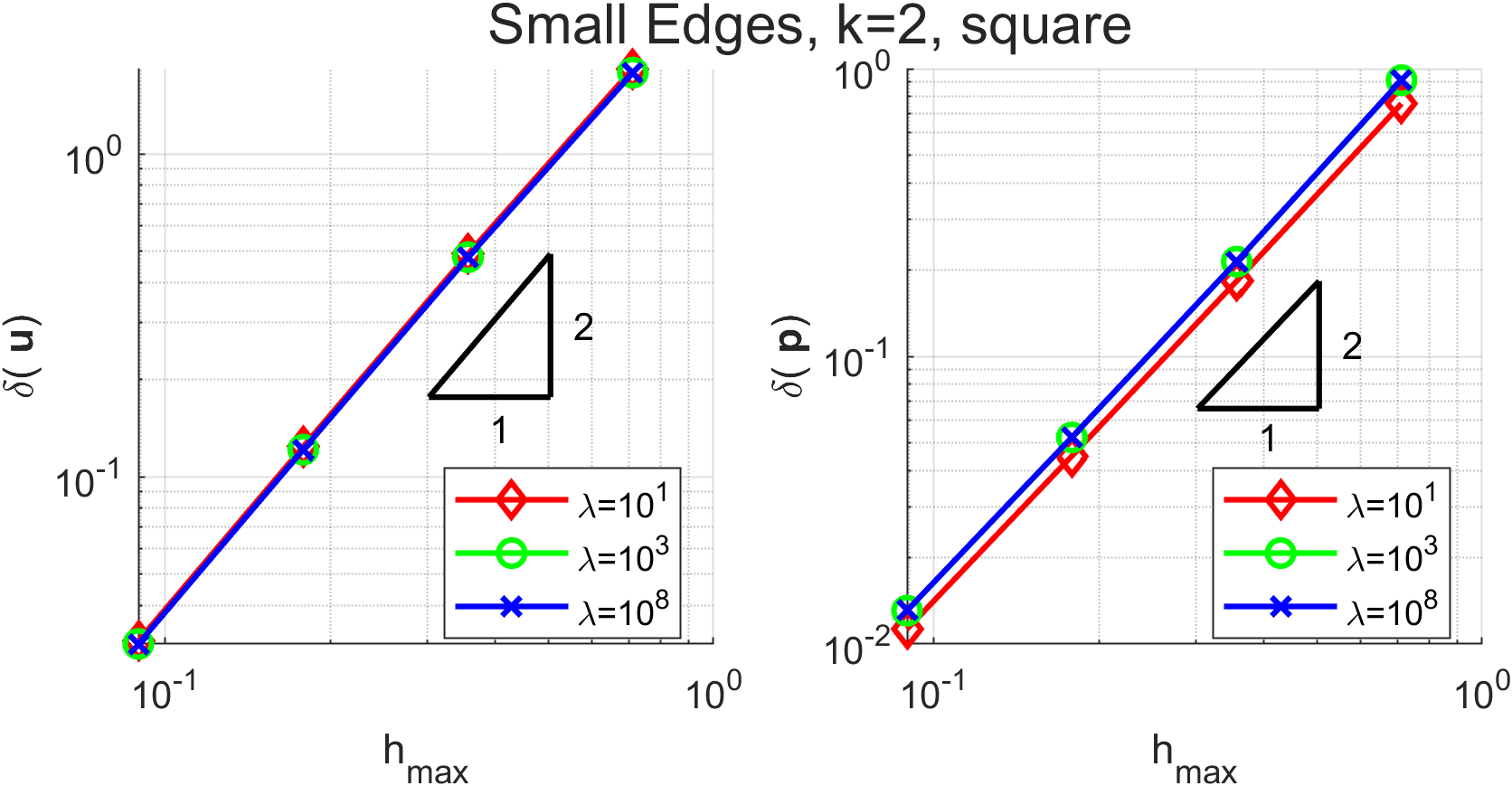}
    \end{subfigure}
    \caption{Test with trigonometric displacement solution (quadratic element): Behavior of $\delta({\bf u})$ and $\delta({\bf p})$ for the sequence of meshes $\mathcal{Q}_h$.}
    \label{fig:square_div_free_convergence-deg2}
\end{figure}

\begin{figure}[htbp]
    \centering
    \begin{subfigure}[t]{0.48\textwidth}
        \centering
        \includegraphics[width=\linewidth]{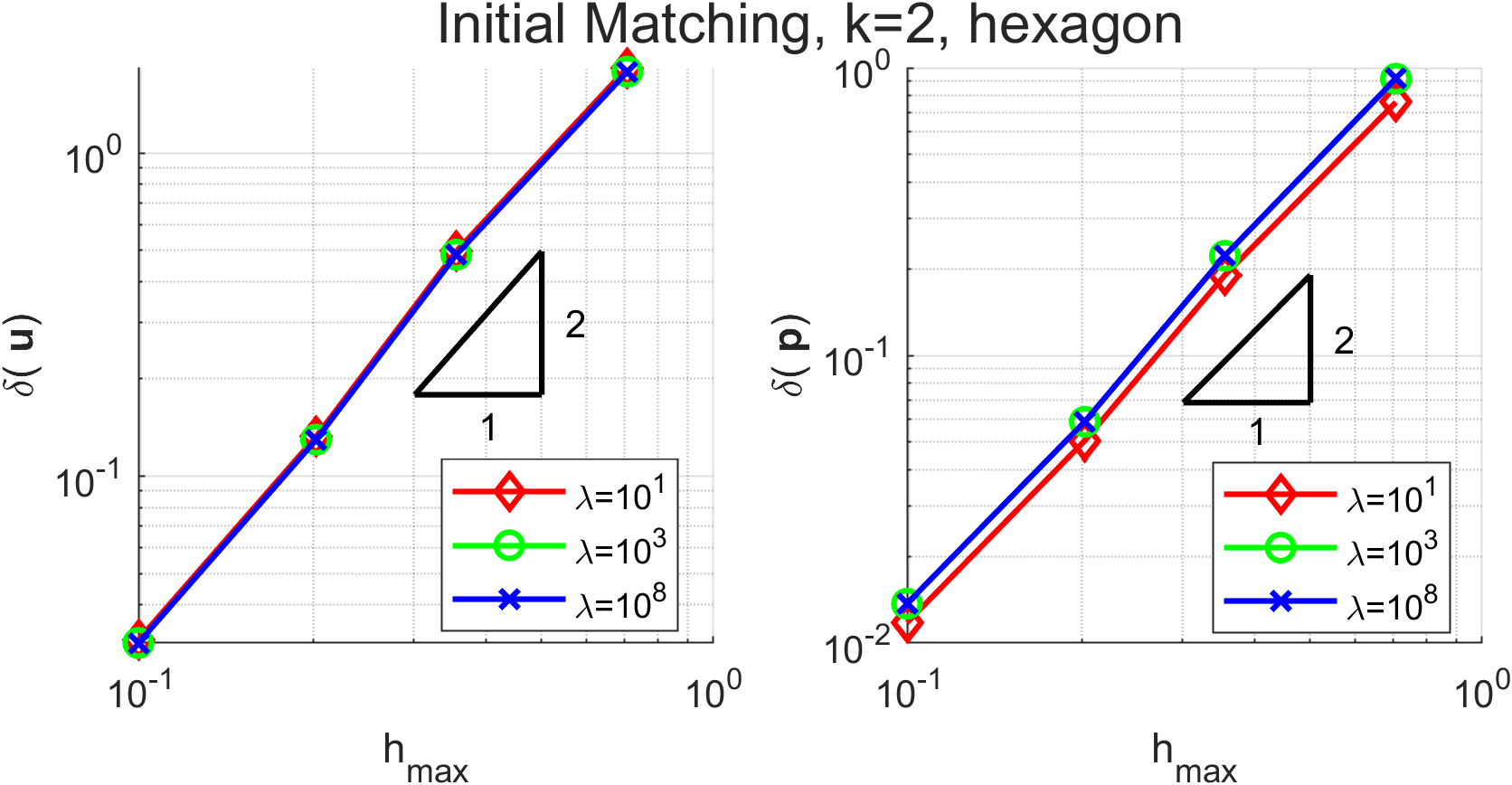}
    \end{subfigure}
    \hfill
    \begin{subfigure}[t]{0.48\textwidth}
        \centering
        \includegraphics[width=\linewidth]{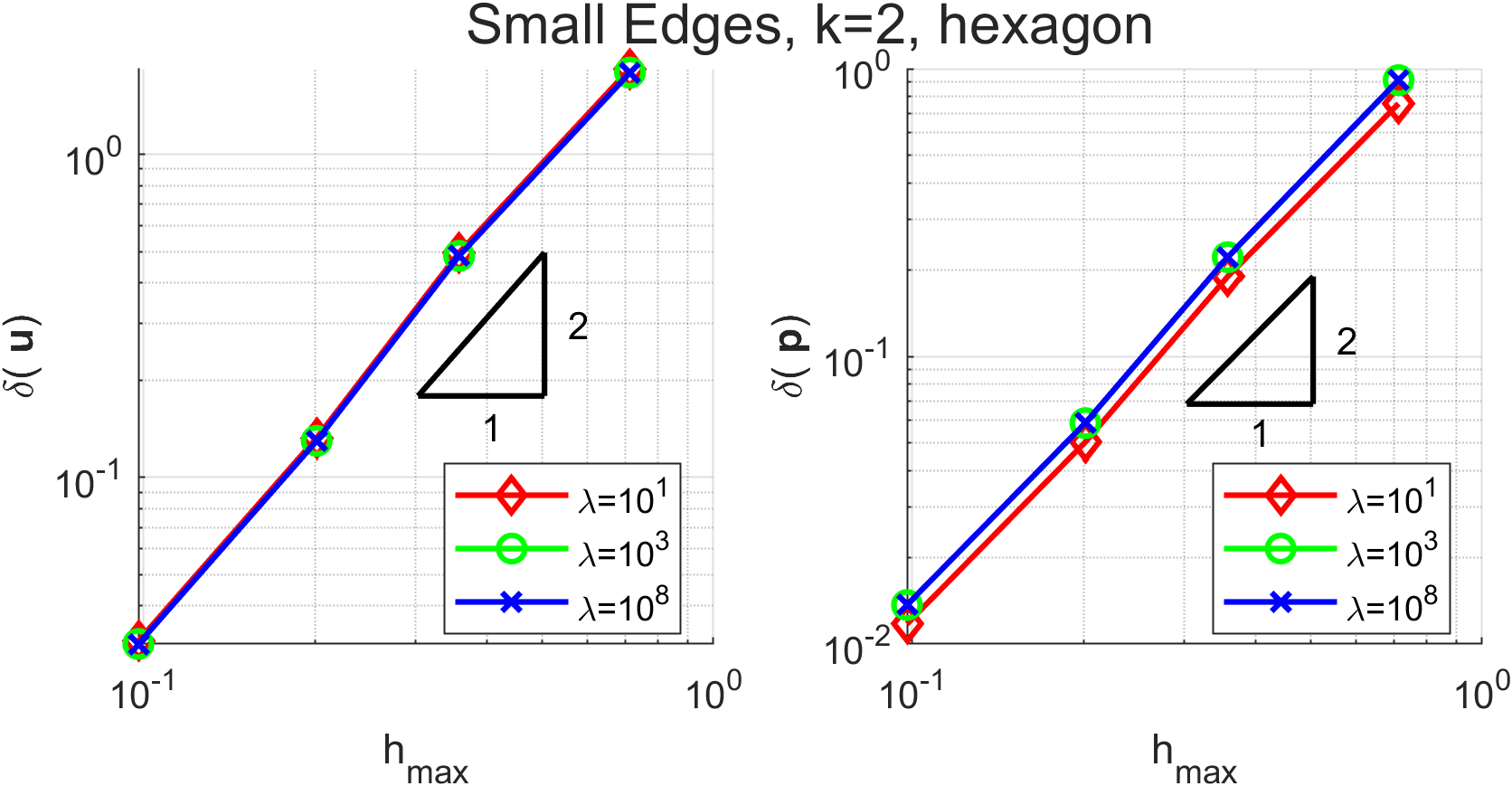}
    \end{subfigure}
    \caption{Test with trigonometric displacement solution (quadratic element): Behavior of $\delta({\bf u})$ and $\delta({\bf p})$ for the sequence of meshes $\mathcal{H}_h$.}
  \label{fig:hexagon_div_free_convergence-deg2}
\end{figure}

\begin{figure}[htbp]
    \centering
    \begin{subfigure}[t]{0.48\textwidth}
        \centering
        \includegraphics[width=\linewidth]{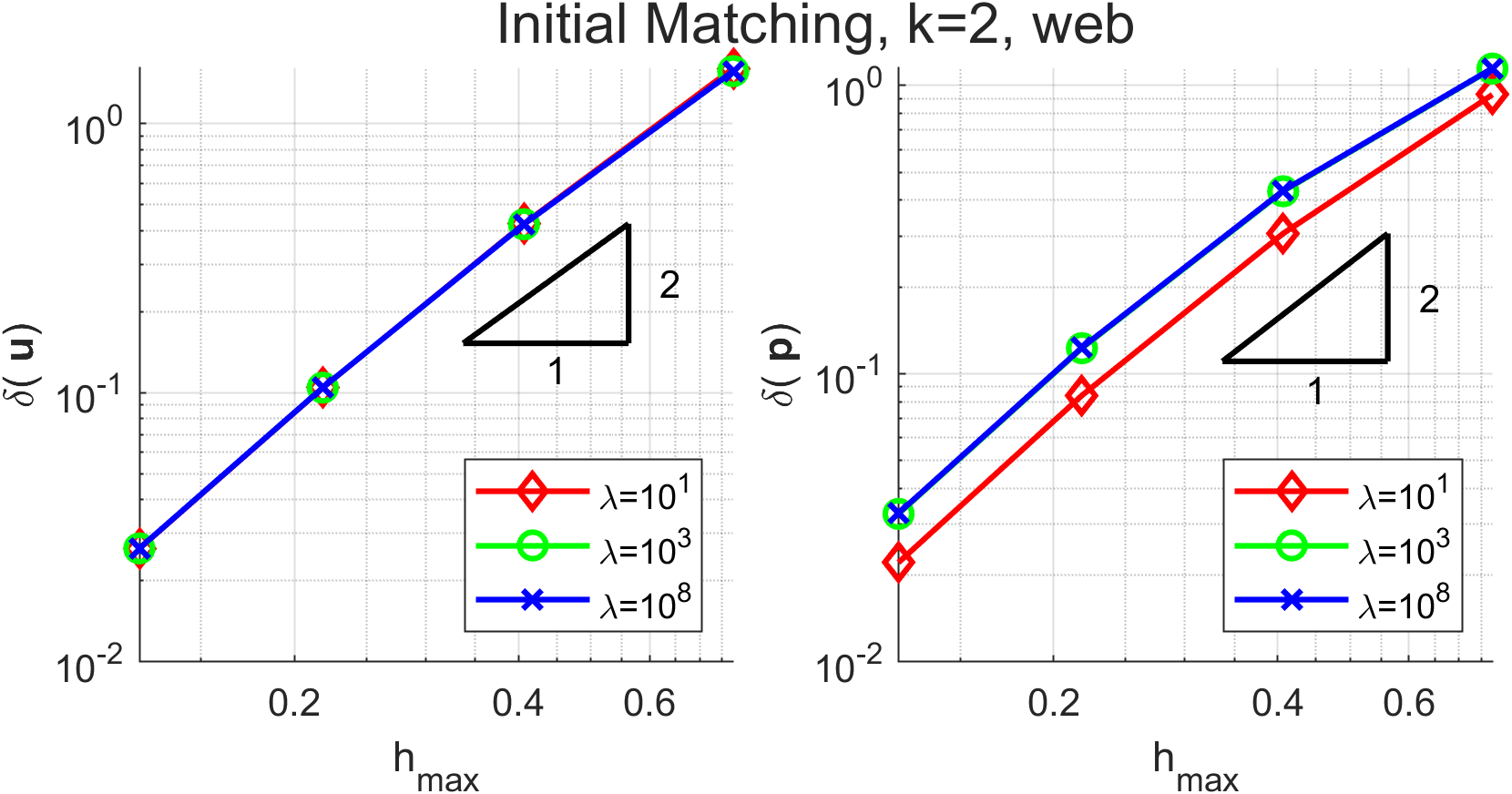}
    \end{subfigure}
    \hfill
    \begin{subfigure}[t]{0.48\textwidth}
        \centering
        \includegraphics[width=\linewidth]{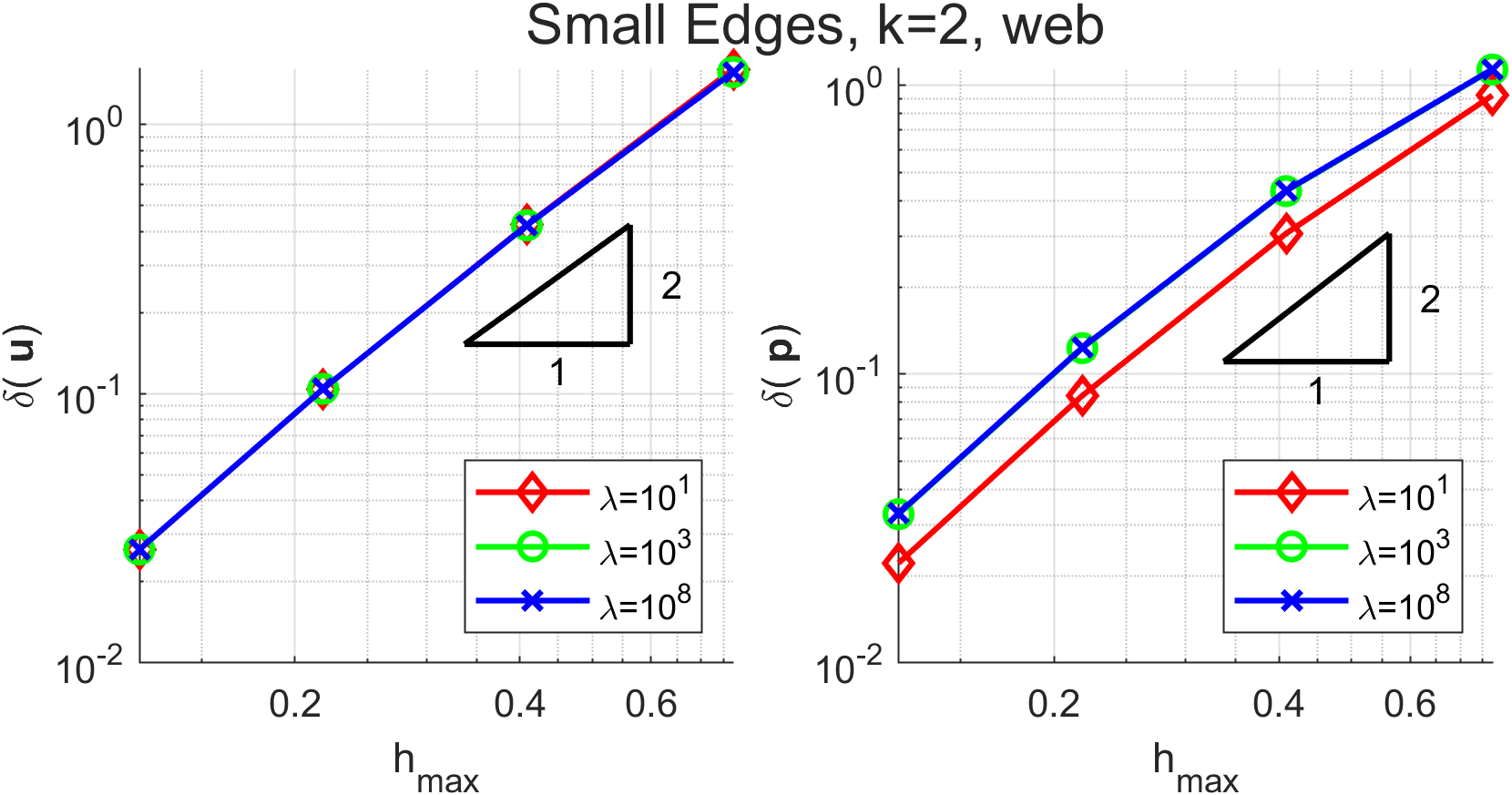}
    \end{subfigure}
     \caption{Test with trigonometric displacement solution (quadratic element): Behavior of $\delta({\bf u})$ and $\delta({\bf p})$ for the sequence of meshes $\mathcal{W}_h$.}
    \label{fig:web_div_free_convergence-deg2}
\end{figure}

\subsection{Hertz problem}\label{ss:hertz}
The last example is the Hertzian contact problem. In the reference configuration, we consider an elastic half-disk ($R=0.5$) initially in contact with a square block ($L=1$) at the origin of the reference system, as shown in Figure \ref{fig:initial_deformed_configurations}.
\begin{figure}[htbp]
  \centering
  \includegraphics[width=1.0\linewidth]{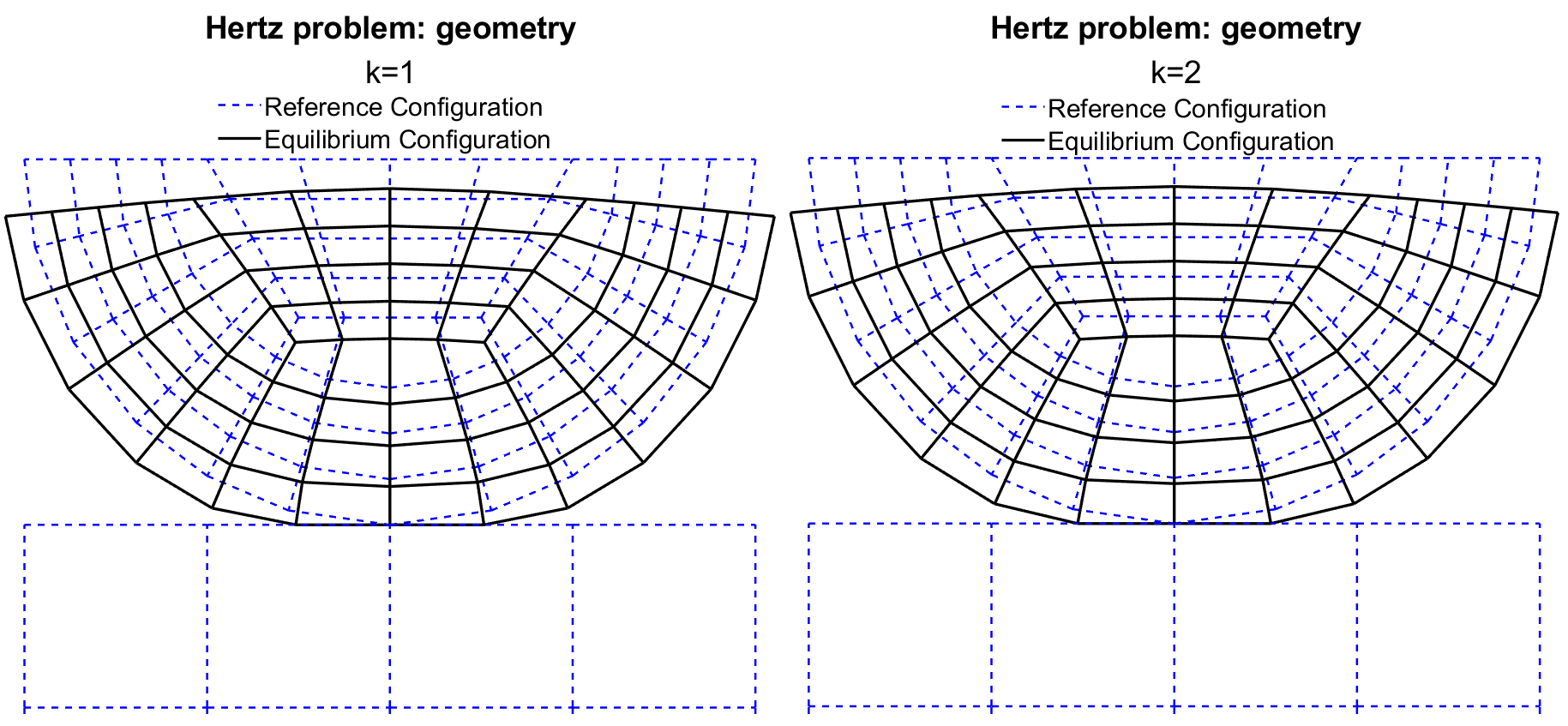}
  \caption{Reference and equilibrium configurations.}
  \label{fig:initial_deformed_configurations}
\end{figure}
The disk is subjected to an external load $F=2.5$, applied as uniform pressure $p$ on its upper surface. Young's modulus and Poisson's ratio are $E_{square} = 70000$, $\nu_{square} = 0.4999$, $E_{disk}=70$ and $\nu_{disk} = 0.3$. Notice that these values of the elastic parameters essentially correspond to considering the square as a rigid body. Unlike the preceding tests, a non-trivial initial gap function $g_0$ is set.
For the Hertzian problem the following expression for the contact pressure $p_n$ is analytically available in terms of the half-contact area $b$ and the effective material parameter $E^\star$. In fact, we have (e.g., see \cite{Johnson}):
\begin{equation*} 
    p_n = \frac{4Rp}{\pi b^2} \sqrt{b^2-x^2},
\end{equation*}
where
\begin{equation*}
    b = 2\sqrt{\frac{2R^2p}{\pi E^{\star}}}, \qquad \frac{1}{E^\star} = \frac{1-\nu_{square}^2}{E_{square}} + \frac{1-\nu_{disk}^2}{E_{disk}}.
\end{equation*}
Numerical tests are performed using quadrilateral meshes, an example of which is shown in Figure \ref{fig:initial_deformed_configurations}. The square body is discretized with a fixed mesh of $16$ elements, while the half-disk is decomposed using progressively finer meshes, up to $256$ elements.

The discrete normal stress over the contact region is plotted in Figure \ref{fig:pressure_distribution}, for both linear ($k=1$) and quadratic ($k=2$) VEM schemes. By evaluating the virtual element solutions at the contact element vertices, we observe that they are able to reproduce the behavior of the quantities analytically available, achieving increasingly higher accuracy as the mesh is refined.

\begin{figure}[htbp]
  \centering
  \includegraphics[width=1.0\linewidth]{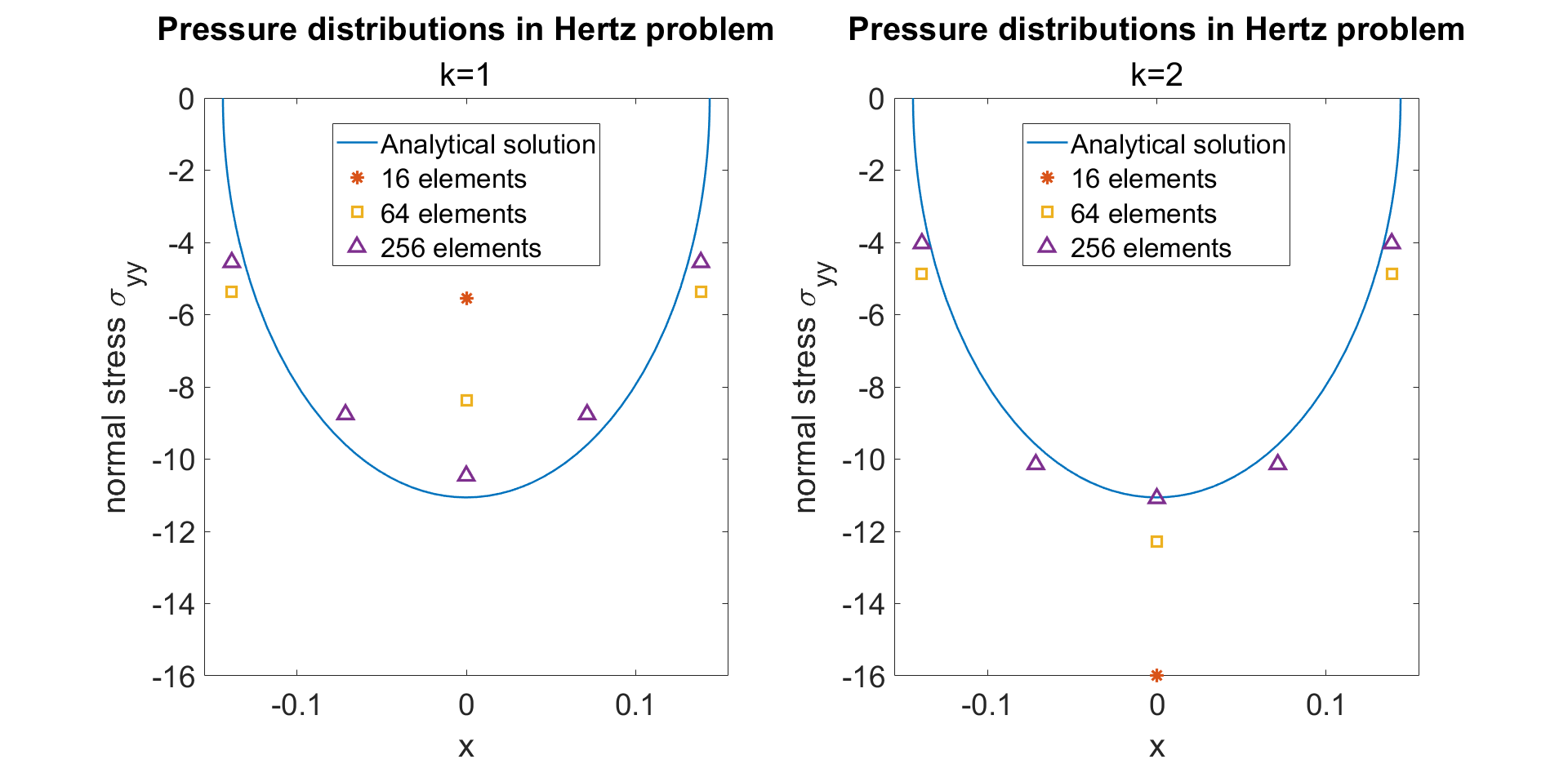}
  \caption{Contact pressure distribution in the Hertz problem: linear (on the left) and quadratic (on the right) VEM approximation.}
  \label{fig:pressure_distribution}
\end{figure}

Finally, in Figure \ref{fig:max_contact_pressure} we display the maximum contact pressure compared to the analytical solution (for the selected parameters: $p_{n,max} = -11.06$); also here, both linear and quadratic schemes are considered.

\begin{figure}[htbp]
  \centering
  \includegraphics[width=1.0\linewidth]{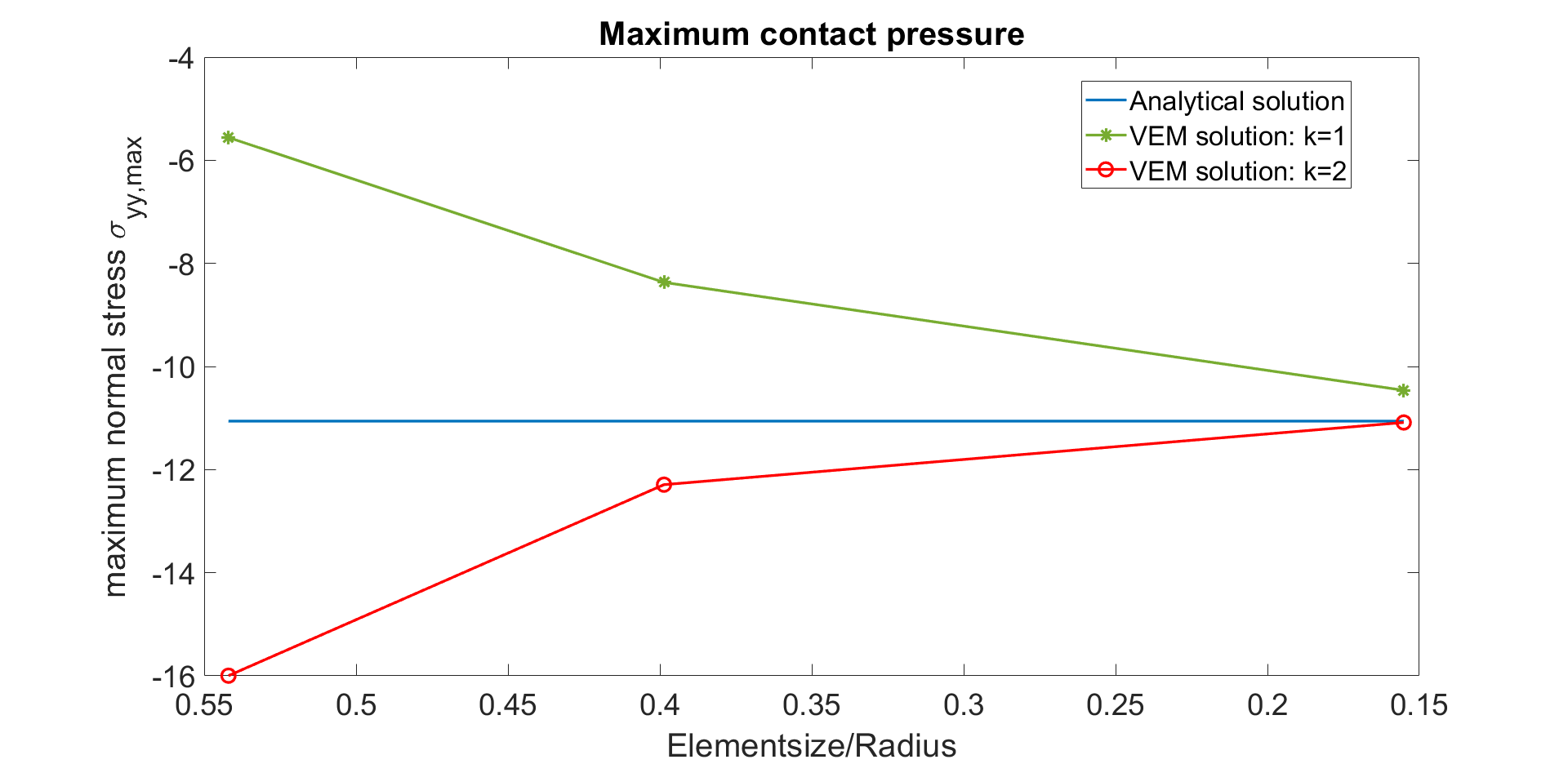}
  \caption{Maximum contact pressure in the Hertz problem: linear and quadratic VEM approximation.}
  \label{fig:max_contact_pressure}
\end{figure}
\section{Conclusions}\label{sec:conclusions}
In this work, we have developed and analysed mixed Virtual Element Methods (VEMs) for the numerical approximation of two-dimensional frictionless contact problems in elasticity, with a particular focus on nearly incompressible materials. By adopting a mixed displacement/pressure formulation, the proposed schemes effectively overcome the volumetric locking phenomenon that typically affects standard displacement-based methods.
As examples of application of the theory, we have considered a couple of schemes; for one of the two we have provided a detailed error analysis, but very similar tools can be used to study the other method. Special attention has been given to the influence of \say{small edges} in the mesh.

Numerical experiments confirm the theoretical predictions, demonstrating:
\begin{itemize}
    
    \item the robustness of the methods with respect to the volumetric parameter, with no evidence of locking in the nearly incompressible regime, even for extremely large value of the first Lam\'e parameter;
    
    \item the good performance and stability in the presence of \say{small edges}, with no significant deterioration observed in practical computations;
    
    \item the achievement of the expected convergence rates, with the second-order scheme exhibiting super-convergent behaviour in some test cases.
    
\end{itemize}

Overall, the results indicate that the proposed mixed VEM schemes provide a flexible and reliable tool for the simulation of contact problems in elasticity, including the case of nearly incompressible materials. Future work may address the extension to three-dimensional problems, the inclusion of frictional contact, and the design of different discretization spaces and VEM stabilization terms.

\section*{Acknowledgements}
This research was funded in part by INdAM-GNCS.
\addcontentsline{toc}{section}{\refname}
\bibliographystyle{plain}
\bibliography{biblio}
\end{document}